\documentclass[11pt,reqno,a4paper]{amsart}
\usepackage[T1]{fontenc}
\usepackage{lmodern}
\usepackage{microtype}
\usepackage{amsmath,amssymb,amsthm,mathtools}
\usepackage{esint}
\usepackage[margin=27mm]{geometry}
\usepackage{enumitem}
\usepackage{mathrsfs,needspace}
\usepackage[hidelinks]{hyperref}
\usepackage[nameinlink,noabbrev]{cleveref}
\hypersetup{
  pdftitle={Finite-energy travelling waves for the easy-plane Landau--Lifshitz equation},
  pdfauthor={Rowan Juneman, Ze Li, Juncheng Wei},hypertexnames=false}

\allowdisplaybreaks
\numberwithin{equation}{section}
\setlist[enumerate]{leftmargin=2.1em,itemsep=2pt,topsep=4pt}
\setlist[itemize]{leftmargin=2em,itemsep=2pt,topsep=4pt}

\newtheorem{theorem}{Theorem}
\newtheorem{proposition}{Proposition}[section]
\newtheorem{lemma}[proposition]{Lemma}
\newtheorem{corollary}[proposition]{Corollary}
\theoremstyle{definition}

\newcommand{\R}{\mathbb R}
\newcommand{\Sph}{\mathbb S}

\newcommand{\dd}{\,\mathrm d}

\newcommand{\Rthree}{\mathbb R^3}
\newcommand{\Stwo}{\mathbb S^2}

\newcommand{\Jac}{D}

\crefname{theorem}{Theorem}{Theorems}
\crefname{proposition}{Proposition}{Propositions}
\crefname{lemma}{Lemma}{Lemmas}
\crefname{corollary}{Corollary}{Corollaries}
\crefname{remark}{Remark}{Remarks}
\crefname{section}{Section}{Sections}
\crefname{subsection}{Section}{Sections}

\title[Travelling waves for the Landau--Lifshitz equation]
{Finite-energy travelling waves for the easy-plane
Landau--Lifshitz equation}

\author[R.~Juneman]{Rowan Juneman}
\address{R.~Juneman: Department of Mathematics, Chinese University of Hong Kong,
Shatin, NT, Hong Kong.}
\email{rowanjuneman@cuhk.edu.hk}

\author[Z.~Li]{Ze Li}
\address{Z.~Li: School of Mathematics and Statistics, Ningbo University,
Ningbo, 315211, P.R. China.}
\email{rikudosennin@163.com}

\author[J.~Wei]{Juncheng Wei}
\address{J.~Wei: Department of Mathematics, Chinese University of Hong Kong,
Shatin, NT, Hong Kong.}
\email{wei@math.cuhk.edu.hk}
\date{}

\begin{document}
\raggedbottom

\begin{abstract}
For every integer $d\ge2$ and every subsonic speed $c\in(0,1)$, we
construct a nonconstant finite-energy travelling wave for the easy-plane
Landau--Lifshitz equation on $\R^d$. This addresses the existence problem
raised in the numerical work of Papanicolaou and Spathis \cite{papanicolaou1999}. The profiles
satisfy the travelling-wave equation in the sense of distributions and
are smooth in dimension two.
\end{abstract}

\maketitle

\section{Introduction}\label{sec:introduction}

\subsection{The equation and the travelling-wave problem}

The evolution of the magnetisation in a ferromagnetic material is
described by the Landau--Lifshitz equation~\cite{landaulifshitz1935}.
Without damping and with uniaxial anisotropy, the equation reads
\begin{equation}\label{eq:intro-evolution}
\partial_t m+m\times(\Delta m+\lambda m_3e_3)=0,
\qquad m:\R^d\times\R\longrightarrow\Sph^2,
\end{equation}
where $e_3=(0,0,1)$ and $\lambda\in\R$. The sign of $\lambda$
distinguishes the isotropic $(\lambda=0)$, easy-axis $(\lambda>0)$
and easy-plane $(\lambda<0)$ regimes. We work in the easy-plane
regime, for which a rescaling reduces the anisotropy coefficient
to $\lambda=-1$. The corresponding energy,
\begin{equation}\label{eq:intro-energy}
E(m)=\frac12\int_{\R^d}\bigl(|\nabla m|^2+m_3^2\bigr)\dd x,
\end{equation}
penalises spatial variation through the gradient term and favours
magnetisation in the equatorial plane through the anisotropy term.
For sufficiently regular solutions, conservation of this energy
follows by integration by parts.

\medskip

Among the solutions of \eqref{eq:intro-evolution}, travelling waves
describe coherent magnetisation patterns that retain their shape
as they propagate. Such waves are expected to play an important
role in the long-time dynamics and take the form
\[
m(x,t)=u(x_1-ct,x_2,\ldots,x_d),
\]
where $c$ denotes the speed of propagation. Substitution
into \eqref{eq:intro-evolution} with $\lambda=-1$, together with
the identity $u\cdot\Delta u=-|\nabla u|^2$, gives the travelling-wave
equation
\begin{equation}\label{eq:main}
\Delta u+\bigl(|\nabla u|^2+u_3^2\bigr)u-u_3e_3
+c\,u\times\partial_1u=0.
\end{equation}

\medskip

The main result of this paper establishes the existence of a nontrivial
weak solution to \eqref{eq:main} in every dimension $d\ge2$, for every
subsonic speed $c\in(0,1)$.

\begin{theorem}[All-dimensional existence]\label{thm:all-dimensional}
For every integer $d\ge2$ and every $c\in(0,1)$ there is a nonconstant finite-energy map $u\in H^1_{\mathrm{loc}}(\R^d;\Sph^2)$ such that \eqref{eq:main} holds in
$\mathcal D'(\R^d;\R^3)$. Equivalently, for every
$\varphi\in C_c^\infty(\R^d;\R^3)$,
\begin{equation}\label{eq:weak}
-\int\sum_{j=1}^d\partial_j u\cdot\partial_j\varphi
+\int (|\nabla u|^2+u_3^2)u\cdot\varphi
-\int u_3e_3\cdot\varphi
+c\int(u\times\partial_1u)\cdot\varphi=0.
\end{equation}
\end{theorem}

The subsonic range is singled out by the behaviour of small
disturbances of a uniformly magnetised state. Linearisation
of \eqref{eq:intro-evolution} about $e_1=(1,0,0)$ yields the
dispersion relation
\[
\omega(\xi)^2=|\xi|^2(1+|\xi|^2),
\]
so linear waves have phase speed $\sqrt{1+|\xi|^2}$. These speeds
are greater than one and approach one in the long-wavelength limit.
Thus a profile with $0<c<1$ travels more slowly than every linear
wave, whereas for $c>1$ some linear waves are stationary in the
moving frame and can resonate with the profile. This change
identifies $c=1$ as a natural threshold and explains the terminology
\emph{subsonic} for the interval $0<c<1$.

\medskip

The existence of nonconstant finite-energy travelling waves at every
speed in this interval was conjectured in dimensions two and three
following the numerical and asymptotic work of Papanicolaou and
Spathis~\cite{papanicolaou1999}. \Cref{thm:all-dimensional} resolves
this existence problem in the plane and extends weak existence
throughout the subsonic range to every higher dimension. The planar
conclusion gives smooth waves: by de Laire~\cite[Proposition~1.4]{delaire2014},
every finite-energy distributional solution in dimension two is smooth.

\medskip

For $d\ge3$, the theorem provides finite-energy weak solutions.
The construction of a smooth vortex-ring branch covering the full
subsonic range, as suggested by the three-dimensional numerical
calculations in~\cite{papanicolaou1999}, remains a separate problem.
A further feature of our construction is that the energy can be
made arbitrarily small. More precisely, for every $d\ge3$, every
fixed $c\in(0,1)$ and every $\varepsilon>0$, there is a nonconstant
weak travelling wave satisfying
\[
0<E(u)<\varepsilon.
\]

\medskip

This additional conclusion has a direct consequence for regularity
in dimensions $d=3,4$. De Laire~\cite[Theorem~1.1]{delaire2014} proved
that uniformly continuous nonconstant travelling waves have energy
bounded below by a positive constant. Our construction produces weak
profiles below this threshold, which therefore cannot be uniformly
continuous. The small-energy solutions thus demonstrate that the
uniform-continuity hypothesis in these energy-gap results is essential.

\subsection{Previous results}
Lin and Wei~\cite{lin2010} proved the existence of smooth finite-energy
solutions of the travelling-wave equation~\eqref{eq:main} at sufficiently
small positive speeds in every dimension $d\ge2$. Their perturbative
construction produces a widely separated vortex--antivortex pair in the
plane, consistent with the numerical work of Papanicolaou and
Spathis~\cite{papanicolaou1999}, and axially symmetric vortex-ring
solutions in higher dimensions. Their analysis also provides detailed
information about the vortex geometry and the small-speed asymptotics.
Related developments include the static variational theory for planar
ferromagnets~\cite{hanglin2001}, studies of magnetic-vortex
dynamics~\cite{komineaspapanicolaou2007,linshatah2003,papanicolaoutomaras1991},
and the construction of travelling vortex helices~\cite{weiyang2016}.

\medskip

In one dimension, the finite-energy travelling-wave profiles for
\eqref{eq:main} have been explicitly
classified~\cite[Proposition~1.3]{delaire2014}. For $0\le c<1$,
set $a=(1-c^2)^{1/2}$. Up to the symmetries of the equation, every
nonconstant profile has the form
\begin{equation}\label{eq:intro-one-dimensional}
u_c(x_1)=\bigl(c\,\operatorname{sech}(a x_1),\,
             \tanh(a x_1),\,a\,\operatorname{sech}(a x_1)\bigr),
\qquad E(u_c)=2a,
\end{equation}
and there are no nonconstant finite-energy profiles for $c\ge1$.
De Laire and Gravejat~\cite{delairegravejat2015} proved orbital stability
in the energy space for these one-dimensional solitons with nonzero speed and
for initially well-separated chains with distinct nonzero speeds,
ordered according to their speeds. Bahri established asymptotic
stability, in the weak energy topology modulo symmetries, first for
individual solitons with nonzero speed~\cite{bahri2016} and then for
suitably separated and ordered chains~\cite{bahri2018}.

\medskip

De Laire~\cite[Theorem~1.6]{delaire2014} also established precise
asymptotics at infinity for finite-energy travelling waves with
$0<c<1$, including convergence to a constant equatorial state. At zero
speed, finite-energy profiles in dimensions $d\ge2$ are
constant~\cite[Proposition~1.2]{delaire2014}. Both results assume uniform
continuity when $d\ge3$. The survey~\cite{delaire2022} gives further
background on the Cauchy problem, asymptotic regimes and soliton
stability for the Landau--Lifshitz equation; see
also~\cite{bejenaruionescukenigtataru2011,changshatahuhlenbeck2000,delairegravejat2018,dingwang2001,gustafsonkangtsai2008,gustafsonnakanishitsai2010,gustafsonshatah2002,nahmodstefanovuhlenbeck2003,sulemsulembardos1986}.

\subsection{Relation with the Gross--Pitaevskii equation}
The Landau--Lifshitz travelling-wave problem shares several features
with the corresponding Gross--Pitaevskii problem.
For the latter, travelling-wave profiles are complex-valued maps
$\psi:\R^d\longrightarrow\mathbb C$ satisfying
\begin{equation}\label{eq:intro-gp}
\Delta\psi+ic\partial_1\psi+(1-|\psi|^2)\psi=0,
\end{equation}
with associated energy
\[
E_{\mathrm{GP}}(\psi)
=\frac12\int_{\R^d}|\nabla\psi|^2\dd x
+\frac14\int_{\R^d}(1-|\psi|^2)^2\dd x.
\]
Both problems have a circle of constant zero-energy states, an
energy--momentum variational structure, and travelling vortex
configurations. Here momentum refers to the formally conserved
quantity associated with translation invariance in the direction of
propagation. Linearisation about the constant states gives sound
speeds $\sqrt2$ for Gross--Pitaevskii and $1$ for Landau--Lifshitz.
For Gross--Pitaevskii, Gravejat~\cite{gravejat2003} proved nonexistence
of nonconstant finite-energy waves at supersonic speeds.

\medskip

The numerical and asymptotic work of Jones and
Roberts~\cite{jonesroberts1982} and Jones, Putterman and
Roberts~\cite{jonesputtermanroberts1986} motivated a substantial
existence theory for~\eqref{eq:intro-gp}. Early constructions include
the small-speed planar waves of B\'ethuel and Saut~\cite{bethuelsaut1999}
and the vortex rings of B\'ethuel, Orlandi and
Smets~\cite{bethuelorlandismets2004}. Chiron~\cite{chiron2004} gave a
mountain-pass construction of vortex-ring solutions in dimensions
$d\ge3$, obtaining existence for all sufficiently small positive speeds.
Further planar multivortex waves
were obtained by Liu and Wei~\cite{liuwei2020}, while Chiron and
Pacherie~\cite{chironpacherie2021} constructed a smooth branch of
vortex--antivortex waves at sufficiently small positive speeds.
For sufficiently large momentum, Chiron and
Pacherie~\cite{chironpacherie2023} subsequently proved uniqueness, up to
the symmetries of the equation, of the energy minimiser at fixed
momentum and established orbital stability of the corresponding
two-vortex travelling wave.
Complementary variational approaches have been developed
in~\cite{bethuelgravejatsaut2009,chironmaris2017}. These include minimisation of the
energy at fixed momentum, with the wave speed arising naturally as
the Lagrange multiplier associated with the momentum constraint.

\medskip

Concerning existence over larger speed intervals,
Mari\c{s}~\cite{maris2013} constructed travelling waves at every
subsonic speed in dimensions $d\ge3$ using a Pohozaev constraint. Bellazzini and
Ruiz~\cite{bellazzini2023} obtained waves of Morse index at most one
for almost every subsonic speed in dimensions two and three,
recovering all speeds in dimension three by compactness; their planar
compactness argument requires an additional hypothesis on the vortex
set. Gui, Lai, Qin and Wei~\cite{guilai2026} established the full planar
range through an energy estimate involving the action and Morse index.
The transonic regime and its relation to the Kadomtsev--Petviashvili~I
equation have been studied
in~\cite{bethuelgravejatsaut2008kpi,chironmaris2014,liuwangweiyang2026}.
See also the survey~\cite{bethuelgravejatsaut2008survey}.

\medskip

The main differences between the travelling-wave
equations~\eqref{eq:main} and~\eqref{eq:intro-gp} arise from the sphere
constraint in the Landau--Lifshitz problem. Where $u_3\ne-1$, the
stereographic coordinate $w=(u_1+iu_2)/(1+u_3)$ transforms~\eqref{eq:main}
into
\begin{equation}\label{eq:intro-stereographic}
\Delta w+ic\partial_1w+\frac{1-|w|^2}{1+|w|^2}w
=\frac{2\overline w}{1+|w|^2}
 \sum_{j=1}^d(\partial_jw)^2.
\end{equation}
The quadratic gradient term on the right-hand side
of~\eqref{eq:intro-stereographic} has no counterpart in the semilinear
equation~\eqref{eq:intro-gp}, and the stereographic chart may become
singular at a smooth magnetic vortex core. Both travelling-wave
problems require care in defining momentum and in handling the
possible absence of a global phase. In the Landau--Lifshitz problem,
the sphere constraint also permits harmonic-map bubbling and creates
difficulties in passing to the limit in the nonlinear term
$|\nabla u|^2u$ in~\eqref{eq:main}. The GP theory therefore supplies
useful variational ideas, while the sphere-valued problem requires
separate geometric estimates and compactness arguments.

\medskip

There is also a distinction in the possible energies of travelling
waves. For the Gross--Pitaevskii equation in dimensions $d\ge3$,
de Laire~\cite{delaire2009} proved that every nonconstant finite-energy
travelling wave has energy bounded below by a positive constant
depending only on the dimension. The weak Landau--Lifshitz travelling
waves constructed here can have arbitrarily small positive energy in
these dimensions, even at a fixed subsonic speed.

\subsection{Ideas of the proof and organisation of the paper}

The planar construction and the construction in dimensions $d\ge3$
both start from the action
\[
J_c(u)=E(u)+cP(u),
\]
whose critical points under the sphere constraint formally solve the
travelling-wave equation~\eqref{eq:main}. Here $P$ is the momentum in
the direction of propagation, which, for a smooth map equal to $e_1$
outside a compact set, can be defined by
\[
P(u)=\int_{\R^d}
x_2\,u\cdot(\partial_1u\times\partial_2u)\dd x,
\]
where $x_1$ is the direction of propagation. The factor $x_2$ makes
extending this definition beyond compact perturbations a substantive
issue: finite energy controls the integral of the vorticity density
$u\cdot(\partial_1u\times\partial_2u)$, but does not ensure convergence
of its first moment. The two
constructions use different ways of handling this difficulty. In
each case, the desired speed $c\in(0,1)$ is fixed at the outset.

\medskip
\needspace{5\baselineskip}
\noindent\textbf{Strategy of the planar proof.}\par
\smallskip

The planar strategy is to construct a mountain-pass sequence for a
regularised action, obtain an exact regularised wave, and then remove
the regularisation. In Section~\ref{sec:planar-framework} we introduce
\[
E_\eta(u)=E(u)+\frac{\eta}{4}\int_{\R^2}|\nabla u|^4,
\qquad
J_{c,\eta}(u)=E_\eta(u)+cP(u),
\qquad \eta>0.
\]
The regularising $|\nabla u|^4$ term provides additional control of
the gradients at fixed $\eta$.

\medskip

Section~\ref{app:thin-belt} constructs the space of sphere-valued maps
on which we define the regularised action $J_{c,\eta}$ and carry out
the mountain-pass construction. This space is obtained as a completion
of compact perturbations of $e_1$. Its topology controls the map
uniformly, its gradient in $L^2\cap L^4$, and its third component
in $L^2$. Maps in this space approach $e_1$
uniformly at infinity, so they lie close to the equator outside a
bounded region. This allows us to extend momentum by removing the
spatial weight from the part of the integral near infinity. More
precisely, for maps whose image lies in a fixed equatorial band,
integration by parts gives the momentum formula
\[
P(u)=-\int_{\R^2}u_3\,\partial_1\theta\dd x,
\qquad \theta=\arg(u_1+iu_2),
\]
whose integrand is controlled by $|u_3||\partial_1u|$. For general maps
in the variational space, a cutoff on the target sphere localises
this phase expression and produces an additional correction
supported in a bounded region. Both contributions are integrable. The resulting momentum is a smooth real-valued
functional, independent of the cutoff, and permits maps to reach
either target pole. This extension is established in
Section~\ref{appsrc-thin-belt-a-real-momentum-on-the-entire-completed-manifold}.

\medskip

Once a suitable formulation of the momentum is established, we set up
the mountain-pass construction using continuous paths joining the
constant map $e_1$ to an explicit map $u_-$ of negative action, within
the connected component containing $e_1$. We minimise the maximal
action along these paths, defining
\[
m_\eta(c)=\inf_{\gamma:e_1\to u_-}
\max_{t\in[0,1]}J_{c,\eta}(\gamma(t)).
\]
The geometry established in Section~\ref{app:planar-window} ensures
that this level is positive: the action vanishes at $e_1$ and is
positive at small nonzero energy, so every admissible path must
cross a positive action barrier before reaching $u_-$. An upper
bound follows from the comparison path constructed in
Section~\ref{app:planar-path}. These two estimates yield
\[
0<b_c\le m_\eta(c)\le8\pi-\delta_c,
\]
where $b_c,\delta_c>0$ are independent of sufficiently small $\eta$.
The strict upper bound below $8\pi$ is important for the subsequent
compactness argument.

\medskip

The above construction allows one to find approximate critical
points of $J_{c,\eta}$ whose actions approach $m_\eta(c)$. The main
obstacle is then that bounded action does not imply bounded energy:
a large positive value of $E_\eta$ could cancel a large negative
contribution from $cP$. We therefore select the approximate critical
points with additional information, using the abstract minimax
result of Section~\ref{app:path-selection} and its application to
the planar action in Section~\ref{app:planar-selector}. They have an
approximate Morse-index bound of one, and the selection controls
the action derivatives under independent stretchings of the
longitudinal coordinate $x_1$ and the transverse coordinate $x_2$.
A deformation of the paths also ensures that the selected maps
have a definite nonzero third component somewhere.
\medskip

These additional properties yield the uniform energy bound proved in
Section~\ref{app:energy-absorption}. The
action and scaling identities control transverse variation. Regions
carrying substantial remaining energy while varying little in the
transverse direction resemble one-dimensional profiles. The
classification of stable entire one-dimensional profiles in
Section~\ref{app:signed-estimate} gives a quantitative inequality between their longitudinal gradient and
anisotropy energies. Regions where the required estimate fails
produce negative directions for the second variation; the index
bound limits how many independent such regions can occur. Localising
these estimates and combining them with the scaling identities gives
\[
E_\eta(u_n)\le C_c,
\]
uniformly along the selected sequences and uniformly for sufficiently
small $\eta$. The one-dimensional profiles arise as local limits and
need not have finite total energy, which is why the classification
includes constant-latitude helices.

\medskip

With this bound established, we pass to the limit in
Section~\ref{app:fixed-eta-recovery}, keeping $\eta>0$ fixed.
We translate the maps so that a point with nonzero
third component remains at the origin. Monotonicity of the
regularised operator gives strong local convergence of the gradients,
allowing passage through the nonlinear terms. The limit is an exact
regularised travelling wave $u_\eta$, and the separate regularity
argument in Section~\ref{app:weak-regularity} shows that $u_\eta$ is smooth. The construction preserves
the transverse-energy bound
\[
T_\eta(u_\eta)
=\int_{\R^2}
\bigl(1+\eta|\nabla u_\eta|^2\bigr)|\partial_2u_\eta|^2
\le8\pi-\delta_c.
\]
For these exact waves, the stress identities proved in
Section~\ref{app:energy-absorption} show that momentum is well defined
by a spatial cutoff limit, which agrees with the earlier definition
on the variational space. They also give the Pohozaev identity
$J_{c,\eta}(u_\eta)=T_\eta(u_\eta)$.

\medskip

Only after obtaining exact regularised waves do we let $\eta$ tend to
zero. Section~\ref{app:regularizer-decay} first proves that the added
quartic energy disappears:
\[
\eta\int_{\R^2}|\nabla u_\eta|^4\longrightarrow0.
\]
The remaining danger, addressed in
Section~\ref{app:concentration-compactness}, is that gradients
concentrate in smaller and smaller regions. Rescaling around such a concentration produces a
nonconstant harmonic sphere, which costs at least $4\pi$ of energy.
The travelling-wave equation supplies an additional restriction:
its stress identities force the signed spherical area to cancel
locally at each concentration point. This raises the minimum
concentration cost to $8\pi$, and the same cost appears in the
transverse energy. The strict bound on $T_\eta$ therefore excludes
concentration.

\medskip

Finally, the height estimate in
Section~\ref{appsrc-quartic-compactness-the-height-estimate-on-the-plane}
shows that every nonconstant regularised wave satisfies
\[
\|(u_\eta)_3\|_\infty\ge\sqrt{1-c}.
\]
After translating a point of substantial height to the origin,
smooth local convergence preserves a nonzero third component there.
A finite-energy constant map must lie on the equator, so the limit is
nonconstant. This gives the required smooth finite-energy travelling
wave at speed $c$. The planar proof is assembled in
Section~\ref{sec:planar-conclusion}.

\medskip
\needspace{5\baselineskip}
\noindent\textbf{Strategy of the higher-dimensional proof.}\par
\smallskip

For $d\ge3$, we minimise the unregularised action $E+cP$ in a class
with prescribed topological defects, introduced in
Section~\ref{app:shell-proof}. In dimension three, these are
point defects with specified positive and negative charges. The
charge describes how the map wraps around the target sphere on a
small sphere surrounding a defect. In higher dimensions, the defects
lie on $(d-3)$-dimensional spheres, with the charge measured in the
three transverse directions. We also impose a finite symmetry
condition. The defect geometry is described in
Section~\ref{appsrc-shell-proof-defect-geometry-the-comparison-map-and-rotations}.

\medskip

Each admissible map equals $e_1$ outside some compact set, so its
weighted momentum is well defined. The compact set can vary from map
to map. The locations and charges of the defects are chosen so that their
contributions cancel after projection along the direction of
propagation. Consequently,
the vorticity integrated in that direction has a scalar potential
on the transverse space, as proved in
Section~\ref{appsrc-shell-proof-cancellation-under-projection-and-the-momentum-potential}.
The estimates for this potential and for longitudinal slices in
Section~\ref{appsrc-shell-proof-longitudinal-loops-the-global-action-bound-and-its-localization}
give momentum bounds independent of the size of the support. For each fixed subsonic speed, these imply
\[
J_c(v)\ge\gamma_c E(v)
\qquad\text{when }E(v)\le M_{d,c},
\]
for suitable positive constants. Thus the action controls the energy
within a sufficiently small energy range.

\medskip

To bound the infimum of the action from above, we first construct a
nonconstant map $W:\R^d\to S^2$, equal to $e_1$ outside a compact set,
with the required defect configuration and symmetry. This construction
is carried out in
Section~\ref{appsrc-shell-proof-defect-geometry-the-comparison-map-and-rotations}
using the three-dimensional profile of
Section~\ref{app:charged-comparison}, and gives $P(W)=0$. We then
consider the rescaled maps
\[
W_\lambda(x)=W(x/\lambda),
\]
which carry the correspondingly rescaled defects. Their Dirichlet
energy is multiplied by $\lambda^{d-2}$ and their anisotropy energy by
$\lambda^d$, while their momentum remains zero. For $d\ge3$, both
energy contributions tend to zero as $\lambda\to0$. Thus $W_\lambda$
provides a nonconstant competitor with arbitrarily small energy and
action $J_c(W_\lambda)=E(W_\lambda)$. Choosing $\lambda$ sufficiently
small places it in the energy range where the preceding coercivity
estimate applies.

\medskip

In Section~\ref{appsrc-shell-proof-collapse-action-the-strict-gap-and-minimizing-sequences},
for a fixed small $\lambda$, we take a minimising sequence for
\[
I_\lambda
=\inf\left\{
J_c(v):v\in\mathcal A_\lambda,\quad E(v)\le M_{d,c}
\right\},
\]
where $\mathcal A_\lambda$ contains the prescribed defect and
symmetry conditions. Comparison with $W_\lambda$, together with the
coercivity estimate, gives
\[
E(v_n)\le C_{d,c}\lambda^{d-2}
\]
for sufficiently late terms of the sequence. Choosing $\lambda$
small places these energies strictly below $M_{d,c}/2$. This margin
leaves room for local variations in either direction, so the energy
constraint does not introduce an additional multiplier into the
travelling-wave equation.

\medskip

The energy bound gives a subsequence converging weakly in
$H^1_{\mathrm{loc}}$ and strongly in $L^2_{\mathrm{loc}}$. We must
still show that its limit is nonconstant. Prescribing defects in the
approximating maps does not by itself ensure this, because
topological information can be lost under weak convergence. We
therefore derive a local lower bound for collapse in
Section~\ref{appsrc-shell-proof-a-local-geometric-lower-bound-for-collapse}
and combine it with the momentum estimates in
Section~\ref{appsrc-shell-proof-collapse-action-the-strict-gap-and-minimizing-sequences}
to prove a strict comparison: any admissible sequence of energy
$O(\lambda^{d-2})$ that converges locally to a constant must
have limiting action greater than the action of the explicit
comparison map. A minimising sequence cannot satisfy this lower
bound, since its limiting action is at most the comparison action.
Its weak limit is consequently nonconstant. The argument does not
require the prescribed charges to survive in that limit.

\medskip

The final step uses the compact-rotation principle proved in
Section~\ref{sec:rotations} to recover the travelling-wave equation
from the near-minimality of the sequence. Weak convergence of gradients is
insufficient to pass directly through the quadratic gradient term
in~\eqref{eq:main}. We instead compare each $v_n$ with a locally
rotated map
\[
Q_t(x)v_n(x),
\]
where $Q_t(x)\in SO(3)$ equals the identity outside a compact set.
Suitable rotations preserve the prescribed charges and symmetry.
Since
\[
\partial_j(Q_tv_n)
=Q_t\partial_jv_n+(\partial_jQ_t)v_n
\]
and $Q_t$ is orthogonal, the quadratic terms $|\partial_jv_n|^2$
cancel in the difference of the Dirichlet energies. The remaining
terms contain at most one gradient factor, paired with a strongly
convergent map. The momentum difference also has a local expression
that passes to the limit.

\medskip

We can therefore pass the action increments
\[
J_c(Q_tv_n)-J_c(v_n)
\]
to the limit using weak convergence of the gradients and strong
local convergence of the maps. Near-minimality makes the limiting
increment nonnegative for both signs of small $t$, so its derivative
at $t=0$ vanishes. This gives the rotational variational identity.
Averaging test fields over the finite symmetry group then yields the
full travelling-wave equation in distributions. A related use of
target rotations preserving topological charges appears in
Ziener~\cite{ziener2026}. The equation is obtained from these local
changes in the action, without requiring a global momentum for the
weak limit.

\medskip

The application of this principle in
Section~\ref{sec:higher-conclusion} completes the higher-dimensional
proof. The resulting profile is a nonconstant finite-energy weak
solution with
\[
E(u)\le C_{d,c}\lambda^{d-2}.
\]
The scale $\lambda$ is fixed throughout the minimising-sequence
limit. Choosing a smaller $\lambda$ defines a new minimisation
problem and produces a new wave of smaller energy at the same
prescribed speed $c$.

\section{The planar variational framework}\label{sec:planar-framework}

We first fix the notation for the planar construction. Throughout
Sections~\ref{sec:planar-framework}--\ref{app:quartic-compactness},
we use coordinates $x=(x_1,x_2)$ and write
$q(u)=u\cdot(\partial_1 u\times \partial_2 u)$.
For smooth maps equal to $e_1$ outside a compact set and of degree zero, put
\begin{equation}\label{eq:planar-functionals}
 P(u)=\int x_2 q(u),\qquad
 E_\eta(u)=E(u)+\frac\eta4\int |\nabla u|^4,\qquad
 J_{c,\eta}(u)=E_\eta(u)+cP(u).
\end{equation}
All unspecified planar integrals and norms are over $\R^2$.
The momentum $P(u)$ is extended to the completed variational space in
\cref{prop:thin-belt}; no absolute first moment is assumed for a weak limit.

To distinguish variations in the two coordinate directions, we separate the gradient energy into its longitudinal and transverse parts and record the associated dilation derivatives:
\begin{equation}\label{eq:planar-scaling-data}
\begin{gathered}
A=\tfrac12\int|\partial_2 u|^2,\quad K_1=\tfrac12\int|\partial_1 u|^2,\quad
V=\tfrac12\int u_3^2,\\
Q_1=\tfrac14\int|\partial_1 u|^4,\quad
Q_m=\tfrac12\int|\partial_1 u|^2|\partial_2 u|^2,\quad
Q_2=\tfrac14\int|\partial_2 u|^4,\\
D_1=A-K_1+V+\eta(-3Q_1-Q_m+Q_2),\\
D_2=-A+K_1+V+cP+\eta(Q_1-Q_m-3Q_2),\\
T_\eta(u)=\int(1+\eta |\nabla u|^2)|\partial_2 u|^2,\qquad
W_\eta(u)=\bigl(1+E_\eta(u)+\sqrt{1+P(u)^2}\bigr)^{1/2}.
\end{gathered}
\end{equation}
Here $D_1(u)$ and $D_2(u)$ are the derivatives of the action under
separate dilations in the $x_1$ and $x_2$ directions. More precisely, set
\[
u_{\alpha,\beta}(x_1,x_2)=u(e^{-\alpha}x_1,e^{-\beta}x_2).
\]
Keeping $c$ and $\eta$ fixed, we have
\[
D_1(u)=\left.\frac{\partial}{\partial\alpha}
J_{c,\eta}(u_{\alpha,\beta})\right|_{\alpha=\beta=0},
\qquad
D_2(u)=\left.\frac{\partial}{\partial\beta}
J_{c,\eta}(u_{\alpha,\beta})\right|_{\alpha=\beta=0}.
\]
Direct subtraction gives
\begin{equation}\label{eq:transverse-action}
J_{c,\eta}(u)-D_2(u)=T_\eta(u).
\end{equation}

Variations preserving the sphere constraint are described by tangent fields. We record the norms used to measure them, together with the first- and second-variation expressions for the action.

For tangent fields $v\cdot u=0$, define
\begin{equation}\label{eq:planar-variation-norms}
\begin{aligned}
N_\eta(v)&=\|\nabla v\|_2+\eta^{1/4}\|\nabla v\|_4
 +\|v_3\|_2+\|v\|_\infty,\\
B_{\eta,u}(v)&=N_\eta(v)+\bigl\|\sqrt{|\nabla u|^2+u_3^2}\,v\bigr\|_2\\
&\quad+\sqrt\eta\bigl\||\nabla u|^2v\bigr\|_2
 +\sqrt\eta\bigl\||\nabla u|\nabla v\bigr\|_2,\\
R_{c,\eta,u}[v]&=\int(1+\eta |\nabla u|^2)\nabla u:\nabla v+u_3\,v_3
 +c\,u\cdot(v\times \partial_1 u),\\
\mathcal H_{c,\eta,u}[v,v]&=\int(1+\eta |\nabla u|^2)(|\nabla v|^2-|\nabla u|^2|v|^2)
 +2\eta(\nabla u:\nabla v)^2\\
&\hspace{2em}+v_3^2-u_3^2|v|^2+c\,u\cdot(v\times \partial_1 v).
\end{aligned}
\end{equation}
The colon denotes contraction of derivative matrices. The threshold index
$i_\theta(u;c,\eta)$ is the largest dimension of a space of compact smooth
tangent fields on which $\mathcal H[v,v]\le-\theta B_{\eta,u}(v)^2$.
For fixed $u,c,\eta$, we also write $Q(v)=\mathcal H_{c,\eta,u}[v,v]$
for this second-variation quadratic form.
We use the same notation with a domain indicated when an estimate is local.

The proof proceeds by constructing a minimax sequence for fixed
$\eta$, bounding its energy, and then taking the two successive limits.
The estimates below are uniform for $c$ in a fixed compact subset of
$(0,1)$ whenever this is stated explicitly.

\section{The completed variational space}\label{app:thin-belt}

We now construct the space of sphere-valued maps for the mountain-pass
argument by completing smooth compact perturbations of $e_1$ in $N_\eta$.
This norm controls $u-e_1$ uniformly, $\nabla u$ in $L^2\cap L^4$, and
$u_3$ in $L^2$, allowing us to extend the momentum $P(u)$ smoothly to the
completed space. Weighted tangent norms then provide a complete metric
and uniform bounds for the second and third derivatives of the action,
including its dependence on independent dilations in $x_1$ and $x_2$.

For $0<\eta\le1$, let $X_\eta$ be the completion of
$C_c^\infty(\R^2;\R^3)$ for $N_\eta$ in
\eqref{eq:planar-variation-norms}. Set \[
 \mathcal M_\eta=\{u=e_1+h:h\in X_\eta,\ |u(x)|=1\text{ for every }x\},
\] and let \(\mathcal M_\eta^0\) be the connected component containing \(e_1\). Use \eqref{eq:planar-scaling-data}, writing $K(u)=K_1(u)$
in this section. On smooth maps equal to \(e_1\) outside a compact set, set \[
 P_0(u)=\int x_2\,u\cdot(\partial_1 u\times \partial_2 u).
\]

\subsection{Completion, approximation, and charts}\label{appsrc-thin-belt-completion-approximation-and-charts}

We begin by giving $\mathcal M_\eta$ smooth local coordinates. We also show that its elements can be approximated by smooth compact perturbations of $e_1$, allowing identities for such maps to pass to the completion.

\begin{lemma}[Completion and spherical charts]\label{lem:thin-belt-charts}
The completion $X_\eta$ is identified with continuous vector fields vanishing at infinity, with the weak derivatives and height component controlled by $N_\eta$.
\(\mathcal M_\eta\) is a smooth real Banach manifold, with tangent space and local spherical chart \[
\begin{aligned}
T_u\mathcal M_\eta&=\{v\in X_\eta:u\cdot v=0\},\\
\Psi_u(\xi)&=\cos|\xi|\,u+\operatorname{sinc}(|\xi|)\xi,\\
\operatorname{sinc}(t)&=
\begin{cases}\sin(t)/t&t\ne0,\\1&t=0.\end{cases}
\end{aligned}
\] Smooth compact perturbations of \(e_1\) are dense.
\end{lemma}

\begin{proof}
Fix \(0<\eta\le1\). An \(N_\eta\)-Cauchy sequence of compact smooth fields converges uniformly to a continuous field vanishing at infinity. Its derivatives converge in \(L^2\) and \(L^4\), and its third component converges in \(L^2\). Testing against compact smooth functions identifies these limits with the weak derivatives and third component of the uniform limit. If the uniform limit is zero, these distributional limits are zero. Thus the abstract completion is injectively identified with the asserted fields.

The product bounds used below are \begin{equation}
\|\nabla(ab)\|_p\le \|a\|_\infty\|\nabla b\|_p
                         +\|b\|_\infty\|\nabla a\|_p,\quad p=2,4,
 \qquad
 \|ab\|_2\le\|a\|_\infty\|b\|_2.
\label{eq:thin-belt-13}
\end{equation} All products intended as third components below have an \(L^2\) height factor. Approximation by compact fields shows that the products belong to the completion, not just to the larger class of fields with finite displayed norm. For the specific smooth compositions used below, the chain rule, \eqref{eq:thin-belt-13}, and Taylor's integral remainder prove smoothness: their coefficient derivatives are bounded on a fixed finite-dimensional range, and every derivative has the same height structure. We display that structure in the relevant formulas.

The sphere constraint is closed under uniform convergence. For \(u=e_1+h\in\mathcal M_\eta\), choose compact smooth \(h_n\to h\) in \(N_\eta\). For large \(n\), \(1/2<|e_1+h_n|<3/2\), and \[
 u_n=(e_1+h_n)/|e_1+h_n|
\] is a smooth compact perturbation of \(e_1\). The bounded, Lipschitz derivative of normalisation on this annulus gives convergence of the gradients in \(L^2\cap L^4\), by the chain rule. The height estimate \[
 \|(u_n)_3-u_3\|_2
 \le 2\|(h_n)_3-h_3\|_2+
                    C\|h_n-h\|_\infty\|h_3\|_2
\] and uniform convergence prove \(N_\eta(u_n-u)\to0\).

The projection \(\Pi_uv=v-(u\cdot v)u\) is a bounded linear operator on \(X_\eta\). In fact, \[
 |\nabla\Pi_uv|\le C(|\nabla v|+|v||\nabla u|),
 \qquad (\Pi_uv)_3=v_3-(u\cdot v)u_3 .
\] Approximating \(u,v\) by the compact fields just constructed shows that this operator maps into \(X_\eta\). Its range is the closed complemented subspace \(u\cdot v=0\).

For \(\|\xi\|_\infty<1/4\) in that subspace, the stated formula \(\Psi_u(\xi)\) is sphere-valued, and \[
 (\Psi_u(\xi)-u)_3=(\cos|\xi|-1)u_3+
                                  \operatorname{sinc}(|\xi|)\xi_3 .
\] The analytic series in \(|\xi|^2\), together with \eqref{eq:thin-belt-13}, show that \(\Psi_u\) is a smooth map into \(\mathcal M_\eta\). Its inverse near \(u\) is \begin{equation}
\xi=\frac{\arccos(u\cdot w)}{\sqrt{1-(u\cdot w)^2}}\,
                   \bigl(w-(u\cdot w)u\bigr),
\label{eq:thin-belt-14}
\end{equation} where the scalar factor has its smooth continuation at \(u\cdot w=1\). This factor is a convergent power series in \(1-u\cdot w\) near zero. The third component in parentheses is \(w_3-(u\cdot w)u_3\); thus \eqref{eq:thin-belt-13} also proves that the inverse is smooth in \(X_\eta\). These charts have derivative the identity at zero and prove the manifold and tangent-space assertions. Restrictions to open balls for equivalent tangent norms are still charts onto open sets.
\end{proof}

\subsection{Extension of the momentum}\label{appsrc-thin-belt-a-real-momentum-on-the-entire-completed-manifold}

To define the action on this completed space, we must extend the momentum beyond smooth compact perturbations of $e_1$.

\begin{proposition}[Extension of momentum]\label{prop:thin-belt}
The functional \(P_0\) has a unique smooth real extension \(P\) to all of \(\mathcal M_\eta\). Here is an absolutely convergent formula for it. Fix any smooth \(\chi:\mathbb R\to[0,1]\) equal to one on \([-1/4,1/4]\) and zero outside \((-1/2,1/2)\). On the sphere put \[
\begin{aligned}
\omega_q(a,b)&=q\cdot(a\times b),\\
\mathfrak a&=-\chi(q_3)q_3\,d\theta,\\
\mathfrak b&=\omega-d\mathfrak a,
\end{aligned}
\] where \(\theta\) is the equatorial angular coordinate and \(\mathfrak a\) is extended by zero near the poles. Thus \(\mathfrak a\) is a globally smooth one-form; the symbol \(d\) here is exterior differentiation on the target sphere. Then \begin{equation}
P(u)=\int\mathfrak a_u(\partial_1 u)
         +\int x_2\,\mathfrak b_u(\partial_1 u,\partial_2 u).
\label{eq:thin-belt-1}
\end{equation} The second integrand has compact support depending on \(u\), and \(P\) is independent of the chosen cutoff \(\chi\).
\end{proposition}

\begin{proof}
In coordinates \((\theta,q_3)\) on the target band, \(\omega=d\theta\wedge dq_3\). Thus, with the notation in the statement, \begin{equation}
\begin{aligned}
\mathfrak b&=\rho(q_3)\omega,&
 \rho(q_3)&=1-\chi(q_3)-q_3\chi'(q_3),\\
\mathfrak a_q(a)&=-\phi(q_3)(q_1a_2-q_2a_1),&
 \phi(q_3)&=\frac{\chi(q_3)q_3}{1-q_3^2}.
\end{aligned}
\label{eq:thin-belt-15}
\end{equation} Extend \(\phi\) by zero where \(\chi=0\) near the poles. It is smooth, all its derivatives are bounded, and \(|\phi(q_3)|\le C|q_3|\). The function \(\rho\) and all its derivatives vanish on \((-1/4,1/4)\).

Let \(L(u)=\int\mathfrak a_u(\partial_1 u)\). Its integrand is bounded by \(C|u_3||\partial_1 u|\), and is therefore absolutely integrable. Since \(u_3(x)\to0\) uniformly at infinity, \(\mathfrak b_u(\partial_1 u,\partial_2 u)\) vanishes outside a compact set. Its absolute value is at most \(C|\partial_1 u||\partial_2 u|\), an \(L^1\) function. The coordinate \(x_2\) is bounded on that compact set. Both terms of \eqref{eq:thin-belt-1} are therefore well-defined.

Moreover, this support is locally uniform in the manifold. For a fixed base map \(u\), choose a disk outside which \(|u_3|<1/8\). If \(\|w-u\|_\infty<1/16\), then outside that same disk \(|w_3|<1/4\), so the curvature term and all its coefficient derivatives vanish there. The integral of its density is smooth in each such chart: two spatial derivative factors are paired in \(L^2\), the coordinate weight is bounded on the fixed disk, and all remaining variation factors are bounded. Differentiated chart heights and gradients have the structure proved in \cref{appsrc-thin-belt-completion-approximation-and-charts}. The same product estimates bound Taylor remainders. For \(L\), a differentiated height coefficient is either still bounded by \(C|u_3|\) times bounded variations, or supplies the \(L^2\) height of a variation. This is paired with one \(L^2\) spatial derivative. Hence \eqref{eq:thin-belt-1} defines a smooth real function.

For a smooth compact perturbation, \(u^*d\mathfrak a=d(u^*\mathfrak a)\). Writing \(a_1=\mathfrak a_u(\partial_1 u)\), \(a_2=\mathfrak a_u(\partial_2 u)\), integration by parts gives \[
 \int x_2\,u^*d\mathfrak a(\partial_1 ,\partial_2 )
   =\int x_2(\partial_1  a_2-\partial_2  a_1)=\int a_1 .
\] It follows that \eqref{eq:thin-belt-1} agrees with \(P_0\). Density from \cref{appsrc-thin-belt-completion-approximation-and-charts} proves both uniqueness of the continuous extension and independence of \(\chi\).

For estimates with no spatial weight, its differential has the following integrable expression: \begin{equation}
DP(u)[v]=DL(u)[v]+\int\mathfrak b_u(v,\partial_1 u).
\label{eq:thin-belt-16}
\end{equation} The terms of \(DL\) have pointwise bound \begin{equation}
C\bigl(|v_3||\partial_1 u|+|u_3||v||\partial_1 u|+|u_3||\partial_1 v|\bigr).
\label{eq:thin-belt-17}
\end{equation} The last integral in \eqref{eq:thin-belt-16} has integrand bounded by \(C|u_3||v||\partial_1 u|\), since \(|\rho(u_3)|\le C|u_3|\).

To prove \eqref{eq:thin-belt-16} including its sign, first use compact smooth maps and tangent fields. Expanding the triple products and using that three tangent vectors to \(S^2\) have zero triple product gives \[
 \frac{d}{dt}\omega_{u_t}(\partial_1 u_t,\partial_2 u_t)\big|_{t=0}
 =\partial_1 \omega_u(v,\partial_2 u)+\partial_2 \omega_u(\partial_1 u,v).
\] Multiplication by \(x_2\) and integration yields \(DP_0(u)[v]=\int\omega_u(v,\partial_1 u)\). Differentiating \(L\) and integrating in \(x_1\) gives \(DL(u)[v]=\int(d\mathfrak a)_u(v,\partial_1 u)\). Their difference is \eqref{eq:thin-belt-16}. For a general chart segment \(w(t)=\Psi_u(\xi+t v)\), approximate \(u\) by compact smooth \(u_n\) and approximate \(\xi,v\) by compact smooth fields followed by \(\Pi_{u_n}\). The projection estimates in \cref{appsrc-thin-belt-completion-approximation-and-charts} give convergence of the tangent fields in \(X_\eta\), and of the chart segments and their velocities uniformly on a small \(t\)-interval in that norm. The integrable bound \eqref{eq:thin-belt-17} and its product-rule derivatives give continuity of the right side of \eqref{eq:thin-belt-16} in these charts. Passing to the fundamental theorem of calculus for the compact segments proves \eqref{eq:thin-belt-16}.
\end{proof}

\subsection{Weighted products and first variations}\label{appsrc-thin-belt-weighted-products-and-first-variations}

Differentiating the action produces products of variations with $u_3$
and $\nabla u$. We therefore use the weighted tangent norm
$B_{\eta,u}$ from \eqref{eq:planar-variation-norms}, which controls
these products directly. This allows the energy and momentum variations
to be estimated with constants independent of $u$ and $\eta$.

\begin{lemma}[Weighted first-variation estimates]\label{lem:thin-belt-products}
For $u\in\mathcal M_\eta$ and $v\in T_u\mathcal M_\eta$,
\begin{equation}
|DE_\eta(u)[v]|+|DP(u)[v]|
         \le C\sqrt{E_\eta(u)}\,B_{\eta,u}(v).
\label{eq:thin-belt-4}
\end{equation}
The same bound holds for $F\in\{A,K,V,\eta Q_1,\eta Q_m,\eta Q_2\}$ in place of $E_\eta$. The constant is independent of $u$ and $\eta$.
\end{lemma}

\begin{proof}
For this section abbreviate \(B(v)=B_{\eta,u}(v)\), \(d=|\nabla u|\). Definition \eqref{eq:planar-variation-norms} immediately gives \begin{equation}
\begin{gathered}
 \|v\|_\infty,\ \|v_3\|_2,\ \|\nabla v\|_2,\
 \|dv\|_2,\ \|u_3\,v\|_2\le B(v),\\
 \eta^{1/4}\|\nabla v\|_4,\quad
 \sqrt\eta\|d^2v\|_2,\quad
 \sqrt\eta\|d\nabla v\|_2\le B(v),\\
 \sqrt\eta\||\nabla v||\nabla h|\|_2\le B(v)B(h),\qquad
 \eta^{1/4}\|dv\|_4\le B(v).                                  
 \end{gathered}
\label{eq:thin-belt-18}
\end{equation} Here \(h\) in the last line is any tangent field at \(u\). The penultimate inequality is Hölder's inequality. For the last one, square its left side and use \[
 \sqrt\eta\bigl(\int d^4|v|^4\bigr)^{1/2}
 \le \sqrt\eta\|d^2v\|_2\|v\|_\infty\le B(v)^2.
\]

The derivative of the quadratic gradient energy pairs \(\nabla v\) with \(\nabla u\), the height energy pairs \(v_3\) with \(u_3\), and the quartic energy pairs \(\sqrt\eta\,d\nabla v\) with \(\sqrt\eta\,d^2\). The latter three unvaried factors have \(L^2\) norm at most \(2\sqrt{E_\eta(u)}\). Formulae \eqref{eq:thin-belt-16}--\eqref{eq:thin-belt-17} pair \(v_3,u_3\,v,\partial_1 v\), or \(dv\) with \(\partial_1 u\) or \(u_3\). This proves \eqref{eq:thin-belt-4}. The identical product bounds prove \begin{equation}
|DF(u)[v]|\le C\sqrt{E_\eta(u)}B(v)
 \quad(F=A,K,V,\eta Q_1,\eta Q_m,\eta Q_2).
\label{eq:thin-belt-19}
\end{equation} Each differentiated quartic directional monomial is bounded by \(C\eta d^3|\nabla v|\), which is the pairing already described.

For subsequent estimates let \(B(\xi)=r\le1\) and put \[
 e=|\nabla\xi|,\quad a_0=d+e,\quad
 a_1=|\nabla v|+a_0|v|,\quad Z=|u_3|+|\xi_3|.
\] From \eqref{eq:thin-belt-18} and Hölder we have \begin{equation}
\begin{aligned}
 \|a_1\|_2+\|a_0v\|_2+\|Zv\|_2&\le C B(v),\\
 \eta^{1/4}\|a_1\|_4+\sqrt\eta\|a_0a_1\|_2+
                  \sqrt\eta\|a_0^2v\|_2&\le C B(v).           
 \end{aligned}
\label{eq:thin-belt-20}\end{equation} To list the extra factors needed for these assertions, \begin{equation}
\begin{gathered}
 \|ev\|_2\le r\|v\|_\infty,\quad
 \eta^{1/4}\|ev\|_4\le r\|v\|_\infty,\quad
 \|\xi_3v\|_2\le r\|v\|_\infty,\\
 \sqrt\eta\|e\nabla v\|_2\le r B(v),\qquad
 \sqrt\eta\|de\|_2\le r,\qquad
 \sqrt\eta\|e^2\|_2\le r^2,\\
 \sqrt\eta\|a_0^2v\|_2
 \le2\sqrt\eta\|d^2v\|_2+
                         2\sqrt\eta\|e^2\|_2\|v\|_\infty .
 \end{gathered}
\label{eq:thin-belt-21}
\end{equation} The \(L^4\) estimate for \(a_0v\) follows from the last inequality in \eqref{eq:thin-belt-18} and the second estimate in \eqref{eq:thin-belt-21}. Expanding \(a_0a_1=a_0|\nabla v|+a_0^2|v|\) completes \eqref{eq:thin-belt-20}.
\end{proof}

\subsection{Transport and uniform derivatives}\label{appsrc-thin-belt-transport-and-uniform-derivatives}

We now apply these product estimates to the spherical charts, controlling both the change in tangent norms and the second and third derivatives of the action.

\begin{proposition}[Uniform chart estimates]\label{prop:thin-belt-uniform}
There are constants \(r_0>0\) and \(C\ge1\), independent of \(u,\eta\), support size, and total energy, with the following properties. If \(r=B_{\eta,u}(\xi)<r_0\), \(w=\Psi_u(\xi)\), and \(T=D\Psi_u(\xi)\), then for every \(v\in T_u\mathcal M_\eta\), \begin{equation}
(1-Cr)B_{\eta,u}(v)\le B_{\eta,w}(Tv)
                      \le(1+Cr)B_{\eta,u}(v).
\label{eq:thin-belt-3}
\end{equation}  For \(k=2,3\) and \(F\in\{E_\eta,P,A,K,V,\eta Q_1,\eta Q_m,\eta Q_2\}\), \begin{equation}
\|D^k(F\circ\Psi_u)(\xi)\|_{\mathrm{op},B_{\eta,u}}\le C.
\label{eq:thin-belt-5}
\end{equation} The operator norm uses \(B_{\eta,u}\) in every argument.
\end{proposition}

\begin{proof}
Choose \(r_0<1/4\). For \(B(\xi)=r<r_0\), put \(w=\Psi_u(\xi)\), \(T=D\Psi_u(\xi)\). Differentiation of the sine and cosine series gives \begin{equation}
\begin{aligned}
 |w-u|&\le C|\xi|,&
 |\nabla w-\nabla u|&\le C(|\xi|d+e),\\
 |Tv-v|&\le C|\xi||v|,&
 |\nabla(Tv-v)|&\le
       C(|\xi||\nabla v|+e|v|+|\xi|d|v|),\\
 |w_3-u_3|&\le C(|\xi|^2|u_3|+|\xi_3|),&
 |(Tv-v)_3|&\le C\bigl(|\xi|^2|v_3|
                      +|\xi||v|(|u_3|+|\xi_3|)\bigr).
 \end{aligned}
\label{eq:thin-belt-22}\end{equation} For example, with \(\operatorname{sinc}'(r)/r\) interpreted at zero, \[
 Tv=\operatorname{sinc}(|\xi|)v
   -\operatorname{sinc}(|\xi|)(\xi\cdot v)u
   +\frac{\operatorname{sinc}'(|\xi|)}{|\xi|}(\xi\cdot v)\xi ,
\] which gives the last height estimate. Equations \eqref{eq:thin-belt-18}--\eqref{eq:thin-belt-22} yield \(N_\eta(Tv-v)\le CrB(v)\).

For the first extra norm in \eqref{eq:planar-variation-norms}, use the tensor \((\nabla u,u_3)\otimes v\). The difference after transport is bounded in \(L^2\) by \[
 \|(\nabla w-\nabla u,w_3-u_3)\otimes v\|_2
 +\|(\nabla w,w_3)\otimes(Tv-v)\|_2\le CrB(v).
\] For the second extra norm use \[
 \bigl||\nabla w|^2-d^2\bigr|
       \le C(rd^2+de+e^2),\qquad |\nabla w|\le C(d+e).
\] The change is at most \[
 C\sqrt\eta\|(rd^2+de+e^2)v\|_2
       +Cr\sqrt\eta\|(d+e)^2v\|_2\le CrB(v)
\] by \eqref{eq:thin-belt-18}--\eqref{eq:thin-belt-21}. For the final extra norm represent it by \(\sqrt\eta\|\nabla u\otimes\nabla v\|_2\). The difference is bounded by \[
 \sqrt\eta\||\nabla w-\nabla u||\nabla v|\|_2
 +\sqrt\eta\||\nabla w||\nabla(Tv-v)|\|_2.
\] After \eqref{eq:thin-belt-22}, its factors are \(rd|\nabla v|,\ e|\nabla v|,\ de|v|,\
e^2|v|,\ rd^2|v|,\ rde|v|\), each controlled by \eqref{eq:thin-belt-18}--\eqref{eq:thin-belt-21}. The reverse triangle inequality for these norms proves \eqref{eq:thin-belt-3}. Reduce \(r_0\) so that \(Cr_0<1/4\).

We now prove \eqref{eq:thin-belt-5}. Fix \(\xi,v\) in the same tangent space and write \(w(t)=\Psi_u(\xi+tv)\). For \(j=1,2,3\), at \(t=0\) the differentiated chart expressions obey \begin{equation}
|\partial_t^jw|\le C_j|v|^j,\quad
 |\nabla\partial_t^jw|\le C_j a_1|v|^{j-1},\quad
 |(\partial_t^jw)_3|\le
       C_j(|v_3|+Z|v|)|v|^{j-1}.
\label{eq:thin-belt-23}
\end{equation} Also \(|\nabla w|\le Ca_0\), \(|w_3|\le CZ\). The second and third derivatives of a quadratic gradient term are bounded respectively by \[
 C\int(a_1^2+a_0a_1|v|),\qquad
 C\int(a_1^2+a_0a_1|v|)|v|.
\] Pair \(a_1\) and \(a_0v\) in \(L^2\) using \eqref{eq:thin-belt-20}. The bounds are \(CB(v)^2\) and \(CB(v)^3\). For the height term replace \(a_1,a_0\) by \(|v_3|+Z|v|,Z\), whose required \(L^2\) products satisfy the same bounds in \eqref{eq:thin-belt-20}.

For each quartic directional monomial the bounds are \begin{equation}
C\eta\int(a_0^2a_1^2+a_0^3a_1|v|),\qquad
 C\eta\int(a_0a_1^3+a_0^2a_1^2|v|+a_0^3a_1|v|^2).
\label{eq:thin-belt-24}
\end{equation} The term with \(a_0^3a_1|v|\) is a pairing of \(\sqrt\eta a_0^2|v|\) and \(\sqrt\eta a_0a_1\) in \(L^2\). The cubic term is bounded by \[
 \eta\int a_0a_1^3
 \le \bigl(\sqrt\eta\|a_0a_1\|_2\bigr)
                      \bigl(\eta^{1/4}\|a_1\|_4\bigr)^2 .
\] The remaining terms in \eqref{eq:thin-belt-24} use the same factors and \(\|v\|_\infty\le B(v)\). This proves the required powers of \(B(v)\).

For momentum, every positive \(j\)-th derivative, \(1\le j\le3\), of \(\phi(w_3(t))\) satisfies \begin{equation}
|\partial_t^j\phi(w_3(t))|_{t=0}
 \le C_j\bigl(|v_3||v|^{j-1}+Z|v|^j\bigr),
\label{eq:thin-belt-25}
\end{equation} while its zeroth derivative is bounded by \(CZ\). In each product-rule term retain the height structure of one positive-order derivative of \(w_3\); bound all other such derivatives by their \(C|v|^\ell\) bound from \eqref{eq:thin-belt-23}. This proves \eqref{eq:thin-belt-25}. Using the ambient formula for \(\mathfrak a\) in \eqref{eq:thin-belt-15}, the \(k\)-th derivative of \(L(w(t))\), \(k=2,3\), is bounded by \begin{equation}
C\int\bigl(
 Za_0|v|^k+|v_3|a_0|v|^{k-1}
       +Z|\nabla v||v|^{k-1}
       +|v_3||\nabla v||v|^{k-2}\bigr).
\label{eq:thin-belt-26}
\end{equation} For the other term in \eqref{eq:thin-belt-16}, differentiate \(\int\mathfrak b_{w(t)}(w_t(t),\partial_1 w(t))\) exactly \(k-1\) times. Every coefficient derivative vanishes unless \(|w_3|\ge1/4\). By reducing \(r_0\) again we have \(|w_3-u_3|\le1/8\), so on this set \(|u_3|\ge1/8\). Each bounded coefficient derivative is therefore bounded by \(C|u_3|\le CZ\). The product rule and \eqref{eq:thin-belt-23} bound this contribution by the first and third terms in \eqref{eq:thin-belt-26}. Pair, in order, \[
 (Zv,a_0v),\quad(v_3,a_0v),\quad(Zv,\nabla v),\quad(v_3,\nabla v)
\] in \(L^2\), and use \(\|v\|_\infty^{k-2}\) for the remaining powers. Equation \eqref{eq:thin-belt-20} gives \(CB(v)^k\). Thus the momentum derivatives are uniformly bounded without any factor measuring the position or diameter of the polar region.

These are diagonal bounds for symmetric multilinear derivatives. The real polarisation formula \[
 T(v_1,\ldots,v_k)=\frac1{2^kk!}
 \sum_{\epsilon_i=\pm1}\Bigl(\prod_i\epsilon_i\Bigr)
 T\Bigl(\sum_i\epsilon_i v_i,\ldots,\sum_i\epsilon_i v_i\Bigr)
\] turns them into the operator bounds \eqref{eq:thin-belt-5}. Integrable products in \eqref{eq:thin-belt-20}, \eqref{eq:thin-belt-24}, and \eqref{eq:thin-belt-26} also bound Taylor remainders on a small chart neighbourhood, justifying the differentiated integrals.
\end{proof}

\subsection{Continuity and completeness of the metric}\label{appsrc-thin-belt-continuity-and-completeness-of-the-metric}

The preceding estimates are local. We now show that the weighted norms also define a complete length metric on each connected component, compatible with the topology already placed on $\mathcal M_\eta$.

\begin{proposition}[Complete length metric]\label{prop:thin-belt-metric}
The norms $B_{\eta,u}$ vary continuously in local tangent trivialisations and are locally uniformly equivalent to \(N_\eta\). The length metric, defined by taking the infimum of the integrals of these norms along piecewise \(C^1\) curves, is finite and complete on each connected component and induces its manifold topology.
\end{proposition}

\begin{proof}
Hölder's inequality and \(Q_1+Q_m+Q_2=\tfrac14\int |\nabla u|^4\) give \begin{equation}
N_\eta(v)\le B_{\eta,u}(v)
 \le\bigl[1+\sqrt{2E_\eta(u)}+2\sqrt{E_\eta(u)}
                       +(4E_\eta(u))^{1/4}\bigr]N_\eta(v).
\label{eq:thin-belt-27}
\end{equation} The three extra terms in \eqref{eq:planar-variation-norms} are bounded respectively by \(\sqrt{2E_\eta(u)}\|v\|_\infty\), \(2\sqrt{E_\eta(u)}\|v\|_\infty\), and \((4E_\eta(u))^{1/4}\eta^{1/4}\|\nabla v\|_4\). This proves local uniform equivalence because \(E_\eta\) is continuous.

For continuity of these norms in tangent trivialisations, use the tensor \((\nabla u,u_3)\otimes v\) for the first extra term, \[
 \||\nabla u|^2-|\nabla w|^2\|_2
 \le(\|\nabla u\|_4+\|\nabla w\|_4)\|\nabla u-\nabla w\|_4
\] for the second, and Hölder applied to \((|\nabla u|-|\nabla w|)\nabla v\) for the third. These estimates are uniform for \(N_\eta(v)\le1\) locally. Composing with the continuous chart derivative gives the asserted continuity.

The \(N_\eta\)-distance between endpoints of any piecewise \(C^1\) curve is at most its \(B\)-length. A Cauchy sequence for the length metric therefore converges in the ambient complete space to a point of the closed sphere set. In its inverse chart \eqref{eq:thin-belt-14}, the coordinates of the sequence tend to zero in \(N_\eta\). Local norm equivalence and bounded chart derivatives show that radial chart segments have length tending to zero. Hence this is also convergence in the length metric, and the limit lies in the same component. The same segments show that the length metric induces the manifold topology. The set of points reachable from a fixed point by finitely many chart segments is open and has open complement inside its connected component; it is therefore the component. Each such finite sequence of chart segments has finite length. This proves all metric assertions on \(\mathcal M_\eta\).
\end{proof}

\subsection{Augmentation, planar scaling, and translations}\label{appsrc-thin-belt-augmentation-planar-scaling-and-translations}

We now include changes of spatial scale among the variables of the action. Two additional parameters describe independent dilations in $x_1$ and $x_2$, allowing their derivatives to be treated together with variations of the map. We weight these directions by $W_\eta(u)$ to compensate for the growth of the scaling derivatives with energy and momentum. Together with the preceding estimates, this gives a complete metric and uniform second- and third-derivative bounds on the enlarged space.

\begin{proposition}[The augmented variational space]\label{prop:thin-belt-augmented}
For \(c\in[-1,1]\) define \begin{equation}
  \begin{aligned}
   f_{c,\eta}(\alpha,\beta,u)={}&
   e^{\alpha-\beta}A(u)+e^{-\alpha+\beta}K(u)+e^{\alpha+\beta}V(u)
                      +ce^\beta P(u)\\
   &+\eta\left(e^{-3\alpha+\beta}Q_1(u)
            +e^{-\alpha-\beta}Q_m(u)
            +e^{\alpha-3\beta}Q_2(u)\right).
   \end{aligned}
\label{eq:thin-belt-6}\end{equation} With $W_\eta$ as in
\eqref{eq:planar-scaling-data}, the continuously varying tangent norm on
$\mathcal Y_\eta=\mathbb R^2\times\mathcal M_\eta^0$ \begin{equation}
\|(a,b,v)\|_{(\alpha,\beta,u)}
           =W_\eta(u)(|a|+|b|)+B_{\eta,u}(v)
\label{eq:thin-belt-7}
\end{equation} has a finite complete length metric inducing the product topology. There are constants \(R>0\), \(M\ge1\), independent of \(c,\eta,u\), such that at every \(q=(\alpha,\beta,u)\) with \(|\alpha|,|\beta|\le1\) the chart \begin{equation}
\Phi_q(a,b,\xi)=(\alpha+a,\beta+b,\Psi_u(\xi)),
   \qquad \|(a,b,\xi)\|_q<R
\label{eq:thin-belt-8}
\end{equation} is a diffeomorphism onto an open set, has derivative the identity at zero, and satisfies throughout that ball \begin{equation}
\tfrac12\|h\|_q\le\|D\Phi_q(x)h\|_{\Phi_q(x)}\le2\|h\|_q,
   \qquad
   \|D^k(f_{c,\eta}\circ\Phi_q)(x)\|_{\mathrm{op},q}\le M
   \quad(k=2,3).
\label{eq:thin-belt-9}
\end{equation}
\end{proposition}

\begin{proof}
Since \(W_\eta\ge1\), \(W_\eta^2\ge E_\eta\), equations \eqref{eq:thin-belt-4} and \eqref{eq:planar-scaling-data} give \begin{equation}
|DW_\eta(u)[v]|
 =\left|\frac{DE_\eta(u)[v]+
       P(u)(1+P(u)^2)^{-1/2}DP(u)[v]}{2W_\eta(u)}\right|
 \le C B_{\eta,u}(v).
\label{eq:thin-belt-28}
\end{equation} Integration along \(t\mapsto\Psi_u(t\xi)\), using \eqref{eq:thin-belt-3}, shows \begin{equation}
|W_\eta(\Psi_u(\xi))-W_\eta(u)|\le Cr,\qquad
                         r=B_{\eta,u}(\xi)<r_0.
\label{eq:thin-belt-29}
\end{equation} Together with \eqref{eq:thin-belt-3} and \(W_\eta\ge1\), this proves a \(1\pm Cr\) comparison of the product norms under the derivative of \eqref{eq:thin-belt-8}. Choose \(R<r_0\) so that the comparison factors lie in \([1/2,2]\). These are actual tangent coordinates; the norm of scalar displacement \(a\) is \(W_\eta(u)|a|\), and the chart derivative at zero is the identity.

The augmented metric dominates scalar Euclidean distance and the map length distance. A Cauchy sequence has limits in both factors. Local norm equivalence and product chart segments imply convergence in the augmented metric. They also give the product topology. Finite product chart chains give finite distances on \(\mathcal Y_\eta\).

For the uniform derivative bound, fix a centre as in \eqref{eq:thin-belt-8} and set \(W=W_\eta(u)\). In its radius-\(R\) tangent ball, \(|a|+|b|\le R/W\), \(B(\xi)\le R\), and the scalar exponential coefficients in \eqref{eq:thin-belt-6}, with their first three derivatives, are bounded by numerical constants. Equations \eqref{eq:thin-belt-29} and \eqref{eq:planar-scaling-data} imply, at \(w=\Psi_u(\xi)\), \[
 A(w),K(w),V(w),\eta Q_1(w),\eta Q_m(w),\eta Q_2(w),|P(w)|
                                                     \le C W^2 .
\] Their first map derivatives have chart norm at most \(CW\), by \eqref{eq:thin-belt-4}, \eqref{eq:thin-belt-19}, \eqref{eq:thin-belt-3}, and \eqref{eq:thin-belt-29}; their second and third map derivatives have norm at most \(C\), by \eqref{eq:thin-belt-5}. In a direction of norm at most one under \eqref{eq:thin-belt-7}, each scalar component is at most \(1/W\), and its map component has \(B\)-norm at most one. At total derivative order \(k=2,3\), a product-rule term with zero map derivatives is bounded by \(CW^2/W^k\), one with one map derivative by \(CW/W^{k-1}\), and one with \(j\ge2\) map derivatives by \(C/W^{k-j}\). All are bounded by \(C\). The finite sum proves \eqref{eq:thin-belt-9}.
\end{proof}

The following identities connect this augmented functional with spatial dilations and describe the effect of translations on the energy and momentum.

\begin{lemma}[Dilations and translations]\label{lem:thin-belt-scaling}
The spatial dilations \(\mathcal T_{\alpha,\beta}u(x_1,x_2)=u(e^{-\alpha}x_1,e^{-\beta}x_2)\) act jointly continuously, preserve each connected component, and \begin{equation}
f_{c,\eta}(\alpha,\beta,u)
     =E_\eta(\mathcal T_{\alpha,\beta}u)
                    +cP(\mathcal T_{\alpha,\beta}u).
\label{eq:thin-belt-10}
\end{equation} Define \(\mathcal Q(u)=\int u\cdot(\partial_1 u\times \partial_2 u)\). It is locally constant on \(\mathcal M_\eta\) and equals zero on \(\mathcal M_\eta^0\). For \(a=(a_1,a_2)\in\mathbb R^2\), with \(\tau_a u(x)=u(x-a)\), \begin{equation}
P(\tau_a u)=P(u)+a_2\mathcal Q(u).
\label{eq:thin-belt-11}
\end{equation} In particular translations preserve \(P,E_\eta\), and the norms \eqref{eq:planar-variation-norms}, \eqref{eq:thin-belt-7} on the constant-map component. The uniform constants in \eqref{eq:thin-belt-3}--\eqref{eq:thin-belt-9} remain valid after arbitrary translations.
\end{lemma}

\begin{proof}
Dilation is a bounded linear map on \(X_\eta\), with bounds locally uniform in \((\alpha,\beta)\); this follows by changing variables in the four components of \(N_\eta\). On compact smooth fields it is jointly continuous, and density extends this assertion to \(X_\eta\). It preserves sphere values and connected components, using the path \(t\mapsto\mathcal T_{t\alpha,t\beta}u\). On compact smooth maps the Jacobian is \(e^{\alpha+\beta}\), the derivative factors are \(e^{-\alpha},e^{-\beta}\), and the coordinate \(x_2\) contributes \(e^\beta\). Consequently \begin{equation}
\begin{array}{c|ccccccc}
 F&A&K&V&P_0&Q_1&Q_m&Q_2\\ \hline
 F(\mathcal T_{\alpha,\beta}u)/F(u)&
 e^{\alpha-\beta}&e^{-\alpha+\beta}&e^{\alpha+\beta}&e^\beta&
 e^{-3\alpha+\beta}&e^{-\alpha-\beta}&e^{\alpha-3\beta}.
 \end{array}
\label{eq:thin-belt-30}
\end{equation} The notation in the table means the multiplicative scaling identity, also when \(F(u)=0\). Continuity and density extend all these identities to the completed space. This proves \eqref{eq:thin-belt-10} with the planar exponents.

The total area \(\mathcal Q(u)=\int\omega_u(\partial_1 u,\partial_2 u)\) is continuous: the difference of its integrands is a sum with one uniformly small map difference or one \(L^2\)-small gradient difference, paired with the other bounded \(L^2\) gradient. If two smooth compact maps are sufficiently uniformly close, their short spherical geodesic has common compact support. The transgression identity in \cref{appsrc-thin-belt-a-real-momentum-on-the-entire-completed-manifold}, without the factor \(x_2\), says that \(\mathcal Q\) is constant along this path. For arbitrary uniformly close points of the manifold, approximate each by compact smooth points that remain uniformly close. Continuity gives equality of their total areas. Thus \(\mathcal Q\) is locally constant, and \(\mathcal Q(e_1)=0\) gives its value on \(\mathcal M_\eta^0\).

Spatial translations are linear isometries on the ambient norm and act jointly continuously, first on compact smooth fields and then by density. They preserve components via \(t\mapsto\tau_{ta}u\). Changing variables in the compact moment gives \(P_0(\tau_a u)=P_0(u)+a_2\mathcal Q(u)\). Continuity proves \eqref{eq:thin-belt-11}. The unweighted energy and the expressions in \eqref{eq:planar-variation-norms} are translation invariant. On the constant-map component so is \(P\), hence also \(W_\eta\) and \eqref{eq:thin-belt-7}. Charts commute with translations, which proves the last uniformity assertion.
\end{proof}

\subsection{First and second variations}\label{appsrc-thin-belt-physical-variations-and-scope}

We finish by checking that differentiation in the chosen charts gives precisely the residual and second variation associated with the travelling-wave equation.

\begin{lemma}[Variations of the map]\label{lem:thin-belt-variations}
If \(u-e_1\) and a field \(v\) are smooth and compactly supported, \(|u|=1\), and \(u\cdot v=0\), the differential and the map-chart Hessian of \(E_\eta+cP\), evaluated at the chart centre, are \begin{equation}
D(J_{c,\eta}\circ\Psi_u)(0)[v]=R_{c,\eta,u}[v],\qquad
D^2(J_{c,\eta}\circ\Psi_u)(0)[v,v]=\mathcal H_{c,\eta,u}[v,v].
\label{eq:thin-belt-12}
\end{equation} For noncompact tangent fields the differential of \(P\) is the continuous differential supplied by \eqref{eq:thin-belt-1}; absolute convergence of the formal integral \(\int u\cdot(v\times \partial_1 u)\) is not asserted.
\end{lemma}

\begin{proof}
For compact smooth \(u,v\) as in the statement, let \(w(t)=\Psi_u(tv)\). Then \(w_t(0)=v\), \(w_{tt}(0)=-|v|^2u\). The first energy derivative is the first two terms of \(R_{c,\eta,u}\), and the compact first momentum derivative was computed in \cref{appsrc-thin-belt-a-real-momentum-on-the-entire-completed-manifold}. For the second energy derivative use \[
 \nabla u:\nabla(-|v|^2u)=-|\nabla u|^2|v|^2.
\] The quadratic gradient and height terms give \(|\nabla v|^2-|\nabla u|^2|v|^2+v_3^2-u_3^2|v|^2\). The quartic gradient term gives \(\eta |\nabla u|^2(|\nabla v|^2-|\nabla u|^2|v|^2)+2\eta(\nabla u:\nabla v)^2\). Differentiating \(\int\omega_w(w_t,\partial_1 w)\) gives only \(\int u\cdot(v\times \partial_1 v)\): the other two terms vanish because \(v\cdot(v\times \partial_1 u)=0\) and \(u\cdot((-|v|^2u)\times \partial_1 u)=0\). This proves \eqref{eq:thin-belt-12}.
\end{proof}

This completes the variational framework: the action is smooth on a
complete manifold, its second and third chart derivatives are uniformly
controlled, and the two spatial dilations are represented by independent
variables.

\section{A planar comparison path}\label{app:planar-path}

We now construct a path from $e_1$ to a map of negative action, with
maximal action strictly below $8\pi$. This provides the upper bound
for the mountain-pass level needed in the compactness argument. The
bounds also hold for $J_{c,\eta}$, uniformly for speeds in a fixed
compact subset of $(0,1)$ and sufficiently small $\eta$. The
construction uses two oppositely oriented parts of a folded sphere map.

Let $\mathcal C$ be the class of smooth maps $u:\mathbb R^2\to S^2$
equal to $e_1$ outside a compact set. For the comparison arguments below,
write
\[
D(u)=\frac12\int|\nabla u|^2,\qquad
Z(u)=2V(u),\qquad
\mathsf p(u)=-P(u).
\]
Thus
\begin{equation}
E(u)=D(u)+\frac12Z(u),\qquad J_c(u)=E(u)-c\mathsf p(u).
\label{eq:planar-path-1}
\end{equation}

\begin{proposition}[A path below the two-bubble threshold]\label{prop:planar-path}
For every nonempty compact set \(K\subset(0,1)\), there are \(\eta_K>0\), \(R_K<\infty\), and a path \(\gamma:[0,1]\to\mathcal C\), continuous in every finite-order \(C^k\) norm and in every \(W^{1,p}\) norm for \(1\le p<\infty\) after subtracting \(e_1\), such that \begin{equation}
\gamma(0)=e_1,\qquad
\gamma(t,x)=e_1\quad\text{for all }|x|\ge R_K,\ t\in[0,1],
\label{eq:planar-path-2}
\end{equation} \begin{equation}
J_c(\gamma(1))<-1\quad(c\in K),\qquad
\sup_{\substack{c\in K\\t\in[0,1]}}J_c(\gamma(t))
\le8\pi-\eta_K.
\label{eq:planar-path-3}
\end{equation} For every $t\in[0,1]$, the map $\gamma(t)$ has degree zero when
regarded as a map from the compactified plane to $S^2$.
\end{proposition}

The proof is assembled at the end of the section from the constructions
and estimates below. We first smooth a folded sphere profile, separate
its oppositely oriented parts, and make the resulting maps equal to
$e_1$ outside a fixed ball. We then estimate their energy and momentum
and construct the path in two stages: first creating a profile at a
small spatial scale, then enlarging it until its action is negative,
while keeping the action below $8\pi$ throughout.

\subsection{The folded profile}

Identify the oriented spatial plane with \(\mathbb C\) by \(\zeta=x_1+i x_2\). For \(h=a+ib\in\mathbb C\), use the target stereographic parametrisation \begin{equation}
\mathcal S(h)=\frac{1}{1+|h|^2}(1-|h|^2,2a,2b).
\label{eq:planar-path-6}
\end{equation} It has \(\mathcal S(0)=e_1\), extends smoothly to \(\mathcal S(\infty)=-e_1\), and preserves orientation: at \(h=0\), \(\mathcal S\cdot(\mathcal S_a\times\mathcal S_b)=4>0\).

For \(w\in\mathbb C\setminus\{0\}\), write \(t=\log|w|\) and \(w/|w|=e^{i\theta}\). An initially nonsmooth folded profile is \begin{equation}
h_T^0(w)=T e^{-|t|}e^{i\theta}
=
\begin{cases}
Tw,& |w|\le1,\\
T/\bar w,& |w|\ge1.
\end{cases}
\label{eq:planar-path-7}
\end{equation} It extends to zero at \(w=0\) and \(w=\infty\). Its sphere map is continuous and has a derivative jump on \(|w|=1\). In the radial form \[
\mathcal S(h_T^0)=(\cos f,\sin f\cos\theta,\sin f\sin\theta),
\qquad f(t)=2\arctan(Te^{-|t|}),
\] the planar Dirichlet energy is \begin{equation}
\pi\int_{\mathbb R}(f_t^2+\sin^2 f)\,dt.
\label{eq:planar-path-8}
\end{equation} Indeed \(dx=r\,dr\,d\theta\), the radial derivative contributes \(f_r^2\), the angular derivative contributes \(r^{-2}\sin^2f\), and \(dt=dr/r\). Away from \(t=0\), \(f_t=\pm\sin f\). Substitution of \(a=Te^{-t}\) on \(t>0\) gives \begin{equation}
D(\mathcal S(h_T^0))
=4\pi\int_0^\infty\frac{4T^2e^{-2t}}{(1+T^2e^{-2t})^2}\,dt
=8\pi\frac{T^2}{1+T^2}.
\label{eq:planar-path-9}
\end{equation}

\subsection{Smoothing the fold}\label{appsrc-planar-path-smoothing-the-fold-with-an-explicit-energy-budget}

We smooth the fold while keeping its additional energy smaller than the gap just computed.

Fix \(\delta=1/100\). Choose a nonnegative even smooth function \(\rho\), supported in \((-\delta,\delta)\), with integral one, and put \[
\psi(t)=\int_{\mathbb R}|t-a|\rho(a)\,da.
\] Then \(\psi\) is smooth and even, \(\psi(t)=|t|\) when \(|t|\ge\delta\), \(|\psi'|\le1\), and \(|t|\le\psi(t)\le|t|+\delta\). On \(|t|\le\delta\) one has \(0\le\psi(t)\le2\delta\).

Replace \eqref{eq:planar-path-7} by \begin{equation}
h_T(w)=T e^{-\psi(\log|w|)}\,\frac{w}{|w|},
\qquad H_T(w)=\mathcal S(h_T(w)).
\label{eq:planar-path-10}
\end{equation} Set \(h_T(0)=h_T(\infty)=0\). Near zero \(h_T=Tw\); near infinity \(h_T=T/\bar w\). Thus \(H_T\) is a smooth map of the domain sphere to the target sphere. The family is smooth for each finite interval of parameters \(T\ge0\), including \(T=0\).

\begin{lemma}[Energy of the smoothed fold]\label{lem:path-fold}
With $\beta(T)=T/(1+T^2)$, the smooth family $H_T$ satisfies, for every $T\ge0$,
\begin{equation}
D(H_T)\le8\pi\frac{T^2}{1+T^2}+\pi\beta(T)^2.
\label{eq:planar-path-13}
\end{equation}
\end{lemma}

\begin{proof}
On \(|t|\le\delta\), \begin{equation}
\sin^2(2\arctan(Te^{-\psi(t)}))
=\frac{4T^2e^{-2\psi(t)}}{(1+T^2e^{-2\psi(t)})^2}
\le4e^{4\delta}\beta(T)^2.
\label{eq:planar-path-11}
\end{equation} For the inequality, divide by \(4\beta(T)^2\) when \(T>0\) and use \[
\frac{e^{-\psi}(1+T^2)}{1+T^2e^{-2\psi}}\le e^\psi\le e^{2\delta}.
\] At \(T=0\) both sides of \eqref{eq:planar-path-11} are zero. For the smoothed profile \(f=2\arctan(Te^{-\psi})\), \(f_t=-\psi'\sin f\). Its contribution to \eqref{eq:planar-path-8} on the interval of length \(2\delta\) is at most \begin{equation}
16\pi\delta e^{4\delta}\beta(T)^2\le\pi\beta(T)^2.
\label{eq:planar-path-12}
\end{equation} The last inequality follows from \(e^{4/100}<2\) and \(32/100<1\). Outside this interval it agrees with the folded profile. Discarding the nonnegative original energy on the interval and using \eqref{eq:planar-path-9}, we get \eqref{eq:planar-path-13}.
\end{proof}

\subsection{Separation of the two orientations}\label{appsrc-planar-path-put-the-two-orientations-at-different-transverse-positions}

Opposite orientations cancel in total degree, but their transverse positions determine the momentum. We therefore use a conformal change of variables to separate them, without changing the Dirichlet energy.

Use the conformal domain map \begin{equation}
M(\zeta)=\frac{\zeta+i}{\zeta-i}.
\label{eq:planar-path-14}
\end{equation} It maps \(-i\) to zero, \(i\) to infinity, and the spatial point at infinity to \(1\). Let \(\widetilde v_T(\zeta)=H_T(M(\zeta))\), using the smooth extensions at \(\zeta=\pm i\). Dirichlet energy is unchanged by this conformal reparametrisation. At every regular finite point the derivative of \(M\) is a scalar times an orthogonal orientation-preserving matrix, and the square of that scalar is its Jacobian. These factors cancel in the change of variables for \(\frac12\int|\nabla H_T|^2\). Removing disks about the two coordinate poles and taking their radii to zero proves the equality for the full integral.

Write \(a_T=T e^{-\psi(0)}\). The value at spatial infinity is the equatorial point \[
p_T=\mathcal S(a_T)
=(\cos(2\arctan a_T),\sin(2\arctan a_T),0).
\] Let \(\mathcal R_T\in SO(3)\) be rotation about \(e_3\) through angle \(-2\arctan a_T\), and set \begin{equation}
q_T=\mathcal R_T\widetilde v_T.
\label{eq:planar-path-15}
\end{equation} Then \(q_T(\infty)=e_1\). These target rotations preserve Dirichlet energy, the third component, and signed vorticity. In particular \eqref{eq:planar-path-13} applies to \(q_T\).

\begin{lemma}[Uniform decay at infinity]\label{lem:path-tail}
There are constants \(R_0,C<\infty\), independent of \(T\ge0\), such that for \(|\zeta|\ge R_0\), \begin{equation}
|q_T(\zeta)-e_1|\le\frac{C\beta(T)}{|\zeta|},
\qquad
|\nabla q_T(\zeta)|\le\frac{C\beta(T)}{|\zeta|^2}.
\label{eq:planar-path-16}
\end{equation}
\end{lemma}

\begin{proof}
On a sufficiently small convex neighbourhood \(U\) of \(w=1\), put \[
F(w)=e^{-\psi(\log|w|)}w/|w|.
\] This smooth complex function and its first derivatives are bounded, and \(0<b\le|F(w)|\le B<\infty\). The differential of \eqref{eq:planar-path-6}, applied to an increment \(\xi\), has norm \(2|\xi|/(1+|h|^2)\). Hence \[
|D_w\mathcal S(TF(w))|\le
\frac{2T|DF(w)|}{1+T^2|F(w)|^2}
\le C\frac{T}{1+T^2}=C\beta(T).
\] Integrate this bound along segments in \(U\) to get \(|H_T(w)-H_T(1)|\le C\beta(T)|w-1|\). Finally, \[
M(\zeta)-1=\frac{2i}{\zeta-i},\qquad
|DM(\zeta)|\le C|\zeta|^{-2}
\] for sufficiently large \(|\zeta|\). Orthogonality of \(\mathcal R_T\) yields \eqref{eq:planar-path-16}.
\end{proof}

\subsection{Compact truncation of the family}\label{appsrc-planar-path-make-the-family-exactly-constant-outside-one-fixed-ball}

The decay estimate allows the entire family to be made exactly equal to $e_1$ outside one fixed ball. The cutoff error is small enough to preserve the strict energy bound.

Take a smooth radial cutoff \(\chi_R\), equal to one on \(B_R\), zero outside \(B_{2R}\), with \(0\le\chi_R\le1\) and \(|\nabla\chi_R|\le C_0/R\). Choose \(R\ge R_0\) large enough that \eqref{eq:planar-path-16} makes \(|q_T-e_1|\le1/4\) on \(|x|\ge R\) for every \(T\); use \(\beta(T)\le1/2\). Define \begin{equation}
v_T(x)=
\frac{e_1+\chi_R(x)(q_T(x)-e_1)}
{|e_1+\chi_R(x)(q_T(x)-e_1)|}.
\label{eq:planar-path-17}
\end{equation} On \(B_R\) this is exactly \(q_T\), with denominator one. On the transition annulus the denominator is at least \(3/4\). Outside \(B_{2R}\) it is \(e_1\). It is a smooth sphere-valued family with fixed support for \(v_T-e_1\).

\begin{lemma}[Compact comparison family]\label{lem:path-family}
There is a smooth one-parameter family \(v_T\in\mathcal C\), \(T\ge0\), with one fixed compact support for \(v_T-e_1\), for which \begin{equation}
v_0=e_1,\qquad D(v_T)\le8\pi-\frac{6\pi}{1+T^2}<8\pi.
\label{eq:planar-path-4}
\end{equation}  Every $v_T$ has degree zero.
\end{lemma}

\begin{proof}
The differential of \(a\mapsto a/|a|\) has norm at most \(1/|a|\). On the annulus, \eqref{eq:planar-path-16} gives \[
|\nabla v_T|\le C_1\bigl(|\nabla q_T|+R^{-1}|q_T-e_1|\bigr)
\le C_2\beta(T)R^{-2}.
\] Integrating there bounds its new Dirichlet energy by \(C_3\beta(T)^2/R^2\). These constants are independent of \(R\ge R_0\) and \(T\). Increase the fixed \(R\) so that \(C_3/R^2\le\pi\). The energy inside \(B_R\) is bounded by the full energy in \eqref{eq:planar-path-13}, and outside \(B_{2R}\) it vanishes. Consequently \begin{equation}
D(v_T)\le8\pi\frac{T^2}{1+T^2}+2\pi\beta(T)^2
=8\pi-\frac{8\pi}{1+T^2}
+\frac{2\pi T^2}{(1+T^2)^2}
\le8\pi-\frac{6\pi}{1+T^2}.
\label{eq:planar-path-18}
\end{equation} This proves \eqref{eq:planar-path-4}. Every \(v_T\) has degree zero: \(t\mapsto v_{tT}\), \(0\le t\le1\), is a homotopy through smooth maps constant at spatial infinity, starting at \(e_1\).
\end{proof}

\subsection{The limiting potential and momentum}\label{appsrc-planar-path-the-limiting-potential-and-momentum}

We now compute the potential and momentum as $T\to\infty$. The energy approaches $8\pi$, but the potential tends to zero while the momentum remains positive; this makes spatial dilation effective in lowering the action.

\begin{lemma}[Limits of the comparison family]\label{lem:path-limits}
The family in \cref{lem:path-family} satisfies
\begin{equation}
D(v_T)\longrightarrow8\pi,\qquad
Z(v_T)\longrightarrow0,\qquad
\mathsf{p}(v_T)\longrightarrow8\pi
\quad(T\to\infty).
\label{eq:planar-path-5}
\end{equation}
\end{lemma}

\begin{proof}
At every fixed finite \(\zeta\ne\pm i\), \(h_T(M(\zeta))/T\ne0\), so \(\mathcal S(h_T(M(\zeta)))\to-e_1\). Also \(\mathcal R_T\) tends to rotation through angle \(-\pi\); hence \(q_T(\zeta)\to e_1\). At \(\zeta=\pm i\), the third component is zero for every \(T\). Formula \eqref{eq:planar-path-17} gives \(v_T\to e_1\) almost everywhere and \((v_T)_3\to0\) everywhere. Since \(|(v_T)_3|\le1\) and its support is in the fixed ball \(B_{2R}\), dominated convergence proves \begin{equation}
Z(v_T)\to0.
\label{eq:planar-path-19}
\end{equation}

Choose disjoint small fixed disks \(U_-,U_+\) about \(-i,i\), contained in \(B_R\), such that \(|M|<e^{-\delta}\) on \(U_-\) and \(|M|>e^\delta\) on \(U_+\). On the first disk the unrotated formula is exactly \(\mathcal S(TM(\zeta))\); on the second it is \(\mathcal S(T\overline{1/M(\zeta)})\).

On \(U_-\), \(M\) is holomorphic and one-to-one, with a simple zero. Its image contains \(B_b(0)\) and is contained in \(B_B(0)\), for some \(0<b<B<\infty\). Change variables to \(h=a+ib'=TM(\zeta)\). Formula \eqref{eq:planar-path-6} gives \[
\int_{|h|<Tb}\frac{4\,da\,db'}{(1+|h|^2)^2}
\le D(v_T;U_-)
\le\int_{|h|<TB}\frac{4\,da\,db'}{(1+|h|^2)^2}.
\] Both bounds tend to \(4\pi\) by radial integration. The map is conformal and orientation preserving there, so signed vorticity equals Dirichlet energy. On \(U_+\), the holomorphic local coordinate \(1/M\), followed by complex conjugation, gives energy tending to \(4\pi\) and signed vorticity tending to \(-4\pi\).

The upper bound \eqref{eq:planar-path-18} implies \(D(v_T)\to8\pi\) and \begin{equation}
D(v_T;\mathbb R^2\setminus(U_-\cup U_+))\to0.
\label{eq:planar-path-20}
\end{equation} Since \(|q(v_T)|\le\frac12|\nabla v_T|^2\), total vorticity variation outside these disks tends to zero. Taking the disks arbitrarily small proves \begin{equation}
q(v_T)\,dx\ \rightharpoonup\ 4\pi\delta_{(0,-1)}-4\pi\delta_{(0,1)}.
\label{eq:planar-path-21}
\end{equation} The measures have common compact support, so their first transverse moments converge. The sign convention \eqref{eq:planar-path-1} yields \begin{equation}
\mathsf{p}(v_T)\longrightarrow-4\pi(-1)+4\pi(1)=8\pi.
\label{eq:planar-path-22}
\end{equation} This proves \eqref{eq:planar-path-5}.
\end{proof}

\subsection{A path uniform in the speed}\label{appsrc-planar-path-one-common-path-for-every-speed-in-a-compact-interval}

We combine the comparison family with spatial dilation. First we create the profile at a small scale, then enlarge a fixed member until its action is negative, keeping the maximal action below $8\pi$.

\begin{proof}[Proof of \cref{prop:planar-path}]
Put \(c_-=\min K>0\), \(c_+=\max K<1\). Choose finite \(T_*>0\) so large that, with \[
d_*=D(v_{T_*}),\quad v_*=Z(v_{T_*}),\quad p_*=\mathsf{p}(v_{T_*}),
\] one has \(p_*>0\), \(v_*>0\), and \begin{equation}
\frac{c_-^2p_*^2}{2v_*}>d_*+1.
\label{eq:planar-path-23}
\end{equation} These choices follow from \eqref{eq:planar-path-5}. If a smooth map has identically zero third component, both spatial derivatives lie in the one-dimensional tangent line to the equator, so its vorticity and momentum vanish. Thus \(p_*>0\) guarantees \(v_*>0\).

For \(\sigma>0\), spatial dilation \(v^\sigma(x)=v(x/\sigma)\) gives \[
D(v^\sigma)=D(v),\quad Z(v^\sigma)=\sigma^2Z(v),\quad
\mathsf{p}(v^\sigma)=\sigma \mathsf{p}(v),
\] \begin{equation}
J_c(v^\sigma)=D(v)+\frac12\sigma^2Z(v)-c\sigma \mathsf{p}(v).
\label{eq:planar-path-24}
\end{equation} For momentum the factor \(\sigma\) is the product of the factor in \(x_2\), the two inverse derivative factors, and the two-dimensional Jacobian.

Set \(\sigma_*=c_-p_*/v_*>0\). Then \begin{equation}
J_c(v_{T_*}^{\sigma_*})\le J_{c_-}(v_{T_*}^{\sigma_*})
=d_*-\frac{c_-^2p_*^2}{2v_*}<-1
\quad(c\in K).
\label{eq:planar-path-25}
\end{equation} The family \(v_T\), \(0\le T\le T_*\), is smooth in \(T,x\) and has common compact support. Hence its potentials and momenta are continuous and bounded. Set \[
\Delta=\frac{6\pi}{1+T_*^2}>0,\quad
Z_{\max}=\max_{0\le T\le T_*}Z(v_T),\quad
\mathsf{p}_{\max}=\max_{0\le T\le T_*}|\mathsf{p}(v_T)|.
\] Choose \(0<\sigma_0<\sigma_*\) such that \[
\frac12\sigma_0^2Z_{\max}+c_+\sigma_0\mathsf{p}_{\max}<\Delta/2.
\] The creation path \(T\mapsto v_T^{\sigma_0}\) has, by \eqref{eq:planar-path-18} and \eqref{eq:planar-path-24}, \begin{equation}
J_c(v_T^{\sigma_0})\le8\pi-\Delta/2
\quad(0\le T\le T_*,\ c\in K).
\label{eq:planar-path-26}
\end{equation} Then dilate the fixed profile from \(\sigma_0\) to \(\sigma_*\). For every speed in \(K\), \begin{equation}
\frac{d}{d\sigma}J_c(v_{T_*}^{\sigma})
=\sigma v_*-cp_*
\le\sigma_*v_*-c_-p_*=0
\quad(\sigma_0\le\sigma\le\sigma_*).
\label{eq:planar-path-27}
\end{equation} This second part does not increase the action.

Parametrise the two parts continuously on \([0,1/2]\) and \([1/2,1]\). They agree at the join. Smoothness for finite \(T\), smooth dependence on positive dilations, and common support inside \(B_{2R\sigma_*}\) give the stated norm continuity. Take \(R_K=2R\sigma_*\), \(\eta_K=\Delta/2\). Equations \eqref{eq:planar-path-25}--\eqref{eq:planar-path-27} prove \eqref{eq:planar-path-2}--\eqref{eq:planar-path-3}. Degree zero is preserved by dilation. This completes the proof.
\end{proof}

\section{Mountain-pass geometry and small-energy coercivity}\label{app:planar-window}

We now prove that the action is bounded below by a positive multiple of the
energy when the energy is sufficiently small. This gives the positive barrier
needed for the mountain-pass argument. Together with the comparison path
constructed in Section~\ref{app:planar-path}, it yields uniform lower and upper
bounds for the regularised mountain-pass level.

Let \(\mathcal C_0\) consist of all smooth maps \(u:\mathbb R^2\to S^2\) equal to \(e_1\) outside a compact set and satisfying \[
\int_{\mathbb R^2}u\cdot(\partial_1 u\times \partial_2 u)=0.
\] This is the degree-zero condition for their smooth extensions to the compactified plane. No spatial symmetry or avoidance of target poles is assumed. 

We retain the notation of Section~\ref{app:planar-path} and write
\begin{equation}
Q(u)=\frac14\int|\nabla u|^4,\qquad
J_{c,\eta}(u)=J_c(u)+\eta Q(u)\quad(\eta\ge0).
\label{eq:planar-window-1}
\end{equation}

\subsection{The area primitive}

To estimate momentum without using the size of the support, we integrate the vorticity along the propagation direction. The resulting area primitive can be bounded by the Dirichlet energy and compared with each longitudinal slice.
Fix \(u\in\mathcal C_0\) and set \begin{equation}
p(x_2)=\int_{-\infty}^{x_2}\int_{\mathbb R}q(u)(x_1,t)\dd x_1\,dt.
\label{eq:planar-window-6}
\end{equation} 
\begin{lemma}[The area primitive]\label{lem:window-area}
The function $p$ is smooth and compactly supported, $\mathsf p(u)=\int p$, and $\int|p\prime|\le D(u)$. In particular $|p(x_2)|\le D(u)/2$.
\end{lemma}

\begin{proof}
Because \(u-e_1\) has compact support and \(\int q(u)=0\), \(p\) is a smooth compactly supported real function. Integration by parts gives \(\mathsf{p}(u)=\int p(x_2)\dd x_2\). Moreover, \begin{equation}
\int|p'(x_2)|\dd x_2
\le\int|q(u)|\le\frac12\int(|\partial_1 u|^2+|\partial_2 u|^2)=D(u).
\label{eq:planar-window-7}
\end{equation} Since \(p\) begins and ends at zero, for every \(x_2\) its total variation is at least \(2|p(x_2)|\): the variation on each of \((-\infty,x_2)\) and \((x_2,\infty)\) is at least \(|p(x_2)|\). Hence \begin{equation}
|p(x_2)|\le D(u)/2.
\label{eq:planar-window-8}
\end{equation}
\end{proof}

For the next two lemmas, assume $0<D(u)\le2$ and use the following notation.

Set \begin{equation}
A=D(u)/2\in(0,1],\qquad
h_*=1-D(u)/4=1-A/2\in[1/2,1),
\qquad H_*=\sqrt{1-h_*^2}\in(0,1).
\label{eq:planar-window-9}
\end{equation} For a fixed \(x_2\), let \(q(x_1)=u(x_1,x_2)\), \(h(x_1)=q_3(x_1)\), and \begin{equation}
\ell(x_2)=\int_{\mathbb R}|h(x_1)|\,|q'(x_1)|\dd x_1.
\label{eq:planar-window-10}
\end{equation} All these one-dimensional functions are smooth and constant outside a bounded interval; the interval can be chosen independently of \(x_2\).
\subsection{The energy cost of a loop reaching a fixed height}\label{appsrc-planar-window-a-longitudinal-loop-reaching-a-specified-height}

We first consider slices that reach height $H_*$. Their excursion away from the equator already forces enough weighted variation to control the area primitive.

\begin{lemma}[Height-crossing loops]\label{lem:window-height}
If $\sup_{x_1}|h(x_1)|\ge H_*$, then $\ell(x_2)\ge A$ and $|p(x_2)|\le h_*^{-1}\ell(x_2)$.
\end{lemma}

\begin{proof}
Suppose \(\sup_{x_1}|h(x_1)|\ge H_*\). The constraint \(|q|=1\) gives \[
|h'|\le\sqrt{1-h^2}\,|q'|
\] where \(|h|<1\), because the tangent projection of \(e_3\) has length \(\sqrt{1-h^2}\). Take a portion of the curve going from height zero to either height \(H_*\) or height \(-H_*\), and a subsequent portion returning to zero. The one-dimensional chain rule and total variation imply \begin{equation}
\ell(x_2)\ge2\int_0^{H_*}\frac{t}{\sqrt{1-t^2}}\,dt
=2(1-\sqrt{1-H_*^2})=2(1-h_*)=A.
\label{eq:planar-window-11}
\end{equation} One can stop the two portions at their first crossing of magnitude \(H_*\), so no behaviour at a pole is needed in this argument. If the actual curve oscillates, its weighted variation only increases. By \eqref{eq:planar-window-8}, \(|p(x_2)|\le A\), so in this case \begin{equation}
|p(x_2)|\le\ell(x_2)\le h_*^{-1}\ell(x_2).
\label{eq:planar-window-12}
\end{equation}
\end{proof}

\subsection{Loops in an equatorial belt}\label{appsrc-planar-window-a-loop-confined-to-a-height-belt-and-its-area-branch}

For a slice confined to the equatorial belt, a phase lift gives a different estimate. We must also keep track of the integer ambiguity between its phase integral and the area primitive.

\begin{lemma}[Loops confined to the belt]\label{lem:window-belt}
If $\sup_{x_1}|h(x_1)|<H_*$, then $|p(x_2)|\le h_*^{-1}\ell(x_2)$.
\end{lemma}

\begin{proof}
Now suppose \(\sup_{x_1}|h(x_1)|<H_*\). The loop \(q\) has a smooth real phase lift \[
q(x_1)=(\sqrt{1-h(x_1)^2}\cos\theta(x_1),
\sqrt{1-h(x_1)^2}\sin\theta(x_1),h(x_1)),
\] chosen so that \(\theta(-\infty)=0\). At the other constant end, \(\theta(+\infty)=2\pi n\) for some integer \(n\). Define \begin{equation}
b(x_2)=\int_{\mathbb R}h(x_1)\theta'(x_1)\dd x_1.
\label{eq:planar-window-13}
\end{equation} Since \(|q'|^2=h'^2/(1-h^2)+(1-h^2)\theta'^2\) and \(\sqrt{1-h^2}\ge h_*\), \begin{equation}
|b(x_2)|\le h_*^{-1}\ell(x_2).
\label{eq:planar-window-14}
\end{equation}

We next relate the area primitive to the phase integral. Orient the sphere by \(\omega_q(v_1,v_2)=q\cdot(v_1\times v_2)\). In the height/phase coordinates, \(\omega=d\theta\wedge dh\). The one-form \[
\alpha=(1-h)\,d\theta
\] extends smoothly across the north pole to all of \(S^2\setminus\{-e_3\}\), satisfies \(d\alpha=\omega\), and has \begin{equation}
\int_q\alpha=2\pi n-b(x_2).
\label{eq:planar-window-15}
\end{equation} This extension can be checked in the horizontal Cartesian coordinates near \(h=1\): \(1-h\) vanishes quadratically and cancels the angular-coordinate singularity.

The loop \(q\) avoids \(-e_3\) and admits a smooth filling disk in \(S^2\setminus\{-e_3\}\): in its stereographic identification with \(\mathbb R^2\), contract the loop to a point, taking the contraction constant on a neighbourhood of the disk centre. Stokes gives its oriented area as \eqref{eq:planar-window-15}. On the other hand \(u\) on the lower half-plane up to height \(x_2\), capped along its constant part at infinity, is a disk whose area is \(p(x_2)\) and whose boundary is the reversed loop \(q\): the induced orientation of the upper horizontal edge of an oriented \((x_1,x_2)\)-domain is decreasing \(x_1\).

Glue the two fillings along their common loop. Their sum of areas is an integer multiple of \(4\pi\). For completeness, the glued continuous sphere map is piecewise smooth; it can be smoothed in a collar of the seam with area-integral convergence. For smooth maps of the sphere, the integral of the pullback of \(\omega\) equals \(4\pi\) times the integer topological degree, by the change-of-variables degree formula for regular values. Smoothing preserves degree. It follows that \[
p(x_2)+2\pi n-b(x_2)\in4\pi\mathbb Z.
\] In particular, for some integer \(m=m(x_2)\), \begin{equation}
p(x_2)=b(x_2)+2\pi m.
\label{eq:planar-window-16}
\end{equation}

If \(m=0\), \eqref{eq:planar-window-14} gives \(|p(x_2)|\le h_*^{-1}\ell(x_2)\). If \(m\ne0\), then \eqref{eq:planar-window-8} and \eqref{eq:planar-window-16} imply \[
|b(x_2)|\ge2\pi-|p(x_2)|\ge2\pi-A.
\] By \eqref{eq:planar-window-14}, \(\ell(x_2)\ge h_*(2\pi-A)\). Since \(A\le1<\pi\), \(A\le2\pi-A\); hence \begin{equation}
|p(x_2)|\le A\le\frac{\ell(x_2)}{h_*}.
\label{eq:planar-window-17}
\end{equation} The belt and height-crossing cases exhaust all loops, including equality at height \(H_*\). Equations \eqref{eq:planar-window-12}, \eqref{eq:planar-window-14}, and \eqref{eq:planar-window-17} therefore give \(|p(x_2)|\le h_*^{-1}\ell(x_2)\) for every \(x_2\).
\end{proof}

\subsection{A momentum estimate independent of support}\label{appsrc-planar-window-the-support-independent-momentum-estimate}

Both types of slice now obey the same bound, which we integrate to obtain the momentum estimate.

\begin{proposition}[Small-energy momentum bound]\label{prop:window-momentum}
For every \(u\in\mathcal C_0\) with \(D(u)\le2\), \begin{equation}
\boxed{\quad
|\mathsf{p}(u)|\le\frac{1}{1-D(u)/4}
\left(\int u_3^2\right)^{1/2}
\left(\int|\partial_1 u|^2\right)^{1/2}
\le\frac{E(u)}{1-D(u)/4}.
\quad}
\label{eq:planar-window-2}
\end{equation}
\end{proposition}

\begin{proof}
If $D(u)=0$, then $u=e_1$ and the conclusion is immediate. Otherwise apply \cref{lem:window-height,lem:window-belt}. Integrate the loop bound and apply Cauchy--Schwarz: \begin{equation}
|\mathsf{p}(u)|\le\int|p(x_2)|\dd x_2
\le h_*^{-1}\int|u_3|\,|\partial_1 u|
\le h_*^{-1}\left(\int u_3^2\right)^{1/2}
\left(\int|\partial_1 u|^2\right)^{1/2}.
\label{eq:planar-window-18}
\end{equation} The arithmetic-geometric mean inequality gives \[
\left(\int u_3^2\right)^{1/2}
\left(\int|\partial_1 u|^2\right)^{1/2}
\le\frac12\int(u_3^2+|\partial_1 u|^2)\le E(u).
\] Substitute \(h_*=1-D(u)/4\) to obtain \eqref{eq:planar-window-2}. No size or location of the support appears.
\end{proof}

\subsection{Uniform small-energy positivity}\label{appsrc-planar-window-uniform-small-energy-positivity}

At small energy, the momentum estimate makes the positive energy term dominate any negative contribution to the action. This gives a positive barrier between the constant map and any map of negative action.

Fix a nonempty compact set \(K\subset(0,1)\), put \[
c_+=\max K,\qquad a_K=2(1-c_+),\qquad
b_K=\frac{2(1-c_+)^2}{1+c_+}>0,
\] and write $\mathcal N=N_1$, with $N_\eta$ defined in
\eqref{eq:planar-variation-norms}. 
\begin{lemma}[The positive action barrier]\label{lem:window-barrier}
For $u\in\mathcal C_0$, $c\in K$, $\eta\ge0$ and $E(u)\le a_K$, one has $J_{c,\eta}(u)\ge (1-c_+)(1+c_+)^{-1}E(u)$. Every $\mathcal N$-continuous path from $e_1$ to a map of negative action therefore has maximal action at least $b_K$.
\end{lemma}

\begin{proof}
If \(E(u)\le a_K=2(1-c_+)\), then \(D(u)\le E(u)\le2\), so \eqref{eq:planar-window-2} applies and \begin{equation}
J_c(u)\ge\left(1-\frac{c}{1-D(u)/4}\right)E(u)
\ge\left(1-\frac{c_+}{1-a_K/4}\right)E(u)
=\frac{1-c_+}{1+c_+}E(u)
\label{eq:planar-window-19}
\end{equation} for every \(c\in K\). The denominator is positive. Adding \(\eta Q(u)\ge0\) preserves the lower bound for every \(\eta\ge0\).

Any map with \(J_{c,\eta}<0\) must therefore have energy \(>a_K\). Along such a path, \(E\) is continuous for the norm $\mathcal N$. It starts at zero and ends above \(a_K\), so the intermediate value theorem supplies a point with \(E=a_K\). Equation \eqref{eq:planar-window-19} shows that the action there is at least \begin{equation}
\frac{1-c_+}{1+c_+}a_K=b_K.
\label{eq:planar-window-20}
\end{equation} Thus every such path has maximum at least \(b_K\), for all the parameters for which the endpoint is negative.
\end{proof}

\subsection{The upper bound for the regularised level}\label{appsrc-planar-window-a-common-upper-cap-including-a-small-quartic-regularization}

The comparison path supplies the complementary upper bound. Its quartic energy is bounded along the path, so a sufficiently small regularisation preserves both the negative endpoint and the strict gap below $8\pi$.

\begin{proposition}[Uniform mountain-pass bounds]\label{prop:planar-window}
There exist \(u_-\in\mathcal C_0\), \(\zeta_K>0\), and \(\bar\eta_K>0\), such that \(J_{c,\eta}(u_-)<0\) for all \(c\in K\) and \(0\le\eta\le\bar\eta_K\), and the following class is nonempty: \[
\begin{gathered}
\Gamma(u_-)=\{\gamma\in C([0,1],\mathcal C_0):\\
\gamma\text{ is }\mathcal N\text{-continuous},\\
\gamma(0)=e_1,\quad \gamma(1)=u_-,\\
\gamma(t)=e_1\text{ off one common compact set}\}.
\end{gathered}
\] Define \begin{equation}
m_\eta(c)=\inf_{\gamma\in\Gamma(u_-)}
\max_{t\in[0,1]}J_{c,\eta}(\gamma(t)).
\label{eq:planar-window-4}
\end{equation} Then \begin{equation}
\boxed{\qquad b_K\le m_\eta(c)\le8\pi-\zeta_K
\quad(c\in K,\ 0\le\eta\le\bar\eta_K).\qquad}
\label{eq:planar-window-5}
\end{equation} The support may depend on the path. The constants and endpoint can be chosen uniformly on \(K\).
\end{proposition}

\begin{proof}
Invoke \cref{prop:planar-path}. With exactly the convention \(\mathsf{p}(u)=-\int x_2 q(u)\) used here, it supplies a smooth common-support path \(\gamma_*\) from \(e_1\) to a degree-zero endpoint \(u_-\), a number \(\zeta_*>0\), and \begin{equation}
J_c(u_-)<-1,\qquad
\max_{t\in[0,1]}J_c(\gamma_*(t))\le8\pi-\zeta_*
\quad(c\in K).
\label{eq:planar-window-21}
\end{equation} Its norm continuity implies \(\gamma_*\in\Gamma(u_-)\). Its common support and finite-order smooth norm continuity also imply \[
Q_*:=\max_{t\in[0,1]}Q(\gamma_*(t))<\infty.
\] This number is positive, because a path with a negative-action endpoint cannot consist only of constant maps. Choose \[
\bar\eta_K=\min\{1,\zeta_*/(2Q_*),1/(2Q_*)\}>0,
\qquad \zeta_K=\zeta_*/2>0.
\] For \(0\le\eta\le\bar\eta_K\), \eqref{eq:planar-window-21} gives \(J_{c,\eta}(u_-)<-1/2\) and \begin{equation}
\max_tJ_{c,\eta}(\gamma_*(t))\le8\pi-\zeta_K.
\label{eq:planar-window-22}
\end{equation} The maxima in \eqref{eq:planar-window-4} exist: along any path in its definition, common support and \(\mathcal N\)-continuity imply continuity of \(E,Q,\mathsf{p}\). For \(\mathsf{p}\), use the bounded weight \(x_2\) on that support, strong \(L^2\) convergence of derivatives, and uniform convergence of maps in its quadratic vorticity integrand. Equations \eqref{eq:planar-window-20} and \eqref{eq:planar-window-22} establish \eqref{eq:planar-window-5} and complete the proof.
\end{proof}

\section{Second-order minimax selection on the original path}\label{app:path-selection}

We now consider a path along which the maximum of \(J_{c,\eta}\) is close to
the mountain-pass level. We seek a point on this path where the first derivative
of \(J_{c,\eta}\) is small and its second variation admits no two-dimensional
negative space at a fixed threshold. For this purpose, we prove an abstract
selection result for a general functional on a complete Banach manifold, with
the selected point belonging to the original path.

Let \(Y\) be a \(C^3\) real Banach manifold. Equip each tangent space \(T_yY\) with a norm \(\|\cdot\|_y\), varying continuously and locally uniformly equivalent to the model-space norm in every local trivialisation. Let \(d\) be the length metric of these norms, and assume that \(d\) is finite on the component under consideration, induces its manifold topology, and is complete. Let \(f:Y\to\mathbb R\) be \(C^3\). Fix endpoints \(a,b\in Y\), let \[
\Gamma=\{p\in C([0,1],Y):p(0)=a,\ p(1)=b\},\qquad
m=\inf_{p\in\Gamma}\max_{t\in[0,1]}f(p(t)),
\] and assume that \(\Gamma\ne\varnothing\), \(m\in\mathbb R\), and \[
g=m-\max\{f(a),f(b)\}>0.
\]

Fix constants \(R>0\) and \(M\ge1\). Consider any path \(p\in\Gamma\) for which, for every \(q\in p([0,1])\), there is a \(C^3\) chart \[
\Phi_q:B_{T_qY}(0,R)\longrightarrow Y
\] that is a diffeomorphism onto an open set, satisfies \(\Phi_q(0)=q\) and \(D\Phi_q(0)=\mathrm{id}_{T_qY}\), and has the following properties for every \(x\in B_{T_qY}(0,R)\) and \(v\in T_qY\): \[
\tfrac12\|v\|_q\le\|D\Phi_q(x)v\|_{\Phi_q(x)}\le2\|v\|_q,
\qquad
\|D^k(f\circ\Phi_q)(x)\|_{\mathrm{op},q}\le M\quad(k=2,3).
\] The operator norms in this display use \(\|\cdot\|_q\) in every argument. Define the chart-centre quadratic form \[
H_q[v,v]=D^2(f\circ\Phi_q)(0)[v,v]
\] and the dual norm \(\|df(q)\|_{q,*}=\sup_{\|v\|_q\le1}|df(q)[v]|\).

\begin{proposition}[Second-order selection on a given path]\label{prop:path-selection}
For \(0<\lambda\le1\), set \[
\begin{gathered}
\delta=\min\{1,R/100,\lambda/(100M)\},\qquad
\tau=\delta/100,\\
D=\lambda\tau^2/16,\qquad
\gamma(\lambda)=\min\{g/16,D/10\}.
\end{gathered}
\] Then for every \(0<\varepsilon\le\gamma(\lambda)\), if \[
\max_{t\in[0,1]}f(p(t))\le m+\varepsilon,
\] there exists \(t_*\in[0,1]\), with \(q=p(t_*)\), such that \[
|f(q)-m|\le\varepsilon,\qquad \|df(q)\|_{q,*}\le3\lambda,
\] and there is no two-dimensional linear subspace \(L\subset T_qY\) satisfying \[
H_q[v,v]\le-\lambda\|v\|_q^2\quad\text{for every }v\in L.
\] In particular the selected point belongs to the original path, and its quadratic form is evaluated at the centre of its original prescribed chart. The function \(\gamma\) depends only on \(R,M,g\), is positive for \(\lambda>0\), and tends to zero as \(\lambda\downarrow0\). No bound on the length of \(p\), its number of subdivisions, its location, or the dimension of the model space enters this function. Only charts centred on the original path are required.
\end{proposition}

In the following lemmas, fix $\lambda,\varepsilon,p$ as in \cref{prop:path-selection}. A chart ball without a subscript uses its centre's tangent norm.

\subsection{Metric balls and charts}\label{appsrc-path-selection-metric-access-to-each-prescribed-chart}

We begin by relating metric distance to chart coordinates, so that small deformations remain inside the charts on which derivative bounds are available.

\begin{lemma}[Metric access to a chart]\label{lem:selection-charts}
For each centre \(q\) the chart assumptions imply \begin{equation}
d(q,y)<R/2\quad\Longrightarrow\quad
y\in\Phi_q(B(0,R)),\qquad
\|\Phi_q^{-1}(y)\|_q\le2d(q,y).
\label{eq:path-selection-1}
\end{equation}
\end{lemma}

\begin{proof}
Take any piecewise \(C^1\) curve from \(q\) to \(y\) of length \(L<R/2\). Its lift, wherever defined in the chart, has speed at most twice its metric speed, by the lower derivative comparison. On the initial portion of the curve its lifted coordinates therefore have norm at most \(2L<R\). If the lift had a first finite endpoint before the curve ended, its finite remaining length would make its coordinates Cauchy. Completeness of the Banach space \(T_qY\) would give a limit of norm at most \(2L<R\). Continuity of the chart and the original curve would identify its image with the curve's limiting point, and the inverse chart would extend the lift past that endpoint. This is a contradiction. The lift is defined through the endpoint and has coordinate norm at most \(2L\). Taking curve lengths tending down to \(d(q,y)\) proves \eqref{eq:path-selection-1}. The same argument applies to each piece of a piecewise \(C^1\) curve. A straight segment of chart-coordinate length \(r\) has metric length at most \(2r\), by the upper derivative comparison.
\end{proof}

\subsection{Finite-dimensional deformations}\label{appsrc-path-selection-two-elementary-finite-dimensional-ingredients}

The deformation uses two-dimensional negative spaces. The next two lemmas give uniform coordinates on such a space and a continuous choice of directions along an interval.

\begin{lemma}[A basis in a normed plane]\label{lem:selection-plane}
For every two-dimensional normed real space \(L\) there is a linear isomorphism \(A:\mathbb R^2\to L\) such that \begin{equation}
|z|/\sqrt2\le\|Az\|\le\sqrt2|z|.
\label{eq:path-selection-2}
\end{equation}
\end{lemma}

\begin{proof}
Choose a nonzero alternating bilinear form on \(L\), and maximise the absolute value of its value on pairs in the closed unit ball. Finite-dimensional compactness gives a maximising pair \(e_1,e_2\). The maximum is positive, and scaling either vector shows that both have norm one. If \(v=a_1e_1+a_2e_2\) has norm at most one, replacing \(e_1\) or \(e_2\) by \(v\) in the determinant shows that \(|a_1|,|a_2|\le1\). Thus \(\max(|a_1|,|a_2|)\le\|v\|\le|a_1|+|a_2|\), which proves \eqref{eq:path-selection-2}.
\end{proof}

\begin{lemma}[Continuous circle directions]\label{lem:selection-circle}
If \(I\subset\mathbb R\) is a compact interval, \(\ell:I\to(\mathbb R^2)^*\) is continuous, and \(z>0\), there is a continuous \(e:I\to S^1\) such that \begin{equation}
\ell(t)e(t)\le z\quad(t\in I).
\label{eq:path-selection-3}
\end{equation}
\end{lemma}

\begin{proof}
Identify the dual of \(\mathbb R^2\) with \(\mathbb R^2\) using its Euclidean norm. On the closed set \(C=\{t:|\ell(t)|\ge z\}\), prescribe \(e(t)=-\ell(t)/|\ell(t)|\). On each complementary interval with two endpoints in \(C\), connect the endpoint values by a shortest circle arc; for antipodal endpoint values choose either semicircle. On a complementary interval with one endpoint in \(C\), extend constantly, and if \(C\) is empty choose a constant value on all of \(I\). This gives a continuous extension. For completeness, continuity at accumulation points of distinct complementary intervals follows because their lengths tend to zero, both endpoints tend to the accumulation point in \(C\), and their prescribed values converge to the same circle point; the diameters of the chosen shortest arcs then tend to zero. On the complement of \(C\), inequality \eqref{eq:path-selection-3} follows from \(|\ell|<z\), and on \(C\) it follows from the prescription.
\end{proof}

\subsection{Descent near points with small differential}\label{appsrc-path-selection-decrease-the-high-level-parameters-with-small-differential}

We first treat high-action points where the differential is small. If each has a two-dimensional negative space, the following deformation lowers their action while moving the whole path only a controlled distance.

Set \[
K=\{t\in[0,1]:f(p(t))\ge m-\varepsilon,
\ \|df(p(t))\|_{p(t),*}\le3\lambda\}.
\]

\begin{lemma}[Deformation along negative planes]\label{lem:selection-negative}
Suppose that for every $t\in K$ there is a two-dimensional space $L\subset T_{p(t)}Y$ on which
$H_{p(t)}[v,v]\le-\lambda\|v\|_{p(t)}^2$.
There is a path $\beta\in\Gamma$ such that
\[
d(\beta(t),p(t))\le9\tau,\qquad
f(\beta(t))\le f(p(t))+\varepsilon/100
\]
for every $t$, and $f(\beta(t))\le f(p(t))-D+\varepsilon/100$ for $t\in K$.
\end{lemma}

\begin{proof}
The set $K$ is compact. Indeed the dual norm is lower semicontinuous: in a local trivialisation it is the supremum, over nonzero fixed model vectors, of the continuous functions obtained by dividing the absolute differential evaluation by the tangent norm of that vector. Its sublevel set at \(3\lambda\) is therefore closed. Since \(\varepsilon\le g/16\), neither endpoint belongs to \(K\). By hypothesis, for each \(t\in K\) its centre \(p(t)\) has a two-dimensional plane satisfying the stated negative inequality.

If \(K\ne\varnothing\), choose an integer \(J\ge1\), and put \(h=1/J\), sufficiently small that the following hold: whenever \(|s-t|\le3h\), \[
d(p(s),p(t))<\delta/100;
\] and every closed mesh interval \([kh,(k+1)h]\) meeting \(K\), together with its one-cell enlargement \([(k-1)h,(k+2)h]\), lies in \((0,1)\). Uniform continuity of \(p\) and the positive distance of \(K\) from the endpoints ensure these choices. Retain exactly the mesh intervals meeting \(K\), and enumerate them as \(I_1,\ldots,I_N\). If \(I_i=[k_ih,(k_i+1)h]\), set \(J_i=((k_i-1)h,(k_i+2)h)\). Choose continuous \(h_i:[0,1]\to[0,1]\) equal to one on \(I_i\) and zero outside \(J_i\), linear on each of the two intervening cells. At any parameter at most three of the functions \(h_i\) are nonzero.

Choose \(t_i\in K\cap I_i\), write \(q_i=p(t_i)\), and choose a negative plane at \(q_i\). Choose its isomorphism \(A_i\) using \eqref{eq:path-selection-2}, with norm \(\|\cdot\|_{q_i}\). Write \(F_i=f\circ\Phi_{q_i}\). For every \(\|x\|_{q_i}<\delta\), the third derivative estimate gives \begin{equation}
D^2F_i(x)[v,v]\le-(\lambda-M\delta)\|v\|_{q_i}^2
\le-\tfrac\lambda2\|v\|_{q_i}^2
\quad(v\in\operatorname{range}A_i).
\label{eq:path-selection-4}
\end{equation}

Starting with \(p_0=p\), construct \(p_i\) from \(p_{i-1}\) in the chart at \(q_i\). We verify inductively that each earlier change has metric length at most \(3\tau h_j(t)\). Since at most three changes affect any fixed parameter, its total displacement from \(p(t)\) is at most \(9\tau\). For \(t\in\overline{J_i}\), this gives \[
d(p_{i-1}(t),q_i)<\delta/100+9\tau=\delta/10.
\] By \eqref{eq:path-selection-1} its coordinates \(x_i(t)=\Phi_{q_i}^{-1}(p_{i-1}(t))\) are defined and have norm less than \(\delta/5\). Define on this closed supporting interval the continuous functional \(\ell_i(t)=DF_i(x_i(t))A_i\). Choose a number \[
0<z_i\le\min\{\lambda\tau/16,\sqrt{\lambda\varepsilon/(200N)}\},
\] and use \eqref{eq:path-selection-3} to choose a continuous \(e_i(t)\in S^1\) with \(\ell_i(t)e_i(t)\le z_i\). Define \[
p_i(t)=\Phi_{q_i}(x_i(t)+\tau h_i(t)A_ie_i(t))
\quad(t\in\overline{J_i}),
\] and \(p_i(t)=p_{i-1}(t)\) elsewhere. Both definitions agree at the support boundary. The coordinate segment used has norm at most \(\delta/5+\sqrt2\tau<\delta\). Its metric length is at most \(2\sqrt2\tau h_i(t)<3\tau h_i(t)\), proving the induction assertion and well-definedness.

Taylor's formula along this segment, \eqref{eq:path-selection-2}, and \eqref{eq:path-selection-4} give \begin{equation}
f(p_i(t))-f(p_{i-1}(t))
\le z_i\tau h_i(t)-\frac\lambda8\tau^2h_i(t)^2
\le\frac{2z_i^2}{\lambda}\le\frac\varepsilon{100N}.
\label{eq:path-selection-5}
\end{equation} If \(t\in I_i\), the first expression is at most \(-\lambda\tau^2/16=-D\). Consequently the resulting path \(\beta=p_N\) satisfies \begin{equation}
d(\beta(t),p(t))\le9\tau,\qquad
f(\beta(t))\le f(p(t))+\varepsilon/100\quad(t\in[0,1]),
\label{eq:path-selection-6}
\end{equation} and \begin{equation}
f(\beta(t))\le f(p(t))-D+\varepsilon/100\quad(t\in K).
\label{eq:path-selection-7}
\end{equation} To see \eqref{eq:path-selection-7}, select any retained interval containing \(t\in K\). At its deformation step the drop is at least \(D\); the sum of every other possible increase is bounded by \(\varepsilon/100\) using \eqref{eq:path-selection-5}. This accounts for increases both before and after that step. All endpoints are fixed. If \(K=\varnothing\), set \(\beta=p\); \eqref{eq:path-selection-6} holds and \eqref{eq:path-selection-7} is vacuous, without introducing \(N\) or the numbers \(z_i\).
\end{proof}

\subsection{Construction of the descent field}\label{appsrc-path-selection-a-finite-locally-lipschitz-descent-field}

After this deformation, any point still near the minimax level comes from a part of the original path with large differential. We use those directions to construct a bounded descent field, whose flow exists globally by completeness.

\begin{lemma}[Descent away from the small-differential set]\label{lem:selection-flow}
Under the hypotheses of \cref{lem:selection-negative}, let $\beta$ be its deformed path and set
$H=\{t:f(\beta(t))\ge m-\varepsilon/2\}$ and $C=\beta(H)$.
If $H$ is nonempty, there is a continuous forward flow $\varphi_s$ fixing $a,b$, with metric speed at most one and nonincreasing $f$, such that
\[
\frac{d}{ds}f(\varphi_s(y))\le-\lambda
\]
whenever $d(\varphi_s(y),C)<\delta/16$ and $f(\varphi_s(y))\ge m-3\varepsilon/4$.
\end{lemma}

\begin{proof}
Define \[
H=\{t:f(\beta(t))\ge m-\varepsilon/2\},\qquad
C=\beta(H),\qquad Q=p(H).
\] These sets are compact, and $H$ is nonempty by hypothesis. From \eqref{eq:path-selection-7}, \(D\ge10\varepsilon\), and \(f(p(t))\le m+\varepsilon\), no \(t\in H\) belongs to \(K\). From \eqref{eq:path-selection-6}, every \(t\in H\) satisfies \[
f(p(t))\ge m-51\varepsilon/100>m-\varepsilon.
\] The definition of \(K\) therefore gives \(\|df(q)\|_{q,*}>3\lambda\) for every \(q\in Q\).

Choose a finite collection of centres \(r_1,\ldots,r_L\in Q\) whose metric balls of radius \(\delta/100\) cover \(Q\). At each centre choose a unit tangent vector \(v_j\in T_{r_j}Y\) such that \(df(r_j)[v_j]>5\lambda/2\). On the metric ball \(d(r_j,y)<\delta\), define \[
X_j(y)=-\tfrac12D\Phi_{r_j}(\Phi_{r_j}^{-1}(y))v_j.
\] This is well-defined by \eqref{eq:path-selection-1}. On the smaller ball \(d(r_j,y)<\delta/2\), its coordinate \(x\) has norm less than \(\delta\); hence \begin{equation}
\|X_j(y)\|_y\le1,\qquad
df(y)[X_j(y)]
\le-\tfrac12(5\lambda/2-M\delta)\le-\lambda.
\label{eq:path-selection-8}
\end{equation} Define nonnegative Lipschitz weights \[
w_j(y)=\max\{0,\delta/2-d(r_j,y)\},\qquad S(y)=\sum_{j=1}^Lw_j(y),
\] and define the tangent field \begin{equation}
V(y)=\frac{\sum_{j=1}^Lw_j(y)X_j(y)}{\max\{\delta/4,S(y)\}}.
\label{eq:path-selection-9}
\end{equation} In the numerator each term is defined to be zero outside its weight's support. This extension is locally Lipschitz: the distance functions are Lipschitz in the metric and locally Lipschitz in smooth coordinates; the vector field \(X_j\) is \(C^1\) on an open neighbourhood of the support, since it is defined on \(d(r_j,y)<\delta\). The denominator is bounded below by \(\delta/4\), and the sum is finite. Thus \(V\) is a globally defined, locally Lipschitz tangent field. Formula \eqref{eq:path-selection-8}, the triangle inequality, and positivity of the weights show that \(\|V(y)\|_y\le1\) and \(df(y)[V(y)]\le0\) everywhere.

If \(d(y,C)<\delta/16\), choose \(t\in H\) with \(d(y,\beta(t))<\delta/16\), and then a centre \(r_j\) with \(d(p(t),r_j)<\delta/100\). By \eqref{eq:path-selection-6}, \[
d(y,r_j)<\delta/16+9\tau+\delta/100<\delta/4.
\] It follows that \(w_j(y)>\delta/4\), so \(S(y)>\delta/4\). In \eqref{eq:path-selection-9} the denominator is then \(S(y)\), and \eqref{eq:path-selection-8} yields \begin{equation}
df(y)[V(y)]\le-\lambda
\quad\text{whenever }d(y,C)<\delta/16.
\label{eq:path-selection-10}
\end{equation} Let \(\theta:\mathbb R\to[0,1]\) be zero on \(( -\infty,m-\varepsilon]\), one on \([m-3\varepsilon/4,\infty)\), and affine on the intervening interval. The field \(Z(y)=\theta(f(y))V(y)\) is locally Lipschitz, has norm at most one, never increases \(f\), and is zero at both endpoints \(a,b\).

The forward flow \(\varphi_s\) of \(Z\) exists for all \(s\ge0\) and depends continuously on its initial point. Locally this is the Banach-space integral equation for a locally Lipschitz vector field: on a sufficiently small coordinate ball and sufficiently short time interval, iteration of its integral operator is a contraction; this gives the unique local solution and continuous dependence. For global existence, a maximal solution with finite terminal time has metric speed at most one, hence is Cauchy as time approaches that terminal time. Completeness gives a limit in \(Y\), whose local existence neighbourhood extends the solution, contradicting maximality. These observations also justify continuity of each fixed-time flow map.
\end{proof}

\subsection{The minimax contradiction}\label{appsrc-path-selection-contradiction-and-the-claimed-uniformity}

The two deformations now combine to give the contradiction: failure of the selection conclusion would produce an admissible path with maximum below the minimax level.

\begin{proof}[Proof of \cref{prop:path-selection}]
Suppose the conclusion fails. At every $t\in K$ there is then a two-dimensional negative space as in \cref{lem:selection-negative}; let $\beta$ be the path supplied by that lemma. If $H$ is empty, $\max f\circ\beta<m$ already contradicts the definition of $m$. Otherwise use the flow from \cref{lem:selection-flow}.

Set \(T=4\varepsilon/\lambda\). Since \(\varepsilon\le D/10=\lambda\delta^2/1{,}600{,}000\) and \(\delta\le1\), \[
T\le\delta^2/400{,}000<\delta/16.
\] For \(t\in H\), the trajectory starting at \(\beta(t)\in C\) stays in \(d(\cdot,C)<\delta/16\) up to time \(T\). While its value of \(f\) exceeds \(m-3\varepsilon/4\), equations \eqref{eq:path-selection-10} and the definition of \(\theta\) give derivative at most \(-\lambda\). Its initial value is at most \(m+101\varepsilon/100\), by \eqref{eq:path-selection-6}. If it did not reach \(m-3\varepsilon/4\) by time \(T\), integration would give a decrease of at least \(4\varepsilon\), contradicting that failure to reach the threshold. After reaching the threshold it cannot increase. Therefore \[
f(\varphi_T(\beta(t)))\le m-3\varepsilon/4\quad(t\in H).
\] For \(t\notin H\), nonincrease gives \(f(\varphi_T(\beta(t)))<m-\varepsilon/2\). The continuous path \(t\mapsto\varphi_T(\beta(t))\) fixes \(a,b\) and has maximum at most \(m-\varepsilon/2<m\), contradicting the definition of \(m\).

This proves the conclusion. The integer \(N\) appears only in the freely chosen tolerances \(z_i\). The constants \(\delta,\tau,D,\gamma\) are independent of \(N\) and of the path. Finally \(\gamma(\lambda)\le\lambda^3/(16{,}000{,}000{,}000M^2)\), by \(\delta\le\lambda/(100M)\), so \(\gamma(\lambda)\to0\).
\end{proof}

\section{Height-preserving minimax selection}\label{app:planar-selector}

We now apply \cref{prop:path-selection} to the regularised planar action.
First we deform the portion of a path lying close to the equator so
that it has small action. A selected point near the minimax level
must then have positive height. Adding the two coordinate dilations
gives the scalar errors needed in the energy estimate.

Let \(K\) be a nonempty compact subset of \((0,1)\), let \(c_+=\max K\), and choose \[
0<H<\sqrt{(1-c_+^2)/2},\qquad 0<H_0\le H/2.
\] Let \(\mathcal C_0\) be the degree-zero class introduced in
\cref{app:planar-window}. Use the functionals and strong norm in
\eqref{eq:planar-functionals} and \eqref{eq:planar-variation-norms}.
Put \[
\delta=1-\frac{c_+}{\sqrt{1-H^2}}>0,\qquad C_H=\frac{2}{(1-H^2)^2},
\] and let \(\vartheta:[0,\infty)\to[0,1]\) be zero on \([0,H_0]\), affine on \([H_0,2H_0]\), and one on \([2H_0,\infty)\).

\subsection{The normalised phase lift}\label{appsrc-planar-selector-the-normalized-real-lift-and-its-continuity}

For maps in a fixed equatorial belt, height and phase provide coordinates in which contraction to $e_1$ can be described explicitly. We first choose the phase consistently and check continuity of this contraction.

\begin{lemma}[Normalised phase lift and contraction]\label{lem:selector-lift}
For $u\in\mathcal C_0$, if \(h=u_3\) and \(\|h\|_\infty\le H\), there is a unique smooth real function \(\phi\), zero outside a sufficiently large disk, such that \[
  u=(\sqrt{1-h^2}\cos\phi,\sqrt{1-h^2}\sin\phi,h).
  \] For \(0\le t\le1\), define \[
  \mathcal C_tu=(\sqrt{1-t^2h^2}\cos(t\phi),
                \sqrt{1-t^2h^2}\sin(t\phi),th).
  \] The map \((t,u)\mapsto\mathcal C_tu\) is jointly continuous in \(N_\eta\) on this belt for each fixed \(\eta>0\).
\end{lemma}

\begin{proof}
Let \(u\) be a smooth compact perturbation with \(|h|\le H<1\), and put \(a=(u_1+iu_2)/\sqrt{1-h^2}\). This is a smooth map from the simply connected plane to the circle, equal to one outside a disk. Here is a direct construction of the lift. The real one-form \[
\alpha_j=\frac{u_1\partial_j u_2-u_2\partial_j u_1}{1-h^2}
\] is closed: locally one may choose an angular coordinate for \(a\), in which \(\alpha=d(\arg a)\), and this calculation covers the plane. Its integral along a loop vanishes, since every planar loop contracts to a point and integration of a closed smooth form is invariant under such a contraction, by the two-variable fundamental theorem of calculus. Hence integration from one point outside the disk defines a single-valued smooth function \(\phi\) with \(d\phi=\alpha\). The derivative of \(ae^{-i\phi}\) vanishes, so this function is identically one. On the connected exterior of the disk the lift is zero. Two lifts differ by a constant in \(2\pi\mathbb Z\), and the exterior normalisation makes that constant zero. This also proves compact support and uniqueness. A lift may be a nonzero multiple of \(2\pi\) in a bounded hole of the original support; this is why the support assertion is formulated using an enclosing disk.

Fix a base map \(u\) in the belt and let \(u_n\to u\) in \(N_\eta\), all in the belt, with normalised lifts \(\phi_n\). The normalised circle maps \(a_n\) converge uniformly to \(a\). For large \(n\), \(a_n\bar a\) lies in a fixed proper circle arc about one; let \(\psi_n\) be its smooth principal argument in that arc. Then \(\|\psi_n\|_\infty\le C_H\|u_n-u\|_\infty\to0\). Both \(\phi_n-\phi\) and \(\psi_n\) lift \(a_n\bar a\). Their difference is a constant integer multiple of \(2\pi\); it is zero outside a disk containing the supports of the two maps. Therefore \begin{equation}
\phi_n-\phi=\psi_n,\qquad \|\phi_n-\phi\|_\infty\longrightarrow0.
\label{eq:planar-selector-6}
\end{equation} The derivative formula for \(\alpha\) and denominators bounded below by \(1-H^2\) give, for \(p=2,4\), \begin{equation}
\|\nabla\phi_n-\nabla\phi\|_p
\le C(H)\left(\|\nabla u_n-\nabla u\|_p+
\|u_n-u\|_\infty\|\nabla u\|_p\right)\longrightarrow0.
\label{eq:planar-selector-7}
\end{equation} The heights converge uniformly, in \(L^2\), and in derivative \(L^2\cap L^4\).

If also \(t_n\to t\), \eqref{eq:planar-selector-6} implies \(t_n\phi_n\to t\phi\) uniformly. The factor \((t_n-t)\phi\) tends uniformly to zero because this fixed \(\phi\) is bounded. The derivatives of \(t_n\phi_n\) converge in \(L^2\cap L^4\), and the heights \(t_nh_n\) converge in all the stated height norms. Apply the chain rule to \[
F(b,\theta)=(\sqrt{1-b^2}\cos\theta,\sqrt{1-b^2}\sin\theta,b).
\] On \(|b|\le H\), its first derivatives are bounded and uniformly continuous in \(b,\theta\), with the angle taken periodically. Products of a uniformly convergent bounded coefficient and an \(L^p\)-convergent derivative converge in \(L^p\), \(p=2,4\). The third component is exactly \(t_nh_n\), which converges in \(L^2\). These observations prove joint \(N_\eta\) continuity of \(\mathcal C\). The same proof works on any slightly wider belt still strictly inside the poles, so it also gives continuity at the gluing height \(2H_0\).
\end{proof}

\subsection{Momentum and the height deformation}\label{appsrc-planar-selector-momentum-derivative-signs-and-coercivity}

We next check how the energy and momentum change along the contraction, to show that it lowers the action.

\begin{lemma}[Action along the height contraction]\label{lem:selector-action}
For $c\in K$, $0<\eta\le1$ and the maps in \cref{lem:selector-lift},  \[
  J_{c,\eta}(u)\ge\delta E_\eta(u),\qquad
  0\le J_{c,\eta}(\mathcal C_tu)\le C_Ht^2E_\eta(u),
  \] and \(t\mapsto J_{c,\eta}(\mathcal C_tu)\) is nondecreasing.
\end{lemma}

\begin{proof}
Set \(a=|\nabla h|^2\), \(b=|\nabla\phi|^2\), and \(k=h^2\), all pointwise functions. Direct differentiation gives \begin{equation}
|\nabla\mathcal C_tu|^2
=\frac{t^2a}{1-t^2k}+t^2(1-t^2k)b,
\label{eq:planar-selector-8}
\end{equation} and \[
u\cdot(\partial_1 u\times \partial_2 u)=\partial_1 \phi \partial_2 h-\partial_1 h\partial_2 \phi
=\partial_2 (h\partial_1 \phi)-\partial_1 (h\partial_2 \phi).
\] All functions in the last formula have compact support. Integration against \(x_2\) yields \begin{equation}
P(u)=-\int h\partial_1 \phi,\qquad P(\mathcal C_tu)=-t^2\int h\partial_1 \phi.
\label{eq:planar-selector-9}
\end{equation} The same density formula without the \(x_2\) factor has integral zero. Thus every \(\mathcal C_tu\) belongs to \(\mathcal C_0\).

For \(0<t\le1\), differentiating \eqref{eq:planar-selector-8} gives \begin{equation}
\frac{1}{2t}\frac{d}{dt}|\nabla\mathcal C_tu|^2
=\frac{a}{(1-t^2k)^2}+(1-2t^2k)b=:L_t\ge0,
\label{eq:planar-selector-10}
\end{equation} because \(H^2<1/2\). Differentiating the action, including its quartic coefficient, gives precisely \begin{equation}
\frac{d}{dt}J_{c,\eta}(\mathcal C_tu)
=t\int\left[\frac{a}{(1-t^2k)^2}+(1-2t^2k)b+h^2-2ch\partial_1 \phi\right]
+\eta t\int |\nabla\mathcal C_tu|^2L_t.
\label{eq:planar-selector-11}
\end{equation} All differentiations under these compact integrals are justified by their smooth integrands and denominators bounded below by \(1-H^2\).

Put \(\kappa=1-2H^2>c_+^2\) and \(\rho=1-c_+/\sqrt\kappa>0\). The inequality \(2c|h\partial_1 \phi|\le(c_+/\sqrt\kappa)(h^2+\kappa(\partial_1 \phi)^2)\) implies \[
h^2+(1-2t^2h^2)|\nabla\phi|^2-2ch\partial_1 \phi
\ge\rho(h^2+\kappa|\nabla\phi|^2).
\] Thus \eqref{eq:planar-selector-11} is nonnegative. Continuity at \(t=0\), where the action is zero, proves nonnegativity and monotonicity on the closed interval.

For coercivity at \(t=1\), set \(r_0=1-H^2\). Pointwise \(1-h^2\ge r_0\), so \eqref{eq:planar-selector-9} and Cauchy--Schwarz give \[
|cP(u)|\le\frac{c_+}{\sqrt{r_0}}\,
\frac12\int\bigl(h^2+(1-h^2)(\partial_1 \phi)^2\bigr)
\le(1-\delta)E(u).
\] Consequently \begin{equation}
J_{c,\eta}(u)\ge\delta E(u)+\eta Q(u)\ge\delta E_\eta(u).
\label{eq:planar-selector-12}
\end{equation} For the upper estimate, \eqref{eq:planar-selector-8} implies \(|\nabla\mathcal C_tu|^2\le t^2r_0^{-1}|\nabla u|^2\), because the height-gradient term is at most its \(t^2\)-multiple and the phase term is at most \(t^2r_0^{-1}\) times its original value. Therefore \[
E(\mathcal C_tu)\le t^2r_0^{-1}E(u),\quad
\eta Q(\mathcal C_tu)\le t^4r_0^{-2}\eta Q(u),\quad
|cP(\mathcal C_tu)|\le t^2(c_+/\sqrt{r_0})E(u).
\] Since \(t\le1\), \(r_0\le1\), and \(c_+<1\), their sum is at most \(2r_0^{-2}t^2E_\eta(u)=C_Ht^2E_\eta(u)\). The constants are independent of $c$ and $\eta$.
\end{proof}

\subsection{The deformation and its height bound}\label{appsrc-planar-selector-the-homotopy-and-the-height-witness}

The contraction can now be applied only to maps of small height, leaving the rest of each path fixed. For paths with a fixed upper action bound, points whose deformed action stays away from zero must then have a definite positive height.

\begin{proposition}[Height-preserving deformation]\label{prop:selector-deformation}
For \(0\le\sigma\le1\), define \[
  \mathcal D_\sigma u=
  \begin{cases}
  \mathcal C_{\,1-\sigma+\sigma\vartheta(\|u_3\|_\infty)}u,
                                     &\|u_3\|_\infty<2H_0,\\
  u,                                    &\|u_3\|_\infty\ge2H_0.
  \end{cases}
  \] This is a jointly strong-continuous homotopy \([0,1]\times\mathcal C_0\to\mathcal C_0\). It starts at the identity, fixes \(e_1\), fixes every map of height at least \(2H_0\), and fixes every map of negative \(J_{c,\eta}\). For every \(u\), \(J_{c,\eta}(\mathcal D_\sigma u)\) is nonincreasing in \(\sigma\). If all maps in a path equal \(e_1\) outside a common closed disk, all maps of its deformed homotopy equal \(e_1\) outside that same disk. For an arbitrary original compact support one may first enclose it in a disk.

Let \(B>0\), \(0<a\le B\), and let \(\gamma:[0,1]\to\mathcal C_0\) be a strong-continuous common-support path with \(\max_r J_{c,\eta}(\gamma(r))\le B\). Then, for every \(r\) with \(J_{c,\eta}(\mathcal D_1\gamma(r))\ge a\), \begin{equation}
\|(\mathcal D_1\gamma(r))_3\|_\infty
\ge H_0\sqrt{\frac{a\delta}{C_HB}}.
\label{eq:planar-selector-1}
\end{equation} The lower bound is independent of the path, its support, \(c\in K\), and \(\eta\in(0,1]\).
\end{proposition}

\begin{proof}
The map \(u\mapsto\|u_3\|_\infty\) is continuous in the strong topology. On its open sublevel set \(r<2H_0\), \cref{lem:selector-lift} therefore proves joint continuity of \(\mathcal D_\sigma\). If \(u_n\to u\), \(r(u)=2H_0\), and \(r(u_n)<2H_0\), the contraction parameters tend to one for every convergent sequence of \(\sigma\)'s. The lifting proof applies in a neighbourhood of this base map because \(2H_0\le H<1\). Its joint continuity gives \(\mathcal D_{\sigma_n}u_n\to u\). On \(r>2H_0\) the homotopy is the identity. At \(r=H_0\), the same joint-continuity argument applies with the continuous function \(\vartheta\), including \(t=0\) when \(\sigma=1\). This proves continuity everywhere. Each changed map and each intermediate map is smooth in the spatial variables and equals \(e_1\) outside the original enclosing disk. Formula \eqref{eq:planar-selector-9} proves zero total area for changed maps; unchanged maps already have it.

The parameter \(1-\sigma+\sigma\vartheta(r)\) is nonincreasing in \(\sigma\), so \cref{lem:selector-action} proves action nonincrease. A map with negative action cannot satisfy \(r\le H\), by \eqref{eq:planar-selector-12}. In particular it has \(r>2H_0\) and is fixed. The other fixed-point and initial-time assertions follow from the definitions.

For \eqref{eq:planar-selector-1}, fix \(u=\gamma(r)\) with deformed action at least \(a>0\). If \(r(u)\le H_0\), then \(\mathcal D_1u=e_1\), contradicting this positivity. If \(r(u)\ge2H_0\), the deformed height is at least \(2H_0\), which is greater than the right side of \eqref{eq:planar-selector-1}, since \(a\le B\), \(0<\delta<1\), and \(C_H\ge2\). It remains that \(H_0<r(u)<2H_0\). Put \(t=\vartheta(r(u))\). The original action is at most \(B\) and the map is in the belt, so \eqref{eq:planar-selector-12} gives \(E_\eta(u)\le B/\delta\). \Cref{lem:selector-action} then yields \[
a\le J_{c,\eta}(\mathcal C_tu)\le C_Ht^2 B/\delta.
\] Hence \(t\ge\sqrt{a\delta/(C_HB)}\) and \(\|(\mathcal C_tu)_3\|_\infty=tr(u)\ge H_0\sqrt{a\delta/(C_HB)}\). This proves \eqref{eq:planar-selector-1}.

The action estimates depend on the phase only through its derivatives.
\end{proof}

\subsection{Selection with two scaling parameters}

We can now apply the abstract selection theorem to the deformed path. Working on the augmented space controls both map variations and the two dilation derivatives, while the height estimate keeps the selected map away from the constant state.

Take the common compact negative endpoint \(u_-\), the number \(\bar\eta_K>0\), the strict upper gap \(\zeta_K>0\), and the widths from \cref{prop:planar-window}, reducing \(\bar\eta_K\) to at most one. The convention $\mathsf p=-P$ in that proposition gives exactly $J_{c,\eta}$. Let \(\Gamma(u_-)\) consist of the strong-continuous common-support paths in \(\mathcal C_0\) from \(e_1\) to \(u_-\), and set \[
  m_\eta(c)=\inf_{\gamma\in\Gamma(u_-)}\max_r J_{c,\eta}(\gamma(r)),
  \quad b_K=\frac{2(1-c_+)^2}{1+c_+}.
  \] Thus \(b_K\le m_\eta(c)\le8\pi-\zeta_K\) for \(c\in K\), \(0<\eta\le\bar\eta_K\). Define the positive constant \[
  h_K=H_0\sqrt{\frac{b_K\delta}{2C_H(8\pi+1)}}.
  \]

\begin{proposition}[Selection with positive height]\label{prop:planar-selector}
There is a function \(\epsilon_K:(0,1]\to(0,\infty)\), independent of \(c,\eta\), with \(\epsilon_K(\lambda)\to0\) as \(\lambda\to0\), such that the following holds. For every \(c\in K\), \(0<\eta\le\bar\eta_K\), \(0<\lambda\le1\), and \(\gamma\in\Gamma(u_-)\), if \(0<\epsilon\le\epsilon_K(\lambda)\) and \(\max_rJ_{c,\eta}(\gamma(r))\le m_\eta(c)+\epsilon\), there is a point \(u=\mathcal D_1\gamma(r_*)\) on the deformed path satisfying \begin{equation}
|J_{c,\eta}(u)-m_\eta(c)|\le\epsilon,\qquad
\|u_3\|_\infty\ge h_K.
\label{eq:planar-selector-2}
\end{equation} It is a smooth compact perturbation of \(e_1\). With the residual, Hessian and index from
\cref{sec:planar-framework}, we have \begin{equation}
|R_{c,\eta,u}[v]|\le3\lambda B_{\eta,u}(v),\qquad
i_\lambda(u;c,\eta)\le1.
\label{eq:planar-selector-3}
\end{equation} With $D_1,D_2,W_\eta$ from \eqref{eq:planar-scaling-data}, the same point satisfies \begin{equation}
|D_1|+|D_2|\le6\lambda W_\eta.
\label{eq:planar-selector-4}
\end{equation} These conclusions hold simultaneously and do not assume a common energy bound.
\end{proposition}

\begin{proof}
Let $\mathcal M_\eta^0$ and
$\mathcal Y_\eta=\R^2\times\mathcal M_\eta^0$ be as in
\cref{prop:thin-belt-augmented}, with functional $f_{c,\eta}$ and complete metric
\eqref{eq:thin-belt-7}. We first check equality of the relevant widths.
Every common-support smooth path from $e_1$ lies in $\mathcal M_\eta^0$.
Conversely, a continuous path in this component can be approximated,
with endpoints fixed, by common-support smooth paths: subdivide its
compact parameter interval into spherical-chart neighbourhoods,
approximate the finitely many subdivision maps strongly by smooth compact
maps, and join consecutive approximations by the short normalised
segments. All segments stay in the prescribed neighbourhoods, have one
common compact support, and are uniformly close to the original path.
The momentum and action are continuous in these charts, so the maxima
converge. Total area is locally constant by \cref{lem:thin-belt-scaling}; these
paths therefore have degree zero. Thus the completed-map width is
$m_\eta(c)$.

If $(\alpha(r),\beta(r),u(r))$ is a path in $\mathcal Y_\eta$ joining
$(0,0,e_1)$ to $(0,0,u_-)$, then
$\mathcal T_{\alpha(r),\beta(r)}u(r)$ is a continuous completed-map path
with the same action at each parameter, by
\eqref{eq:thin-belt-10}. Conversely, every completed-map path lifts
with both scalar coordinates zero. Hence the augmented width is also
$m_\eta(c)$.

Deform the given near-optimal compact path by $\mathcal D_1$ and lift it
as $r\mapsto(0,0,\mathcal D_1\gamma(r))$. The uniform charts, metric
comparisons and second/third derivative bounds in
\cref{prop:thin-belt-augmented} verify every hypothesis of
\cref{prop:path-selection}. For sufficiently small
$\epsilon\le\epsilon_K(\lambda)$, that proposition selects an original
path point $(0,0,u)$ with
$|J(u)-m_\eta(c)|\le\epsilon$, augmented dual residual at most
$3\lambda$, and threshold index at most one. Its action is at least
$b_K/2$, while the original path maximum is at most $8\pi+1$.
Equation \eqref{eq:planar-selector-1} therefore gives the stated height
$h_K$. The permissible tolerance depends only on the uniform chart
constants and $b_K$, so is uniform for $c\in K$ and small $\eta$.

Restriction to the tangent directions $(0,0,v)$ gives the physical
residual and Hessian in \eqref{eq:planar-variation-norms}, with exactly
the norm $B_{\eta,u}(v)$. Every two-dimensional negative threshold space
of these fields would be a forbidden augmented space. This proves
\eqref{eq:planar-selector-3}. The two scalar coordinate derivatives are
$D_1,D_2$, and each scalar unit vector has norm $W_\eta(u)$.
Consequently \eqref{eq:planar-selector-4} follows. All properties hold
at the same smooth point of the deformed compact path.
\end{proof}

\section{Classification of stable one-dimensional profiles}\label{app:signed-estimate}

The local limits in the energy estimate can be one-dimensional
without having finite energy on the line. We classify all stable
entire profiles in this larger class, including the limiting cases
of the stability interval. The classification also gives a sharp pointwise comparison between
longitudinal gradient energy and anisotropy energy.

Fix \(0<c<1\). For an entire map \(q\in H^1_{\mathrm{loc}}(\mathbb R;S^2)\), let primes denote derivatives in its real variable \(x_1\), and write \(q_3=q\cdot e_3\). Suppose, in the sense of distributions, that \begin{equation}
q''+(|q'|^2+q_3^2)q-q_3e_3+cq\times q'=0.
\label{eq:signed-estimate-1}
\end{equation} Call it stable if \begin{equation}
Q_{q,c}(\xi):=\int_{\mathbb R}\bigl(|\xi'|^2-(|q'|^2+q_3^2)|\xi|^2+\xi_3^2+cq\cdot(\xi\times\xi')\bigr)\dd x_1\geq0
\label{eq:signed-estimate-2}
\end{equation} for every \(\xi\in C_c^\infty(\mathbb R;\mathbb R^3)\) with \(\xi(x_1)\cdot q(x_1)=0\) for every \(x_1\). No finite-energy assumption is imposed.

All integrals below are with respect to $x_1$, unless another variable is displayed.
\subsection{Regularity, conserved quantities, and closure of stability}

Before using stability, we establish regularity and the conserved quantities of the profile equation. These also give compactness under translations and allow stability to pass to local limits.

\begin{lemma}[Regularity and conserved quantities]\label{lem:profile-conserved}
Every distributional solution of \eqref{eq:signed-estimate-1} is smooth and has bounded derivatives of every positive order. Writing $h=q_3$, the quantities
\begin{equation}
H=\tfrac12(|q'|^2-h^2),\qquad j=e_3\cdot(q\times q')-ch
\label{eq:signed-estimate-6}
\end{equation}
are constant. Stability extends to compact $H^1$ tangent fields and is preserved under smooth local convergence of profiles and their speeds.
\end{lemma}

\begin{proof}
The right-hand side of \[
q''=-(|q'|^2+q_3^2)q+q_3e_3-cq\times q'
\] belongs to \(L^1_{\mathrm{loc}}(\mathbb R)\). Integrating this distributional identity shows that \(q'\) has a locally absolutely continuous representative. In one dimension \(q\) and then \(q'\) are continuous; the displayed right-hand side is therefore continuous. Thus \(q\in C^2\). Repeated differentiation of the equation gives \(q\in C^\infty\).

For the quantities in \eqref{eq:signed-estimate-6}, the identities \(q\cdot q'=0\) and \eqref{eq:signed-estimate-1}, dotted with \(q'\), give \(H'=0\). Crossing \eqref{eq:signed-estimate-1} on the left by \(q\) gives \[
q\times q''=h(q\times e_3)+cq',
\] because \(q\times(q\times q')=-q'\). Dotting this last identity with \(e_3\) gives \(j'=0\). In particular \(|q'|^2=2H+h^2\leq 2|H|+1\). The equation bounds \(q''\) globally, and repeated differentiation bounds every higher derivative globally. The bounds depend on this particular solution and its speed, which suffices for the translation arguments used here.

Nonnegativity of \(Q_{q,c}\) extends to compactly supported \(H^1\) tangent fields. To see this, approximate such a field in \(H^1\) by smooth compact ambient fields and project each approximant by \(I-q\otimes q\); the resulting smooth tangent fields converge in \(H^1\). On their common compact support the coefficients of \eqref{eq:signed-estimate-2} are bounded, so the quadratic forms converge.

If smooth stable solutions \(q_n\) at speeds \(c_n\) converge smoothly on compact intervals to a solution \(q_\infty\), with \(c_n\to c_\infty\), then the limit is stable. For a smooth compact field \(\xi\) tangent to \(q_\infty\), use \(\xi_n=\xi-(\xi\cdot q_n)q_n\). These are compact tangent fields for \(q_n\), converge in \(C^1\), and satisfy \(Q_{q_n,c_n}(\xi_n)\to Q_{q_\infty,c_\infty}(\xi)\). Translations in \(x_1\) and constant rotations about \(e_3\) preserve the equation and stability.
\end{proof}

\subsection{A zero of a bounded Jacobi field}

Rotational symmetry produces a Jacobi field along each profile. We show that stability severely restricts its zeros: a bounded Jacobi field vanishing at one point must vanish everywhere.

Let \(B\) be the symmetric polarisation of \(Q_{q,c}\); thus, for tangent fields \(u,v\), \begin{equation}
B(u,v)=\int \left[u'\cdot v'-(|q'|^2+h^2)u\cdot v+u_3v_3
+\tfrac c2 q\cdot(u\times v'+v\times u')\right].
\label{eq:signed-estimate-7}
\end{equation} A smooth tangent field \(w\) is called Jacobi here if \(B(w,\psi)=0\) for every smooth compact tangent field \(\psi\); all the integrals in this definition have compact support. 
\begin{lemma}[Zeros of bounded Jacobi fields]\label{lem:profile-jacobi}
For a stable profile, any bounded Jacobi field that vanishes at one point is identically zero. The field $e_3\times q$ is a bounded Jacobi field. The same zero principle holds for a nonnegative scalar form $\int(a\prime{}^2+Va^2)$ with smooth potential $V$.
\end{lemma}

\begin{proof}
Translate the zero of $w$ to the origin. For \(R>1\), choose a smooth function \(\chi_R\) on \([0,\infty)\), equal to one on \([0,R]\), zero on \([2R,\infty)\), and with \(|\chi_R'|\leq C/R\), for an absolute constant \(C\). Define \(w_R(x_1)=\chi_R(x_1)w(x_1)\) for \(x_1\geq0\), and \(w_R(x_1)=0\) for \(x_1<0\). Since \(w(0)=0\), this is a compact \(H^1\) tangent field. Expanding \eqref{eq:signed-estimate-7} on the positive half-line gives \[
Q_{q,c}(w_R)=B_{[0,\infty)}(w,\chi_R^2w)+\int_0^\infty |\chi_R'|^2|w|^2.
\] The first term is zero: integration by parts uses the Jacobi equation in the interior, and its boundary terms vanish at zero since \(w(0)=0\), and at infinity by the cutoff. The first-order terms introduce no extra cutoff square since \(w\times w=0\). Consequently \begin{equation}
0\leq Q_{q,c}(w_R)\leq C^2\|w\|_\infty^2/R.
\label{eq:signed-estimate-8}
\end{equation} Positivity of the quadratic form on compact \(H^1\) tangent fields implies \[
|B(w_R,\psi)|^2\leq Q_{q,c}(w_R)Q_{q,c}(\psi),
\] by applying nonnegativity to \(w_R+t\psi\) for real \(t\). For a fixed compact smooth tangent field \(\psi\), and all sufficiently large \(R\), integration by parts on the positive half-line instead gives \[
B(w_R,\psi)=-w'(0)\cdot\psi(0).
\] The possible first-order boundary term is a multiple of \(q(0)\cdot(w(0)\times\psi(0))\), and is zero. Equation \eqref{eq:signed-estimate-8} therefore gives \(w'(0)\cdot\psi(0)=0\). Differentiating \(q\cdot w=0\) at zero shows that \(w'(0)\) is tangent, and \(\psi(0)\) can be any tangent vector. Thus \(w'(0)=0\). In any smooth tangent frame the Jacobi equation is a linear second-order ordinary differential equation with identity principal matrix. Uniqueness with zero value and derivative gives \(w\equiv0\). The same proof applies to the scalar form \(\int(a'^2+V(x_1)a^2)\) with smooth real \(V\).

For completeness, the Jacobi field used first is \(w=e_3\times q\). Define \[
\mathcal R_c(q)=q''+(|q'|^2+h^2)q-he_3+cq\times q'.
\] For a smooth tangent field \(w\), its linearisation is \[
D\mathcal R_c(q)[w]=w''+(|q'|^2+h^2)w+(2q'\cdot w'+2hw_3)q-w_3e_3
+c(w\times q'+q\times w').
\] For compact tangent \(\psi\), integration by parts shows \begin{equation}
B(w,\psi)=-\int D\mathcal R_c(q)[w]\cdot\psi.
\label{eq:signed-estimate-9}
\end{equation} Indeed the multiples of \(q\) disappear, and \(q'\cdot(w\times\psi)=0\) because \(q',w,\psi\) lie in the same two-dimensional tangent plane; this identity converts the polarised first-order term in \eqref{eq:signed-estimate-7} to the first-order term in \eqref{eq:signed-estimate-9}. Rotating \(q\) about \(e_3\) gives a family of exact solutions. Differentiating this family yields \(D\mathcal R_c(q)[e_3\times q]=0\), so \(e_3\times q\) is Jacobi and bounded.
\end{proof}

\subsection{Exclusion of poles, including asymptotic approach}

Applied to the rotational Jacobi field, the preceding lemma excludes both contact with a pole and approach to a pole at infinity.

\begin{lemma}[Separation from the poles]\label{lem:profile-poles}
Every stable profile satisfies $\sqrt{1-q_3^2}\ge\rho$ on $\mathbb R$ for some $\rho>0$ depending on the profile.
\end{lemma}

\begin{proof}
If a stable solution meets either pole, its bounded rotational Jacobi field vanishes at that point. The result just proved, after translating that point to zero, implies \(e_3\times q\equiv0\). Continuity then makes \(q\) a constant pole. A constant pole is unstable: for \(\xi(x_1)=\chi(x_1/R)e_1\), where \(e_1=(1,0,0)\), \(0\ne\chi\in C_c^\infty(\mathbb R)\), and \(R>0\), \[
Q_{q,c}(\xi)=R^{-1}\int|\chi'|^2-R\int|\chi|^2<0
\] for sufficiently large \(R\).

A stable solution is in fact uniformly separated from both poles. Otherwise there are \(a_n\in\mathbb R\) with \(|h(a_n)|\to1\). The derivative bounds obtained from \eqref{eq:signed-estimate-6}, compactness of the sphere, and a diagonal subsequence give smooth convergence of \(q(x_1+a_n)\) on every compact interval to an entire exact profile. Its speed is \(c\); stability passes to the limit as proved above. The limit meets a pole at zero, which is impossible. Hence for the stable solution in question there is \(\rho>0\) such that \begin{equation}
\sqrt{1-h(x_1)^2}\geq\rho\qquad(x_1\in\mathbb R).
\label{eq:signed-estimate-10}
\end{equation}
\end{proof}

\subsection{Reduction to a scalar stable equation}

Separation from the poles permits global height and phase coordinates. The conserved angular quantity eliminates the phase derivative, reducing the problem to a scalar equation whose stability controls the possible variation of the height.

\begin{lemma}[Scalar reduction and limiting helices]\label{lem:profile-scalar}
For a stable profile, let $\alpha=\arcsin q_3$, $S=\sin\alpha$ and $C=\cos\alpha$. There is a smooth real phase $\theta$ such that
\begin{equation}
q=(C\cos\theta,C\sin\theta,S),\qquad
\theta'=f(\alpha):=\frac{j+c\sin\alpha}{\cos^2\alpha}.
\label{eq:signed-estimate-11}
\end{equation}
With $j$ as in \eqref{eq:signed-estimate-6}, the height coordinate satisfies
\begin{equation}
\alpha''=SC(1-\theta'^2)-cC\theta'=W'(\alpha),\qquad
W(a)=\tfrac12\sin^2a-\frac{(j+c\sin a)^2}{2\cos^2a},
\label{eq:signed-estimate-12}
\end{equation}
and the stability inequality
\begin{equation}
\int(a'^2+W''(\alpha)a^2)\geq0\quad\hbox{for every real }a\in C_c^\infty(\mathbb R).
\label{eq:signed-estimate-15}
\end{equation}
Either $\alpha$ is constant or it is strictly monotone, with two distinct limiting heights giving stable helices with the same conserved quantity $j$.
\end{lemma}

\begin{proof}
By \cref{lem:profile-poles}, $\alpha=\arcsin q_3$ takes values in a compact subinterval of $(-\pi/2,\pi/2)$, so $C\ge\rho>0$. A smooth real phase lift $\theta$ exists on $\mathbb R$: its derivative is the angular derivative of the horizontal unit vector, and integration with a chosen initial argument produces a global lift. The definition of $j$ gives \eqref{eq:signed-estimate-11}. Projecting \eqref{eq:signed-estimate-1} onto the unit vector \((-S\cos\theta,-S\sin\theta,C)\) gives \eqref{eq:signed-estimate-12}. The equality with \(W'\) follows by differentiating its definition in \eqref{eq:signed-estimate-12} and substituting \eqref{eq:signed-estimate-11}. Here \(j\) is the fixed conserved real number from \eqref{eq:signed-estimate-6}, not a variable differentiated with \(a\).

We record the precise second variation in these coordinates. On any finite interval containing the support of a variation the local action is \begin{equation}
\int\left[\tfrac12(\alpha'^2+\cos^2\alpha\,\theta'^2+\sin^2\alpha)-c\sin\alpha\,\theta'\right].
\label{eq:signed-estimate-13}
\end{equation} Its first variation is the negative of \(\mathcal R_c(q)\) paired with the tangent variation: the two coordinate Euler expressions are \(-\alpha''+SC(1-\theta'^2)-cC\theta'\) and \(-(C^2\theta'-cS)'\), respectively. Thus at an exact solution the Hessian equals \eqref{eq:signed-estimate-2}, by \eqref{eq:signed-estimate-9}. If \(a,\beta\) are real smooth compactly supported coordinate variations of \(\alpha,\theta\), differentiating \eqref{eq:signed-estimate-13} twice gives \[
\int\left[a'^2+C^2\beta'^2-2C(2S\theta'+c)a\beta'
+\bigl((C^2-S^2)(1-\theta'^2)+cS\theta'\bigr)a^2\right].
\] Since \(f'(\alpha)=(2S\theta'+c)/C\), differentiating \eqref{eq:signed-estimate-12} as a function of \(\alpha\) gives \[
W''(\alpha)=(C^2-S^2)(1-\theta'^2)+cS\theta'-(2S\theta'+c)^2.
\] The Hessian therefore is exactly \begin{equation}
\int\left[a'^2+W''(\alpha)a^2+C^2\left(\beta'-\frac{2S\theta'+c}{C}a\right)^2\right].
\label{eq:signed-estimate-14}
\end{equation}

For each fixed real \(a\in C_c^\infty(\mathbb R)\), define \(g=(2S\theta'+c)a/C\) and \(m=\int g\). Choose a real function \(\psi\in C_c^\infty((0,1))\) with \(\int\psi=1\). For \(R>1\), place the interval \((T_R,T_R+R)\) strictly to the right of the support of \(a\), and put \[
\psi_R(x_1)=R^{-1}\psi((x_1-T_R)/R),\qquad
\beta_R(x_1)=\int_{-\infty}^{x_1}(g(t)-m\psi_R(t))\,dt.
\] The integrand has integral zero and compact support. Hence \(\beta_R\) is smooth and compactly supported, even though it may be constant on the long gap between the two supports. In \eqref{eq:signed-estimate-14} its square term is exactly \(m^2\int C^2\psi_R^2\leq m^2R^{-1}\int\psi^2\), since \(C^2\leq1\). Stability and \(R\to\infty\) yield \eqref{eq:signed-estimate-15}. No bound on the support size of an admissible variation, and no finite-energy assumption, was used.

Differentiating \eqref{eq:signed-estimate-12} shows that \(\alpha'\) is Jacobi for \eqref{eq:signed-estimate-15}: \(- (\alpha')''+W''(\alpha)\alpha'=0\). It is bounded, since \(\alpha'=h'/C\), \(q'\) is bounded, and \eqref{eq:signed-estimate-10} holds. The scalar zero result therefore says that either \(\alpha'\equiv0\) or \(\alpha'\) has no zero. In the latter case \(\alpha\) is strictly monotone. By \eqref{eq:signed-estimate-10} it has limits \(\alpha_-\) and \(\alpha_+\) in a compact subinterval of \((-\pi/2,\pi/2)\), and the two limits are different.

Both limiting heights give stable helices with the same conserved number \(j\). In detail, for any sequence \(a_n\to+\infty\), rotate \(q(x_1+a_n)\) about \(e_3\) by the angle \(-\theta(a_n)\). On compact \(x_1\)-intervals, \(\alpha(x_1+a_n)\to\alpha_+\) uniformly, and \[
\theta(x_1+a_n)-\theta(a_n)=\int_0^{x_1} f(\alpha(t+a_n))\,dt\longrightarrow f(\alpha_+)x_1.
\] The global derivative bounds permit a subsequence converging smoothly to the helix of height \(\sin\alpha_+\) and frequency \(f(\alpha_+)\). It is exact and stable by closure of the equation and stability. The same construction for \(a_n\to-\infty\) yields the other helix. Formula \eqref{eq:signed-estimate-6} is invariant under these rotations and translations and passes to the limit, giving the same \(j\) for both.
\end{proof}

\subsection{Exact stability criterion and momentum ordering for helices}

The scalar argument leaves either a constant-height profile or a strictly monotone height joining two stable helices. We determine which helices are stable and show that their conserved angular quantity determines their frequency uniquely.

\begin{lemma}[Stable helices and their ordering]\label{lem:profile-helices}
The nonpolar helices that solve \eqref{eq:signed-estimate-1} and are stable are exactly
\begin{equation}
q(x_1)=\left(\sqrt{1-h^2}\cos(\omega x_1+\gamma),\sqrt{1-h^2}\sin(\omega x_1+\gamma),h\right),\qquad
h=\frac{c\omega}{1-\omega^2},
\label{eq:signed-estimate-3}
\end{equation}
where $\gamma\in\mathbb R$, $\omega\in(-1,1)$, and
\begin{equation}
(1-\omega^2)^3\geq c^2(1+3\omega^2).
\label{eq:signed-estimate-4}
\end{equation}
Equality in \eqref{eq:signed-estimate-4} is allowed. On the resulting closed frequency interval, the conserved quantity $j$ is strictly increasing as a function of $\omega$.
\end{lemma}

\begin{proof}
Consider any nonpolar helix with constant height \(h\in(-1,1)\), real frequency \(\omega\), and arbitrary phase. Equation \eqref{eq:signed-estimate-12} for its constant height reduces to \begin{equation}
h(1-\omega^2)=c\omega.
\label{eq:signed-estimate-16}
\end{equation} The phase equation is satisfied because its coefficients are constant. In the height coordinate \(q_3=\sin\alpha\), the action density \eqref{eq:signed-estimate-13} is \[
\frac{(q_3')^2}{2(1-q_3^2)}+\frac{1-q_3^2}{2}\theta'^2+\frac{q_3^2}{2}-cq_3\theta'.
\] For real smooth compact variations \(a\) of \(q_3\) and \(\beta\) of \(\theta\), its Hessian at this helix is \begin{equation}
\begin{aligned}
\mathcal H[a,\beta]&:=\int_{\mathbb R}\Bigl[\frac{a'^2}{1-h^2}+(1-h^2)\beta'^2+(1-\omega^2)a^2\\
&\qquad -2(2h\omega+c)a\beta'\Bigr]\\
&=\int_{\mathbb R}\Bigl[\frac{a'^2}{1-h^2}+(1-h^2)\left(\beta'-\frac{2h\omega+c}{1-h^2}a\right)^2+Ma^2\Bigr],
\end{aligned}\label{eq:signed-estimate-17}
\end{equation} where \begin{equation}
M=1-\omega^2-\frac{(2h\omega+c)^2}{1-h^2}.
\label{eq:signed-estimate-18}
\end{equation} Every compact tangent field has such coordinate components, since the helix is nonpolar. Thus \(M\geq0\) implies stability, including \(M=0\).

Conversely, the remote phase compensation already constructed, with the coefficient in \eqref{eq:signed-estimate-17} in place of that in \eqref{eq:signed-estimate-14}, shows that stability implies \(\int(a'^2/(1-h^2)+Ma^2)\geq0\) for all real compact \(a\). If \(M<0\), choose \(a(x_1)=\chi(x_1/R)\) with a fixed nonzero \(\chi\in C_c^\infty\). The derivative term is a positive constant times \(R^{-1}\), whereas the \(M\)-term is a negative constant times \(R\), a contradiction for large \(R\). Therefore stability is equivalent to \(M\geq0\).

Equation \eqref{eq:signed-estimate-16} rules out \(\omega=\pm1\) because \(c>0\); \eqref{eq:signed-estimate-18} with \(M\geq0\) rules out \(|\omega|>1\). Thus a stable helix has \(|\omega|<1\) and \(h=c\omega/(1-\omega^2)\). Put \(t=\omega^2\) and \(k=c^2\). Substituting this height into \eqref{eq:signed-estimate-18} gives \begin{equation}
M(1-h^2)=\frac{(1-t)^3-k(1+3t)}{(1-t)^2}.
\label{eq:signed-estimate-19}
\end{equation} This proves the necessity of \eqref{eq:signed-estimate-4}. Conversely, \eqref{eq:signed-estimate-4} with \(|\omega|<1\) ensures that the defined height is nonpolar: indeed \[
h^2=\frac{kt}{(1-t)^2}\leq\frac{t(1-t)}{1+3t}<1,
\] where the last strict inequality follows from \(1+3t-t(1-t)=(1+t)^2>0\). Equations \eqref{eq:signed-estimate-16}, \eqref{eq:signed-estimate-17}, and \eqref{eq:signed-estimate-19} then prove that \eqref{eq:signed-estimate-3} is an exact stable solution.

The function \(G_c(t)=(1-t)^3-c^2(1+3t)\) has \(G_c(0)=1-c^2>0\), \(G_c(1)=-4c^2<0\), and \(G_c'(t)=-3(1-t)^2-3c^2<0\) on \([0,1]\). It therefore has a unique root \(\tau(c)\in(0,1)\). The allowed frequency interval is exactly \([-\sqrt{\tau(c)},\sqrt{\tau(c)}]\).

For a helix of the allowed form, the number in \eqref{eq:signed-estimate-6} is \begin{equation}
j_c(\omega)=(1-h^2)\omega-ch
=\omega\left(1-\frac{c^2}{(1-\omega^2)^2}\right),\qquad
j_c'(\omega)=1-\frac{c^2(1+3\omega^2)}{(1-\omega^2)^3}.
\label{eq:signed-estimate-20}
\end{equation} On the allowed interval this derivative is nonnegative, is strictly positive at every interior point, and vanishes only at the two endpoints. For any two distinct points of the closed interval, its integral between those points is strictly positive; consequently \(j_c\) is strictly increasing on the whole closed interval.
\end{proof}

The strict ordering just proved excludes the pair of limiting helices required by a nonconstant monotone height.

\begin{proposition}[Classification of stable profiles]\label{prop:profile-classification}
The stable maps satisfying \eqref{eq:signed-estimate-1} are exactly the helices \eqref{eq:signed-estimate-3} with frequencies satisfying \eqref{eq:signed-estimate-4}, including its endpoints.
\end{proposition}

\begin{proof}
By \cref{lem:profile-scalar}, it remains to exclude a strictly monotone height. The two stable limiting helices obtained from a strictly monotone \(\alpha\) have the same \(j\), so \eqref{eq:signed-estimate-20} makes their frequencies equal. Formula \eqref{eq:signed-estimate-16}, with \(|\omega|<1\), makes their heights equal, and injectivity of sine on \((-\pi/2,\pi/2)\) makes \(\alpha_- =\alpha_+\). This contradicts strict monotonicity. Hence \(\alpha'\equiv0\), and \eqref{eq:signed-estimate-11} makes the original profile a constant-height helix. Combined with the converse proved from \eqref{eq:signed-estimate-17}, this proves the classification, including its endpoints.
\end{proof}

\subsection{Sharp signed coercivity and uniformity}

With the profiles classified, the signed coercivity bound follows by comparing their explicit gradient and height energies.

\begin{corollary}[Sharp signed coercivity]\label{prop:signed-estimate}
Let \(\tau(c)\in(0,1)\) be the unique solution of \[
(1-\tau(c))^3=c^2(1+3\tau(c)),
\] and define \[
\delta_{\mathrm{opt}}(c)=\frac{\tau(c)(3+\tau(c))}{2+\tau(c)+\tau(c)^2}.
\] Every stable profile satisfies, pointwise on \(\mathbb R\), \begin{equation}
|q'|^2-q_3^2\geq\delta_{\mathrm{opt}}(c)(|q'|^2+q_3^2).
\label{eq:signed-estimate-5}
\end{equation} The constant \(\delta_{\mathrm{opt}}(c)>0\) is the largest constant that works for all stable profiles at this speed, and \[
\delta_{\mathrm{opt}}(c)\geq\frac{1-c^2}{3(1+c^2)}.
\] Consequently, for every nonempty compact \(K\subset(0,1)\), putting \(c_+=\max K\) and \[
\delta_K=\frac{1-c_+^2}{3(1+c_+^2)}>0,
\] one has \(|q'|^2-q_3^2\geq\delta_K(|q'|^2+q_3^2)\) for every stable profile at every \(c\in K\).
\end{corollary}

\begin{proof}
For \(t=\omega^2>0\), using \(|q'|^2=(1-h^2)t\) and \(h^2=c^2t/(1-t)^2\), one obtains \begin{equation}
\frac{|q'|^2-h^2}{|q'|^2+h^2}
=\frac{(1-t)^2-c^2(1+t)}{(1-t)(1-t+c^2)}
=1-\frac{2c^2}{(1-t)(1-t+c^2)}.
\label{eq:signed-estimate-21}
\end{equation} For fixed \(c>0\), the denominator in the last fraction is strictly decreasing in \(t\in[0,1)\), so the ratio is strictly decreasing. Its minimum on the allowed nonzero frequencies is attained at \(t=\tau(c)\). Substituting \(c^2=(1-\tau)^3/(1+3\tau)\) into \eqref{eq:signed-estimate-21} gives \[
\min\frac{|q'|^2-h^2}{|q'|^2+h^2}
=\frac{\tau(3+\tau)}{2+\tau+\tau^2}=\delta_{\mathrm{opt}}(c)>0.
\] Both endpoint helices are stable, nonconstant, and have positive denominator in \eqref{eq:signed-estimate-21}, so this constant cannot be improved. At \(\omega=0\), one has \(h=0\), and both sides of \eqref{eq:signed-estimate-5} are zero.

Since \(0<\tau<1\), the inequality \(3+\tau\geq 2+\tau+\tau^2\) gives \(\delta_{\mathrm{opt}}(c)\geq\tau\). Expanding its defining equation gives \[
1-c^2=\tau(3+3c^2-3\tau+\tau^2)\leq3(1+c^2)\tau,
\] because \(-3\tau+\tau^2\leq0\). Hence \(\delta_{\mathrm{opt}}(c)\geq(1-c^2)/(3(1+c^2))\). Finally, \(x\mapsto(1-x^2)/(3(1+x^2))\) decreases on \((0,1)\). This proves the asserted uniform constant for every compact \(K\subset(0,1)\), and completes the proof.
\end{proof}

\section{A uniform planar energy estimate}\label{app:energy-absorption}

We now convert the action, residual, index and dilation bounds in
\cref{prop:planar-selector} into an energy bound independent of the
regularisation parameter. The key estimate controls longitudinal
variation through the signed difference between longitudinal gradient
energy and anisotropy energy. Combining it with the dilation identities
absorbs the errors from minimax selection.

Fix a nonempty compact $K\subset(0,1)$. Use the notation of
\cref{sec:planar-framework}, and set
\[
p=|\partial_1 u|^2,\quad t=|\partial_2 u|^2,\qquad
e_\eta=(|\nabla u|^2+u_3^2)/2+\eta |\nabla u|^4/4,
\]
\[
F_\eta=p+u_3^2+\eta p^2/2,\qquad
D_\eta=p-u_3^2+3\eta p^2/2,\qquad
A_\eta(u;\Omega)=\tfrac12\int_\Omega(1+\eta |\nabla u|^2)t.
\]
Here $F_\eta$ collects the longitudinal-gradient and height contributions,
$D_\eta$ is the signed combination controlled by the one-dimensional
classification, and $A_\eta$ measures transverse variation.

We use the signed coercivity of stable one-dimensional profiles in \cref{prop:signed-estimate}. The consequence needed here is that there is \(\delta_*>0\), depending only on \(K\), such that every stable entire one-dimensional solution at a speed in \(K\) satisfies \[
|q'|^2-q_3^2\geq\delta_*(|q'|^2+q_3^2).
\] We combine this inequality with a small-amplitude estimate, a local
compactness argument, and a summation over cells. The index controls
the cells that carry a negative direction.

First, the definitions give the following pointwise and integrated identities, valid on every bounded measurable set \(\Omega\): \begin{equation}
-F_\eta\leq D_\eta\leq3F_\eta,\qquad
D_\eta-\delta F_\eta=(1-\delta)p-(1+\delta)u_3^2+\frac{3-\delta}{2}\eta p^2,
\label{eq:energy-absorption-A}
\end{equation} \begin{equation}
E_\eta(u;\Omega)=\frac12\int_\Omega F_\eta+
A_\eta(u;\Omega)-\frac{\eta}{4}\int_\Omega t^2.
\label{eq:energy-absorption-B}
\end{equation} In particular \(F_\eta\) is integrable if \(E_\eta(u)<\infty\).

\subsection{The small-amplitude linear estimate}

One possibility is that local energies tend to zero and the maps approach a constant equatorial state. The linearised equation provides the appropriate comparison in this case.

\begin{lemma}[Small-amplitude coercivity]\label{lem:energy-linear}
Put \(c_+=\max K\), \(\lambda_0=\sqrt{1-c_+^2}>0\), and \(C=\max(8,4/\lambda_0^4)\). For every \(G_0>1\), one can choose \(L>0\), depending only on \(K,G_0\), such that every entire real solution of \begin{equation}
b''-c\zeta'=0,\qquad \zeta''-\zeta+cb'=0,\qquad c\in K,
\label{eq:energy-absorption-C}
\end{equation} obeys the following implication. Define \(I_k=((k-\tfrac12)L,(k+\tfrac12)L)\) for \(k=-1,0,1\), and put \[
F_{\rm lin}=b'^2+\zeta'^2+\zeta^2,\quad D_{\rm lin}=b'^2+\zeta'^2-\zeta^2,
\quad f_k=\int_{I_k}F_{\rm lin},\quad d_0=\int_{I_0}D_{\rm lin}.
\] Then \begin{equation}
\max(f_{-1},f_1)\leq G_0f_0\quad\Longrightarrow\quad
d_0\geq\delta_{\rm lin}f_0,\qquad
\delta_{\rm lin}=\frac{\lambda_0^2}{4C}>0.
\label{eq:energy-absorption-D}
\end{equation} The constant \(\delta_{\rm lin}\) is independent of \(G_0\).
\end{lemma}

\begin{proof}
Integrating the first equation and substituting into the second shows that for some real \(\beta,p_+,p_-\), \[
\begin{aligned}
\zeta&=c\beta+X, & b'&=\beta+cX,\\
X&=p_+e^{\lambda x_1}+p_-e^{-\lambda x_1},\\
\zeta'&=\lambda Y, & Y&=p_+e^{\lambda x_1}-p_-e^{-\lambda x_1},\\
\lambda&=\sqrt{1-c^2}\ge\lambda_0.
\end{aligned}
\] Let \(Z=\beta^2+p_+^2e^{2\lambda x_1}+p_-^2e^{-2\lambda x_1}\). The three squares in \(F_{\rm lin}\) are bounded by \(3Z,2Z,3Z\), respectively, so \(F_{\rm lin}\leq8Z\). Conversely \[
\beta=\frac{b'-c\zeta}{\lambda^2},\qquad
X=\frac{\zeta-cb'}{\lambda^2},\qquad Y=\frac{\zeta'}{\lambda},
\] imply \(Z=\beta^2+(X^2+Y^2)/2\leq4F_{\rm lin}/\lambda_0^4\). Hence \(C^{-1}Z\leq F_{\rm lin}\leq CZ\).

Write \(W=\sinh(\lambda L)/\lambda\) and \(P^2=p_+^2+p_-^2\). Integration gives \[
\int_{I_0}Z=L\beta^2+WP^2,\quad
\int_{I_{-1}\cup I_1}Z=2L\beta^2+2W\cosh(2\lambda L)P^2.
\] Thus the neighbouring-interval assumption implies \[
[\cosh(2\lambda L)-C^2G_0]WP^2\leq(C^2G_0-1)L\beta^2.
\] Choose \(L\geq4/\lambda_0\) and \(\cosh(2\lambda_0L)\geq2C^2G_0\). Then \(WP^2\leq L\beta^2\). Also \(W\geq4L\): the function \(\sinh x/x\) increases for \(x>0\), and its series at \(x=4\) exceeds \(4\). Consequently \(P^2\leq\beta^2/4\). Direct expansion gives \(D_{\rm lin}=\lambda^2(\beta^2-4p_+p_-)\geq\lambda_0^2\beta^2/2\), whereas \(f_0\leq2CL\beta^2\). This proves \eqref{eq:energy-absorption-D}. If \(\beta=0\), the preceding bound on \(WP^2\) forces \(p_+=p_-=0\), so the implication also holds.
\end{proof}

\subsection{Local rotational compactness and transfer of a negative direction}

We next consider sequences whose transverse energy and residual tend to zero. Compactly supported target rotations allow passage to a one-dimensional limit and show that any negative direction of that limit persists in the approximating maps.

\begin{lemma}[One-dimensional limits and negative directions]\label{lem:energy-rotational}
Suppose \(\eta_j>0\) tends to zero, \(c_j\in K\) tends to \(c\in K\), and smooth maps \(u_j\) are defined on open boxes exhausting \(\mathbb R^2\). Suppose their \(E_{\eta_j}\)-energies on each fixed box are bounded, their \(A_{\eta_j}\)-energies on each fixed box tend to zero, and for every fixed box \(\Omega\), \begin{equation}
r_j(\Omega):=\sup\{|R_{c_j,\eta_j}(u_j)[v]|:v\text{ compact tangent in }\Omega,\ N_{\eta_j}(v)\leq1\}\longrightarrow0.
\label{eq:energy-absorption-E}
\end{equation} Then a subsequence converges weakly in $H^1_{\rm loc}$ and strongly in $L^2_{\rm loc}$ to a smooth entire one-dimensional solution at speed $c$. Every compact negative direction for its one-dimensional quadratic form transfers to compact tangent fields for $u_j$ on one fixed box, with unit $L^2$ norm, uniformly bounded $N_{\eta_j}$ norm, and a fixed negative form bound.
\end{lemma}

\begin{proof}
Local \(H^1\) compactness and a diagonal subsequence give a weak \(H^1_{\rm loc}\) and strong \(L^2_{\rm loc}\) limit \(q\). Since \(|u_j|=1\), the limit has unit length almost everywhere. Since \(A_{\eta_j}\geq\tfrac12\int|\partial_2 u_j|^2\), it is independent of \(x_2\).

For a fixed \(\Phi\in C_c^\infty(\mathbb R^2;\mathbb R^3)\), the tangent test \(v_j=\Phi\times u_j\) has uniformly bounded \(N_{\eta_j}\). This follows from \(|v_j|\leq|\Phi|\) and \(|\nabla v_j|\leq|\nabla\Phi|+|\Phi||\nabla u_j|\), the local bounds on \(\int|\nabla u_j|^2\) and \(\eta_j\int|\nabla u_j|^4\), and finite volume of the fixed support. In the residual one obtains the exact cancellation \[
R_{c_j,\eta_j}(u_j)[\Phi\times u_j]
=\int(1+\eta_j|\nabla u_j|^2)\sum_i (u_j\times\partial_i u_j)\cdot\partial_i\Phi
+(u_j)_3(u_j\times e_3)\cdot\Phi+c_j\partial_1 u_j\cdot\Phi.
\] The quartic contribution is bounded by \(C_\Phi\eta_j\int |\nabla u_j|^3\), which tends to zero because \[
\eta_j\int |\nabla u_j|^3\leq
\sqrt{\eta_j}\left(\int |\nabla u_j|^2\right)^{1/2}
\left(\eta_j\int |\nabla u_j|^4\right)^{1/2}.
\] In each remaining term a weakly convergent derivative is multiplied by a strongly convergent bounded coefficient. Hence the limit satisfies, in the one-dimensional distributional sense, \begin{equation}
(q\times q')'=q_3(q\times e_3)+cq'.
\label{eq:energy-absorption-F}
\end{equation} The right side is locally \(L^2\), so \(q\times q'\) is locally \(H^1\). Since \(q'=(q\times q')\times q\), one-dimensional Sobolev continuity makes \(q'\) continuous. Equation \eqref{eq:energy-absorption-F} and repeated differentiation then give smoothness. Crossing \eqref{eq:energy-absorption-F} with \(q\), and using \(q\cdot q''=-|q'|^2\), gives \[
q''+(|q'|^2+q_3^2)q-q_3e_3+cq\times q'=0.
\]

We also prove that any compact negative direction for the one-dimensional form at this limit transfers to a fixed compact two-dimensional region for large \(j\). For fixed \(\Phi\), let \[
\begin{gathered}
v=\Phi\times u,\quad a=\Phi\times(\Phi\times u),\\
\mathcal B_\Phi=\sum_i\bigl[|\partial_i\Phi\times u|^2+
\partial_i u\cdot((\partial_i\Phi\times\Phi)\times u)\bigr],\\
\mathcal L_\Phi=\sum_i\partial_i u\cdot(\partial_i\Phi\times u).
\end{gathered}
\] Differentiate \(R_{c,\eta}(u_\tau)[\partial_\tau u_\tau]\) at \(\tau=0\) along \(u_\tau=\exp(\tau[\Phi]_\times)u\), where \([\Phi]_\times w=\Phi\times w\). The Dirichlet density differentiates to \(\mathcal B_\Phi\); its first derivative is \(\mathcal L_\Phi\). The height contribution is \(v_3^2+u_3\,a_3\). For the momentum term, \(u\cdot((\Phi\times u)\times \partial_1 u)=\Phi\cdot \partial_1 u\), whose derivative is \(u\cdot(\Phi\times\partial_1 \Phi)\). Thus this derivative is \begin{equation}
\mathcal H(u;\Phi)=\int\mathcal B_\Phi+v_3^2+u_3\,a_3+
cu\cdot(\Phi\times\partial_1 \Phi)+\eta(|\nabla u|^2\mathcal B_\Phi+2\mathcal L_\Phi^2).
\label{eq:energy-absorption-G}
\end{equation} Differentiation along an arbitrary sphere-valued path with initial velocity \(v\) and acceleration \(a\), compared with acceleration \(-|v|^2u\), also gives directly from the definitions \begin{equation}
\mathcal H(u;\Phi)-Q_{u,c,\eta}(v)
=R_{c,\eta}(u)[a+|v|^2u].
\label{eq:energy-absorption-H}
\end{equation} The field on the right is tangent, compactly supported, and has uniformly bounded \(N_{\eta_j}\) for \(u=u_j\), by the same product bounds as for \(v_j\). The part of \eqref{eq:energy-absorption-G} without \(\eta\) is continuous under the convergence just established, because \(\mathcal B_\Phi\) contains derivatives of \(u\) only linearly. The last part has absolute value at most \(C_\Phi\eta_j\int(|\nabla u_j|^2+|\nabla u_j|^3)\), and tends to zero. Equations \eqref{eq:energy-absorption-E} and \eqref{eq:energy-absorption-H} imply \[
Q_{u_j,c_j,\eta_j}(\Phi\times u_j)\longrightarrow Q_{q,c,0}(\Phi\times q).
\]

Now let \(\xi(x_1)\) be a smooth compact tangent field with negative one-dimensional form \(Q^{1d}_{q,c}(\xi)<0\), where this form is defined in \eqref{eq:signed-estimate-2}. Fix a nonzero real \(\chi\in C_c^\infty(\mathbb R)\), and choose \(\Phi(x_1,x_2)=\chi(x_2/H_0)\,q(x_1)\times\xi(x_1)\). Here \(H_0>0\) is a fixed large number. Since \((q\times\xi)\times q=\xi\), the limiting form is \[
H_0\|\chi\|_2^2Q^{1d}_{q,c}(\xi)+H_0^{-1}\|\chi'\|_2^2\|\xi\|_2^2,
\] which is strictly negative when \(H_0\) is large enough. Its support is one fixed box. The fields \(\Phi\times u_j\) have bounded \(N_{\eta_j}\), and their \(L^2\) norms tend to the positive \(L^2\) norm of \(\chi(x_2/H_0)\xi(x_1)\). Normalising them proves the claimed transfer with a fixed negative bound and a fixed norm bound.
\end{proof}

\subsection{Choice of cells and the cylinder alternative}

The nonlinear compactness argument and the small-amplitude estimate now give a local alternative on rectangular cells. Each cell either has the required signed coercivity or can be accounted for by transverse energy, residual, a negative variation, or growth of the energy in neighbouring cells.

Choose \[
0<\delta<\min(\delta_*,\delta_{\rm lin},1/2).
\] For \(G>1\) put \(q_G=((G+1)/(G-1))^2-1\). Choose \(G\) so large that \(q_G<1\) and \(\rho_G=q_G/(1-q_G)<\delta\). Then choose \(L\) for the linear estimate \eqref{eq:energy-absorption-D} with \(G_0=4G\). These choices depend only on \(K\).

For \(\nu=(\nu_1,\nu_2)\in\mathbb Z^2\), let \[
C_\nu=(\nu_1L,\nu_2)+(-L/2,L/2)\times(-1/2,1/2).
\] For an integer \(N\geq1\), let \(\Omega_{\nu,N}\) be the open rectangular box whose closure is the union of \(\overline{C_\mu}\) with \(|\mu-\nu|_\infty\leq N\). Put \(f_\nu=\int_{C_\nu}F_\eta\) and \(d_\nu=\int_{C_\nu}D_\eta\). For a smooth map defined on \(\Omega\), define \(r(u;\Omega)\) by the supremum in \eqref{eq:energy-absorption-E}, with its own \(c,\eta\).

\begin{lemma}[The cylinder alternative]\label{lem:energy-cells}
There exist \(N\geq1\) and \(\eta_0,\varepsilon_0,\kappa,C_0>0\), depending only on \(K\), such that every smooth map with finite \(E_\eta\)-energy on \(\Omega_{\nu,N}\), every \(c\in K\), and every \(0<\eta\leq\eta_0\) satisfy at least one of \begin{equation}
\begin{array}{ll}
\mathrm{(i)}& d_\nu\geq\delta f_\nu,\\
\mathrm{(ii)}& r(u;\Omega_{\nu,N})^2>\varepsilon_0 f_\nu,\\
\mathrm{(iii)}& A_\eta(u;\Omega_{\nu,N})>\varepsilon_0 f_\nu,\\
\mathrm{(iv)}& \text{there is a compact tangent }v\text{ in }\Omega_{\nu,N}
\text{ with }\|v\|_2=1,\ N_\eta(v)\leq C_0,\ Q_{u,c,\eta}(v)\leq-\kappa,\\
\mathrm{(v)}& f_\mu\geq G^{|\mu-\nu|_1}f_\nu
\text{ for some }0<|\mu-\nu|_\infty\leq N.
\end{array}
\label{eq:energy-absorption-I}
\end{equation}
\end{lemma}

\begin{proof}
Suppose this claim is false. For each positive integer \(j\), take a counterexample with \(N=j\), \(\eta_0=\varepsilon_0=\kappa=1/j\), and \(C_0=j\); translate its cell to \(C_0\). Denote its parameters and map by \(c_j,\eta_j,u_j\), pass to \(c_j\to c\in K\), and put \(m_j=f_{0,j}\). Since (i) fails, \(m_j>0\). The identity \(D_\eta=F_\eta-2u_3^2+\eta p^2\) and \(|u_3|\leq1\) give \begin{equation}
(1-\delta)m_j<2\int_{C_0}(u_j)_3^2\leq2L,\qquad
0<m_j<\frac{2L}{1-\delta}.
\label{eq:energy-absorption-J}
\end{equation} Failures of (ii), (iii), and (v), together with \eqref{eq:energy-absorption-B}, imply for each fixed integer \(R\geq1\) and all \(j\geq R\) \[
r(u_j;\Omega_{0,j})^2\leq m_j/j,\quad
A_{\eta_j}(u_j;\Omega_{0,j})\leq m_j/j,
\] \begin{equation}
E_{\eta_j}(u_j;\Omega_{0,R})
\leq\left(\frac12\sum_{|\nu|_\infty\leq R}G^{|\nu|_1}+\frac1j\right)m_j.
\label{eq:energy-absorption-K}
\end{equation} Take a subsequence with \(m_j\to m\geq0\).

If \(m>0\), local rotational compactness gives a smooth entire one-dimensional profile \(q\). Strong convergence of the height and \eqref{eq:energy-absorption-J} give \(\int_{C_0}q_3^2\geq(1-\delta)m/2>0\). Failure of (i), \eqref{eq:energy-absorption-A}, and lower semicontinuity after dropping the nonnegative quartic term give \begin{equation}
(1-\delta)\int_{C_0}|q'|^2-(1+\delta)\int_{C_0}q_3^2\leq0.
\label{eq:energy-absorption-L}
\end{equation} If this profile were stable, \cref{prop:signed-estimate} would make the left side at least \((\delta_*-\delta)\int_{C_0}(|q'|^2+q_3^2)>0\). Thus it is unstable. A negative compact one-dimensional direction transfers by \eqref{eq:energy-absorption-G}--\eqref{eq:energy-absorption-H} to fields with fixed compact support, bounded norm, unit \(L^2\) norm, and a fixed negative form bound. For large \(j\), their support lies in \(\Omega_{0,j}\), their norm is at most \(j\), and their form is at most \(-1/j\). This contradicts failure of (iv).

If \(m=0\), let \(\bar u_j\) be the average of \(u_j\) over \(C_0\). Poincaré's inequality and \eqref{eq:energy-absorption-K} give \[
|\bar u_j|^2=1-\frac1L\int_{C_0}|u_j-\bar u_j|^2=1-O(m_j),
\qquad |(\bar u_j)_3|=O(\sqrt{m_j}).
\] Rotate each \(u_j\) by a constant rotation about \(e_3\) so that its average horizontal vector points toward \(e_1=(1,0,0)\). Its length is \(1-O(m_j)\), and consequently \(\|u_j-e_1\|_{H^1(C_0)}\leq C\sqrt{m_j}\). On every fixed larger box \(\Omega_{0,R}\) the same bound, with a constant \(C_R\), follows from \eqref{eq:energy-absorption-K} and the anchored Poincaré inequality \(\|v-\operatorname{avg}_{C_0}v\|_{L^2(\Omega_{0,R})}\leq C_R\|\nabla v\|_{L^2(\Omega_{0,R})}\). For example, this last inequality follows by subtracting the average on \(\Omega_{0,R}\), applying Poincaré on that box, and controlling the difference of averages by the \(L^2(C_0)\) norm. Put \(W_j=(u_j-e_1)/\sqrt{m_j}\).

The maps \(W_j\) are bounded in \(H^1\) on every fixed box. After passing to a subsequence, they converge weakly in \(H^1_{\rm loc}\) and strongly in \(L^2_{\rm loc}\). The identity \((W_j)_1=-\sqrt{m_j}|W_j|^2/2\) shows that the first component of the limit is zero. Also \(\int|\partial_2 W_j|^2\leq2/j\) on each fixed box for large \(j\), by \eqref{eq:energy-absorption-K}, so the limit is \((0,b(x_1),\zeta(x_1))\).

Test the rotational residual with \(\Phi\times u_j\), divide by \(\sqrt{m_j}\), and use \eqref{eq:energy-absorption-K}. Its absolute value tends to zero since the residual is at most \(\sqrt{m_j/j}\) times a uniformly bounded test norm. Using the expansion \(u_j=e_1+\sqrt{m_j}W_j\), we find that every unregularised nonlinear remainder has an \(L^1\) bound \(C_\Phi\sqrt{m_j}\|W_j\|_{H^1(\operatorname{supp}\Phi)}^2\), which tends to zero. The regularised remainder is bounded by \[
\frac{C_\Phi\eta_j}{\sqrt{m_j}}\int |\nabla u_j|^3
\leq C_{\Phi,R}\sqrt{\eta_jm_j}\longrightarrow0.
\] The limiting flux equation is \(e_1\times W''=\zeta(e_1\times e_3)+cW'\); its two components are exactly \eqref{eq:energy-absorption-C}. Thus \(b,\zeta\) are smooth solutions of that linear system.

Divide \eqref{eq:energy-absorption-J} by \(m_j\) and use strong height convergence to obtain \(\int_{I_0}\zeta^2\geq(1-\delta)/2\). Divide the failed inequality (i) by \(m_j\), use \eqref{eq:energy-absorption-A}, drop its quartic term, and use lower semicontinuity and strong height convergence. It follows that \[
d^{\rm lin}_0\leq\delta f^{\rm lin}_0,\qquad
f^{\rm lin}_0\geq(1-\delta)/2>1/4.
\] For the two neighbouring longitudinal cells, failure of (v) gives normalised mass at most \(G\), and lower semicontinuity gives \(f^{\rm lin}_{\pm1}\leq G\leq4Gf^{\rm lin}_0\). The choice of \(L\) and \eqref{eq:energy-absorption-D} therefore give \(d^{\rm lin}_0\geq\delta_{\rm lin}f^{\rm lin}_0\), contradicting \(\delta<\delta_{\rm lin}\). This finishes the proof of \eqref{eq:energy-absorption-I}.
\end{proof}

\subsection{Weighted norms and the global cell argument}

We now sum these local alternatives over the plane. The weighted norm estimates make it possible to control residuals and negative directions on disjoint regions without assuming a bound for the total energy.

\begin{lemma}[Weighted norms on local and disjoint supports]\label{lem:energy-norms}
Let $u$ be smooth with finite $E_\eta$. If a compact tangent field $v$ is supported in an open set $\Omega$ with $E_\eta(u;\Omega)\le M$, then
\begin{equation}
N_\eta(v)\leq B_{\eta,u}(v)\leq H(M)N_\eta(v),\qquad
H(M)=1+(\sqrt2+2)\sqrt M+(4M)^{1/4}.
\label{eq:energy-absorption-M}
\end{equation}
For every finite family of compact tangent fields $v_i$ with disjoint supports and every finite list of real numbers $a_i$,
\begin{equation}
B_{\eta,u}\left(\sum_i a_iv_i\right)^2
\leq7\sum_i a_i^2B_{\eta,u}(v_i)^2.
\label{eq:energy-absorption-N}
\end{equation}
\end{lemma}

\begin{proof}
Fix the constants in \eqref{eq:energy-absorption-I}, reduce \(\eta_0\) to at most one if necessary, and let $u$ be smooth with finite $E_\eta$. If \(v\) is supported in an open set \(\Omega\) with \(E_\eta(u;\Omega)\leq M\), the three additional terms of the weighted norm obey \[
\|\sqrt{|\nabla u|^2+u_3^2}\,v\|_2\leq\sqrt{2M}\|v\|_\infty,\qquad
\sqrt\eta\||\nabla u|^2\,v\|_2\leq2\sqrt M\|v\|_\infty,
\] \[
\sqrt\eta\||\nabla u|\,\nabla v\|_2
\leq(\eta\int_\Omega |\nabla u|^4)^{1/4}\eta^{1/4}\|\nabla v\|_4
\leq(4M)^{1/4}\eta^{1/4}\|\nabla v\|_4.
\] These estimates prove \eqref{eq:energy-absorption-M}.

We also need a disjoint-support estimate. Write \(B_{\eta,u}\) as the sum of its seven nonnegative norm components, as defined in \eqref{eq:planar-variation-norms}. For fields \(v_i\) with disjoint supports and a finite list of real numbers \(a_i\), each component \(N_k\) satisfies \[
N_k\left(\sum_i a_iv_i\right)^2\leq\sum_i a_i^2N_k(v_i)^2.
\] For the \(L^2\) components this is equality. For the \(L^4\) component it is the inequality \((\sum_i a_i^4 n_i^4)^{1/2}\leq\sum_i a_i^2n_i^2\), and for the supremum component it follows by bounding a maximum by a sum of squares. Cauchy--Schwarz in the seven components therefore gives \eqref{eq:energy-absorption-N}.
\end{proof}

The norm comparisons allow the residual and index bounds to control the corresponding classes of cells. We can therefore sum the cell estimates while retaining the sign of $\int D_\eta$.

\begin{proposition}[The signed energy estimate]\label{prop:energy-signed}
There are $a_K,C_K,\eta_K,\theta_K>0$, depending only on $K$, such that for $c\in K$, $0<\eta\le\eta_K$ and $0<\theta\le\theta_K$ the following holds. If $u$ is smooth, $E_\eta(u)<\infty$, and
\begin{equation}\label{eq:energy-absorption-S1}
|R_{c,\eta,u}[v]|\le\varepsilon B_{\eta,u}(v)
\quad\text{for every compact smooth tangent field }v,
\end{equation}
where $\varepsilon\ge0$, then
\begin{equation}\label{eq:energy-absorption-S2}
\int F_\eta\le a_K\int D_\eta+
C_K\bigl(A_\eta(u)+i_\theta(u;c,\eta)+\varepsilon^2\bigr).
\end{equation}
\end{proposition}

\begin{proof}
Fix the constants $N,\eta_0,\varepsilon_0,\kappa,C_0$ from \cref{lem:energy-cells}, reducing $\eta_0$ to at most one. Assign every lattice cell to the first alternative it satisfies in the following order: (i), (iii), (v), (ii), (iv). Thus the classes are respectively coercive, transverse, growth, residual, and negative. Denote the sums of \(f_\nu\) over these five classes by \(\mathsf C,\mathsf B_{\rm tr},\mathsf H,\mathsf B_{\rm res},\mathsf B_{\rm neg}\), and put \(\mathsf B=\mathsf B_{\rm tr}+\mathsf B_{\rm res}+\mathsf B_{\rm neg}\). These are finite nonnegative numbers since \(\sum_\nu f_\nu=\int F_\eta<\infty\). Cells of zero mass satisfy (i).

Each growth cell \(\nu\) selects one parent \(\mu\) given by (v), so \(f_\nu\leq G^{-|\mu-\nu|_1}f_\mu\). For a fixed parent, the sum of the masses of its assigned children is at most \[
f_\mu\sum_{k\in\mathbb Z^2\setminus\{0\}}G^{-|k|_1}
=q_Gf_\mu.
\] Summing these nonnegative terms yields \(\mathsf H\leq q_G(\mathsf C+\mathsf B+\mathsf H)\), and hence \(\mathsf H\leq\rho_G(\mathsf C+\mathsf B)\). Since \(D_\eta\geq-F_\eta\) everywhere, while \(d_\nu\geq\delta f_\nu\) on coercive cells, \[
\int D_\eta\geq(\delta-\rho_G)\mathsf C-(1+\rho_G)\mathsf B,\qquad
\int F_\eta\leq(1+\rho_G)(\mathsf C+\mathsf B).
\] Eliminating \(\mathsf C\) using \(\delta-\rho_G>0\) gives a signed estimate for use with the dilation identities: \begin{equation}
\int F_\eta\leq a\int D_\eta+b\mathsf B,\qquad
a=\frac{1+\rho_G}{\delta-\rho_G}>0,\quad
b=\frac{(1+\rho_G)(1+\delta)}{\delta-\rho_G}>0.
\label{eq:energy-absorption-O}
\end{equation}

Let \(K_0=(2N+1)^2\). The enlarged boxes \(\Omega_{\nu,N}\) have overlap multiplicity \(K_0\) away from their boundaries. Colour the lattice by the two coordinate residues modulo \(2N+1\); there are \(K_0\) colours, and the enlarged open boxes in each colour are disjoint. For a transverse cell, (iii) implies \(f_\nu< A_\eta(u;\Omega_{\nu,N})/\varepsilon_0\). Therefore \begin{equation}
\mathsf B_{\rm tr}\leq K_0A_\eta(u)/\varepsilon_0.
\label{eq:energy-absorption-P}
\end{equation}

Every residual or negative cell fails (i), (iii), and (v). The calculation \eqref{eq:energy-absorption-J} gives \[
f_\nu\leq M_0:=\frac{2L}{1-\delta}.
\] Failure of (v) bounds every neighbouring cell's mass by \(f_{\nu+k}\leq G^{|k|_1}f_\nu\) for \(|k|_\infty\leq N\), including \(k=0\). Failure of (iii), together with \eqref{eq:energy-absorption-B}, therefore gives \begin{equation}
E_\eta(u;\Omega_{\nu,N})\leq
M_*:=M_0\left(\frac12\sum_{|k|_\infty\leq N}G^{|k|_1}+\varepsilon_0\right).
\label{eq:energy-absorption-Q}
\end{equation} In particular this local energy bound is independent of the total energy of \(u\).

For each such cell define the weighted local residual \[
\rho_\nu=\sup\{|R_{c,\eta}(u)[v]|:v\text{ compact tangent in }\Omega_{\nu,N},
\ B_{\eta,u}(v)\leq1\}.
\] It is finite because the proposition assumes a global weighted residual bound. By \eqref{eq:energy-absorption-M} and \eqref{eq:energy-absorption-Q}, \(r(u;\Omega_{\nu,N})\leq H(M_*)\rho_\nu\). For residual cells, (ii) consequently implies \begin{equation}
f_\nu\leq H(M_*)^2\rho_\nu^2/\varepsilon_0.
\label{eq:energy-absorption-R}
\end{equation} Within each colour, the enlarged boxes are disjoint, and \(\sum_\nu\rho_\nu^2\leq7\varepsilon^2\), where the sum is over any selection of these cells in the colour. To prove this, first take a finite selection and tests \(v_\nu\) with \(B_{\eta,u}(v_\nu)\leq1\), choose their signs so that \(r_\nu=R_{c,\eta}(u)[v_\nu]\geq0\), and apply the global residual bound and \eqref{eq:energy-absorption-N}: \[
\sum_\nu r_\nu^2
=R_{c,\eta}(u)\left[\sum_\nu r_\nu v_\nu\right]
\leq\varepsilon\sqrt{7\sum_\nu r_\nu^2}.
\] Thus \(\sum_\nu r_\nu^2\leq7\varepsilon^2\). Approximate each local supremum on this finite selection, then increase the finite selections. Summing the colours and using \eqref{eq:energy-absorption-R} yields \begin{equation}
\mathsf B_{\rm res}\leq
\frac{7K_0H(M_*)^2}{\varepsilon_0}\varepsilon^2.
\label{eq:energy-absorption-S}
\end{equation}

For a negative cell choose the test provided by (iv). It has unit \(L^2\) norm, form at most \(-\kappa\), and, by \eqref{eq:energy-absorption-M}--\eqref{eq:energy-absorption-Q}, weighted norm at most \(H(M_*)C_0\). In a single colour these tests have disjoint supports and are linearly independent. The quadratic form is local, so for their finite linear combinations \begin{equation}
Q_{u,c,\eta}\left(\sum_\nu a_\nu v_\nu\right)
\leq-\kappa\sum_\nu a_\nu^2
\leq-\theta_* B_{\eta,u}\left(\sum_\nu a_\nu v_\nu\right)^2,\qquad
\theta_*=\frac{\kappa}{7H(M_*)^2C_0^2}>0,
\label{eq:energy-absorption-T}
\end{equation} where the last inequality is \eqref{eq:energy-absorption-N}. For every \(0<\theta\leq\theta_*\), the number of negative cells in each colour is therefore at most \(i_\theta(u;c,\eta)\), with the convention that if there are infinitely many, the index is infinite. Since each such cell has mass at most \(M_0\), \begin{equation}
\mathsf B_{\rm neg}\leq M_0K_0i_\theta(u;c,\eta).
\label{eq:energy-absorption-U}
\end{equation}

When the index is infinite the desired inequality is immediate in the extended real sense. Otherwise insert \eqref{eq:energy-absorption-P}, \eqref{eq:energy-absorption-S}, and \eqref{eq:energy-absorption-U} into \eqref{eq:energy-absorption-O}. Set \(\eta_K=\eta_0\), \(\theta_K=\theta_*\), \(a_K=a\), and choose \(C_K\) to be \(b\) times the maximum of the three coefficients in \eqref{eq:energy-absorption-P}, \eqref{eq:energy-absorption-S}, and \eqref{eq:energy-absorption-U}. This proves \eqref{eq:energy-absorption-S2}. Set \(L_K=L\) and \(N_K=N\). Every constant was chosen using only \(K\); the only speed requirement in the linear argument was \(\lambda_0>0\), and the only speed requirement in the nonlinear argument was \cref{prop:signed-estimate} for \(0<c<1\). There is consequently no further numerical speed restriction.
\end{proof}

\subsection{Momentum variation, the intrinsic Hessian, and dilation formulas}

The signed estimate becomes an energy bound once its longitudinal and transverse terms are expressed through the action and the dilation derivatives. We first record the variation and scaling identities that justify this conversion.

\begin{lemma}[Variations of the map and dilations]\label{lem:energy-variations}
For a smooth compact perturbation $u$ of $e_1$, the first and intrinsic second variations of $J_{c,\eta}$ are $R_{c,\eta}$ and $Q_{u,c,\eta}$. For every compact smooth tangent field $v$,
\begin{equation}
dP(u)[v]=\int u\cdot(v\times \partial_1 u).
\label{eq:energy-absorption-V2}
\end{equation}
The coordinate dilations $u_{\alpha,\beta}(x_1,x_2)=u(e^{-\alpha}x_1,e^{-\beta}x_2)$ satisfy
\begin{equation}
\begin{aligned}
&A(u_{\alpha,\beta})=e^{\alpha-\beta}A,\quad
K_1(u_{\alpha,\beta})=e^{-\alpha+\beta}K_1,\quad
V(u_{\alpha,\beta})=e^{\alpha+\beta}V,\quad
P(u_{\alpha,\beta})=e^\beta P,\\
&Q_1(u_{\alpha,\beta})=e^{-3\alpha+\beta}Q_1,\quad
Q_m(u_{\alpha,\beta})=e^{-\alpha-\beta}Q_m,\quad
Q_2(u_{\alpha,\beta})=e^{\alpha-3\beta}Q_2.
\end{aligned}
\label{eq:energy-absorption-V3}\end{equation}
The local variation identities also hold for compactly supported variations of locally defined smooth maps.
\end{lemma}

\begin{proof}
First let \(u\) be equal to \(e_1\) off a compact set. If \(u_\tau\) is a smooth sphere-valued variation with compactly supported change and \(v=\partial_\tau u_\tau|_0\), differentiation gives \begin{equation}
\partial_\tau\{u_\tau\cdot(\partial_1 u_\tau\times\partial_2 u_\tau)\}|_0
=\partial_1 \{u\cdot(v\times \partial_2 u)\}
+\partial_2 \{u\cdot(\partial_1 u\times v)\}.
\label{eq:energy-absorption-V1}
\end{equation} To verify this identity, expand the right side. Its mixed second derivatives cancel. Every remaining triple product of \(v,\partial_1 u,\partial_2 u\) vanishes, because these vectors are tangent to the same two-dimensional plane. The terms with derivatives of \(v\) are precisely the derivative of the expression on the left. Integration against \(x_2\), including the compact-support integration by parts, gives \eqref{eq:energy-absorption-V2}. The derivative of \(E_\eta+cP\) is therefore \(R_{c,\eta}\). Along a sphere path of initial velocity \(v\) and acceleration \(-|v|^2u\), differentiate this expression once more. The quadratic-gradient term gives \(|\nabla v|^2-|\nabla u|^2|v|^2\); the quartic term gives \(\eta |\nabla u|^2(|\nabla v|^2-|\nabla u|^2|v|^2)+2\eta(\nabla u:\nabla v)^2\); the height term gives \(v_3^2-u_3^2|v|^2\); and the momentum term gives \(c\,u\cdot(v\times \partial_1 v)\). The other momentum derivative terms vanish since \(v\cdot(v\times \partial_1 u)=0\) and \(u\cdot((-|v|^2u)\times \partial_1 u)=0\). This is exactly \(Q\). For a map only locally defined, the same calculation uses a local action difference supported in a compact subset; hence it also justifies the residual and Hessian calculations used earlier.

For \(u_{\alpha,\beta}(x_1,x_2)=u(e^{-\alpha}x_1,e^{-\beta}x_2)\), the change of variables has Jacobian \(e^{\alpha+\beta}\). This gives the scaling identities \eqref{eq:energy-absorption-V3}. In the momentum identity the factors from \(x_2,\partial_1 u,\partial_2 u\), and volume are \(e^\beta,e^{-\alpha},e^{-\beta},e^{\alpha+\beta}\). This proves its factor \(e^\beta\). Differentiating the action formed from \eqref{eq:energy-absorption-V3} proves the formulas for \(D_1,D_2\).
\end{proof}

We can now insert the dilation identities into the signed estimate and absorb the scalar errors using the weight $W_\eta$.

\begin{proposition}[Energy bounds from the scaling errors]\label{prop:energy-absorption}
Use the constants and parameter ranges of \cref{prop:energy-signed}. For any specified finite value of $P$, use
\eqref{eq:planar-scaling-data} to define $J,D_1,D_2,W_\eta$. Algebra gives
\begin{equation}\label{eq:energy-absorption-S4}
J-D_2=2A_\eta,\qquad
\int D_\eta=\int(t+\eta t^2/2-\eta pt)-2D_1\le2A_\eta-2D_1.
\end{equation}
Under the hypotheses of \cref{prop:energy-signed}, with $H_K=a_K+1+C_K/2$,
\begin{equation}\label{eq:energy-absorption-S5}
E_\eta(u)\le\frac{H_K}{2}(J-D_2)-a_KD_1
 +\frac{C_K}{2}(i_\theta+\varepsilon^2),
\end{equation}
and consequently
\begin{equation}\label{eq:energy-absorption-S6}
E_\eta(u)\le\widehat C_K\bigl(J_++|D_1|+|D_2|+i_\theta+\varepsilon^2\bigr).
\end{equation}

If, in addition, $0\le\varepsilon\le1$ and
\begin{equation}\label{eq:energy-absorption-S7}
|J|\le M,\qquad i_\theta\le I,\qquad
|D_1|+|D_2|\le\varepsilon W_\eta,
\end{equation}
then
\begin{equation}\label{eq:energy-absorption-S8}
E_\eta(u)\le C_K'(1+M+I).
\end{equation}
In particular, no common energy bound is assumed in
\eqref{eq:energy-absorption-S7}.
\end{proposition}

\begin{proof}
For general finite values of the integrals and of \(P\), direct subtraction gives \[
J_{c,\eta}-D_2=2A+2\eta Q_m+4\eta Q_2
=\int t+\eta\int(pt+t^2)=2A_\eta.
\] Also \[
-2D_1=\int(p-t-u_3^2)+\frac\eta2\int(3p^2+2pt-t^2),
\] which rearranges to the equality in \eqref{eq:energy-absorption-S4}. Since \(p,t\ge0\), \[
2A_\eta-\int(t+\eta t^2/2-\eta pt)
=2\eta\int pt+\frac\eta2\int t^2\ge0.
\] Every integral is finite by \(E_\eta<\infty\), including \(\eta\int pt\le(\eta/4)\int(p+t)^2\). This proves \eqref{eq:energy-absorption-S4}.

By \eqref{eq:energy-absorption-B}, \(E_\eta\le\frac12\int F_\eta+A_\eta\). Insert \eqref{eq:energy-absorption-S2}, then \eqref{eq:energy-absorption-S4}, obtaining \[
E_\eta\le(a_K+1+C_K/2)A_\eta-a_KD_1
+\frac{C_K}{2}(i_\theta+\varepsilon^2).
\] Replace \(A_\eta\) by \((J_{c,\eta}-D_2)/2\) to prove \eqref{eq:energy-absorption-S5}. Replacing \(J\) by \(J_+\), and the signed scalar errors by their absolute values, proves \eqref{eq:energy-absorption-S6}.

For \eqref{eq:energy-absorption-S8}, put \(c_-=\min K>0\). From \(cP=J-E_\eta\) and \(|J|\le M\), \[
W_\eta^2=1+E_\eta+\sqrt{1+P^2}
\le2+E_\eta+\frac{M+E_\eta}{c_-}
\le C_1(K)(1+M+E_\eta).
\] Equations \eqref{eq:energy-absorption-S6} and \eqref{eq:energy-absorption-S7} give \[
E_\eta\le C_2(K)\{M+I+\varepsilon^2+
\varepsilon\sqrt{1+M+E_\eta}\}.
\] Use \(\sqrt{1+M+E_\eta}\le\sqrt{1+M}+\sqrt{E_\eta}\), \(C_2\varepsilon\sqrt{E_\eta}\le E_\eta/2+C_2^2\varepsilon^2/2\), \(\varepsilon\le1\), and \(\sqrt{1+M}\le1+M\), then move \(E_\eta/2\) to the left. This gives \eqref{eq:energy-absorption-S8}. In particular the weighted scalar residuals may be absorbed before removing \(\eta\).
\end{proof}

\subsection{Whole-plane stress identities}

For exact critical maps, compactly supported tests of the stress tensor give the whole-plane dilation identities. They also justify a spatial cutoff limit for the momentum, allowing the preceding energy estimate to be applied with zero residual and scaling errors.

\begin{proposition}[Cutoff momentum and stress identities]\label{prop:energy-stress}
Let $c\in K$ and $\eta>0$. For a smooth exact critical map of finite $E_\eta$, the cutoff momentum
\begin{equation}\label{eq:energy-absorption-S9}
P(u)=\lim_{R\to\infty}\int x_2\chi(x/R)q(u)
\end{equation}
exists and is independent of the smooth compact cutoff $\chi$ equal to
one near zero. It gives $D_1=D_2=0$ and $J=2A_\eta=T_\eta$, hence
\begin{equation}\label{eq:energy-absorption-S10}
E_\eta(u)\le C_K''\bigl(J_{c,\eta}(u)+i_\theta(u;c,\eta)\bigr).
\end{equation}
The cutoff identities hold for every fixed $\eta>0$; the energy
estimate holds for $0<\eta\le\eta_K$ and $0<\theta\le\theta_K$, with the constants from \cref{prop:energy-signed}. Absolute integrability of $x_2 q(u)$ is not required.
\end{proposition}

\begin{proof}
For a smooth critical map define the \(2\)-by-\(2\) tensor \begin{equation}
T^\eta_{ij}=(1+\eta |\nabla u|^2)\,\partial_i u\cdot \partial_j u-e_\eta\delta_{ij},
\qquad i,j\in\{1,2\}.
\label{eq:energy-absorption-W1}
\end{equation} Here \(\delta_{ij}\) is the Kronecker symbol. Integration by parts in the residual says that \[
-\operatorname{div}((1+\eta |\nabla u|^2)\nabla u)+u_3\,e_3-c\,u\times \partial_1 u
\] is normal to the sphere. Taking its scalar product with \(u\), using \(u\cdot \partial_i u=0\) and \(u\cdot \partial_i^2u=-|\partial_i u|^2\), identifies its normal coefficient as \((1+\eta |\nabla u|^2)|\nabla u|^2+u_3^2\). Thus \begin{equation}
\operatorname{div}((1+\eta |\nabla u|^2)\nabla u)
+((1+\eta |\nabla u|^2)|\nabla u|^2+u_3^2)u-u_3\,e_3+c\,u\times \partial_1 u=0.
\label{eq:energy-absorption-W2}
\end{equation} Differentiating \eqref{eq:energy-absorption-W1} cancels the terms \((1+\eta |\nabla u|^2)\sum_i \partial_i u\cdot \partial_i\partial_j u\) against the gradient part of \(\partial_j e_\eta\). Therefore \begin{equation}
\partial_i T^\eta_{ij}
=(\operatorname{div}((1+\eta |\nabla u|^2)\nabla u)-u_3\,e_3)\cdot \partial_j u
=-c\,\delta_{j2}q.
\label{eq:energy-absorption-W3}
\end{equation} The sum over \(i\in\{1,2\}\) is understood. The last equality uses \eqref{eq:energy-absorption-W2} and \((u\times \partial_1 u)\cdot \partial_2 u=q\).

Each component of \(T^\eta\) is bounded in absolute value by a fixed constant times \(e_\eta\), because \(|\partial_i u\cdot \partial_j u|\le |\nabla u|^2\). In particular it belongs to \(L^1\) for finite whole-plane energy. With the cutoff in \eqref{eq:energy-absorption-S9}, test \eqref{eq:energy-absorption-W3} using the compact vector field \(X=(x_1\chi_R,0)\). It gives \begin{equation}
\int T^\eta_{11}\chi_R+\int x_1\,T^\eta_{i1}\partial_i\chi_R=0.
\label{eq:energy-absorption-W4}
\end{equation} The second integral tends to zero. Indeed \(|x_1\nabla\chi_R|\) is bounded independently of \(R\), and its support lies in \(\{r_0R\le|x|\le r_1R\}\) for some \(0<r_0<r_1<\infty\) fixed by \(\chi\). Its absolute value is at most a constant times the \(E_\eta\) energy in this annulus. Dominated convergence in the first integral yields \(\int T^\eta_{11}=0\). Direct expansion gives \[
T^\eta_{11}=\frac{p-t-u_3^2}{2}
+\frac\eta4(3p^2+2pt-t^2),
\] so this identity is \(D_1=0\).

Next take \(X=(0,x_2\chi_R)\) in \eqref{eq:energy-absorption-W3}. Since integration by parts yields \(\int T^\eta_{ij}\partial_iX_j=c\int qX_2\), one obtains \begin{equation}
c\int x_2\chi_Rq=\int T^\eta_{22}\chi_R+
\int x_2\,T^\eta_{i2}\partial_i\chi_R.
\label{eq:energy-absorption-W5}
\end{equation} The same tail estimate eliminates the second term. Consequently the limit in \eqref{eq:energy-absorption-S9} exists and is \(c^{-1}\int T^\eta_{22}\), independently of \(\chi\). Since \[
T^\eta_{22}=\frac{t-p-u_3^2}{2}
+\frac\eta4(-p^2+2pt+3t^2),
\] one has \(D_2=-\int T^\eta_{22}+cP=0\). Equations \eqref{eq:energy-absorption-S4} now yield \(J=2A_\eta\). Apply \eqref{eq:energy-absorption-S6} with zero residual and zero scalar errors to obtain \eqref{eq:energy-absorption-S10}.

For completeness this cutoff prescription is compatible with the dilation and compact variation formulas already used. In the integral for \(P(u_{\alpha,\beta})\), change variables to the original coordinates. The new cutoff is \(\widetilde\chi(x_1,x_2)=\chi(e^\alpha x_1,e^\beta x_2)\), which is still smooth, compact, and one near zero. Independence of the original limit from the cutoff gives \(P(u_{\alpha,\beta})=e^\beta P(u)\). The six other formulas in \eqref{eq:energy-absorption-V3} use absolutely integrable densities and therefore hold without cutoffs. This proves the claimed dilation derivatives even though the scaled maps need not themselves be critical. If \(u_\tau-u\) has one common compact support, then its momentum differs by the ordinary compact integral \(\int x_2(q(u_\tau)-q(u))\); this proves \eqref{eq:energy-absorption-V1}-\eqref{eq:energy-absorption-V2} for these maps as well. Thus the stress identities follow from compactly supported tests
and integrability alone, without a prescribed decay rate for derivatives.
\end{proof}

\section{Recovery of an exact regularised wave}\label{app:fixed-eta-recovery}

We now keep $\eta>0$ fixed and pass from the selected sequence to
an exact critical point. The quartic energy gives a uniform modulus
of continuity, while monotonicity of the gradient operator gives
strong local convergence. Translating a point of positive height
preserves nonconstancy in the limit.

Fix $\eta>0$, $M<\infty$, and $h_*>0$, and use the
notation of \cref{sec:planar-framework}. Suppose \(u_n\in C^\infty(\mathbb R^2;S^2)\), \(c_n\in[0,1]\), \(c_n\to c\in(0,1)\), and \(\varepsilon_n\to0\), \(\varepsilon_n\ge0\), satisfy \begin{equation}
E_\eta(u_n)\le M,\qquad
|R_{c_n,\eta,u_n}[w]|\le\varepsilon_n B_{\eta,u_n}(w)
\quad\hbox{for every smooth compact tangent }w.
\label{eq:fixed-eta-recovery-1}
\end{equation} Suppose points \(b_n\in\mathbb R^2\) satisfy \begin{equation}
|(u_n)_3(b_n)|\ge h_* .
\label{eq:fixed-eta-recovery-2}
\end{equation} 
\begin{proposition}[Recovery at fixed regularisation]\label{prop:fixed-eta-recovery}
A subsequence of \(v_n(x)=u_n(x+b_n)\) converges to a continuous sphere map \(u\) with \begin{equation}
v_n\to u\quad\hbox{locally uniformly and strongly in }W^{1,4}_{\rm loc},
\qquad
v_n\to u\quad\hbox{strongly in }W^{1,2}_{\rm loc}.
\label{eq:fixed-eta-recovery-3}
\end{equation} The limit has \(E_\eta(u)\le M\), \(|u_3(0)|\ge h_*\), is nonconstant, and solves \begin{equation}
-\int(1+\eta|\nabla u|^2)\nabla u:\nabla\Phi
+\int\bigl((1+\eta|\nabla u|^2)|\nabla u|^2+u_3^2\bigr)u\cdot\Phi
-\int u_3e_3\cdot\Phi
+c\int(u\times \partial_1 u)\cdot\Phi=0
\label{eq:fixed-eta-recovery-4}
\end{equation} for every \(\Phi\in C_c^\infty(\mathbb R^2;\mathbb R^3)\). There is a numerical constant \(C\) such that every \(v_n\) and \(u\) has \begin{equation}
|v(x)-v(y)|\le C(4M/\eta)^{1/4}|x-y|^{1/2}.
\label{eq:fixed-eta-recovery-5}
\end{equation} For every subsequence used in \eqref{eq:fixed-eta-recovery-3}, \begin{equation}
T_\eta(u)\le\liminf_nT_\eta(u_n).
\label{eq:fixed-eta-recovery-6}
\end{equation}
\end{proposition}

Throughout this section set $v_n(x)=u_n(x+b_n)$; all subsequences below are taken from these translated maps.
\subsection{A uniform modulus and a weak subsequence}\label{appsrc-fixed-eta-recovery-a-uniform-modulus-and-a-weak-subsequence}

At fixed $\eta$, the quartic term gives a common modulus of continuity. Translating a point of positive height to the origin therefore preserves that height under local uniform convergence.

\begin{lemma}[Local uniform and weak convergence]\label{lem:recovery-weak}
A subsequence of $v_n$ converges locally uniformly and weakly in $W^{1,2}_{\rm loc}\cap W^{1,4}_{\rm loc}$ to a continuous sphere map $u$ with $E_\eta(u)\le M$ and $|u_3(0)|\ge h_*$. The common modulus \eqref{eq:fixed-eta-recovery-5} holds. Moreover, for any map $v$ with $E_\eta(v)\le M$ and any tangent field $w$ of finite $N_\eta$ norm,
\begin{equation}
B_{\eta,v}(w)
\le \left(1+(\sqrt2+2)\sqrt M+(4M)^{1/4}\right)N_\eta(w).
\label{eq:fixed-eta-recovery-13}
\end{equation}
\end{lemma}

\begin{proof}
For a smooth scalar function \(f\) on the plane with \(\nabla f\in L^4\), the homogeneous Morrey estimate is \begin{equation}
|f(x)-f(y)|\le C\|\nabla f\|_4|x-y|^{1/2}.
\label{eq:fixed-eta-recovery-11}
\end{equation} Its componentwise application gives the vector estimate.

Translation leaves every hypothesis invariant. We therefore work with the translated maps \(v_n\) from the outset. The bound \(E_\eta(v_n)\le M\) gives \begin{equation}
\|\nabla v_n\|_2\le\sqrt{2M},\quad
\|\nabla v_n\|_4\le(4M/\eta)^{1/4},\quad
\|(v_n)_3\|_2\le\sqrt{2M}.
\label{eq:fixed-eta-recovery-12}
\end{equation} Together with \(|v_n|=1\), \eqref{eq:fixed-eta-recovery-11} gives equiboundedness and a common modulus of continuity on every disk. A finite-net argument and diagonal subsequence give local uniform convergence to a continuous \(u\). Reflexive weak compactness for the derivatives in \(L^2\) and \(L^4\), or locally on increasing disks, identifies their weak limits with \(\nabla u\) by integration against smooth test functions. Thus \(u\in W^{1,4}_{\rm loc}\), \(|u|=1\), and \(|u_3(0)|\ge h_*\). The weak lower semicontinuity of the convex derivative integrals and of the \(L^2\) height integral, first on each disk and then as the disks increase, yields \(E_\eta(u)\le M\). Passing to the uniform limit in \eqref{eq:fixed-eta-recovery-11} gives \eqref{eq:fixed-eta-recovery-5}. These arguments only use the fixed \(\eta\) bound in \eqref{eq:fixed-eta-recovery-12}.

To prove \eqref{eq:fixed-eta-recovery-13}, let $v$ have energy at most $M$. The first extra term in $B_{\eta,v}(w)$ is at most \(\sqrt{2M}\|w\|_\infty\). Its second is at most \(2\sqrt M\|w\|_\infty\). For its third, Hölder gives \[
\sqrt\eta\||\nabla v|\,\nabla w\|_2
\le(\eta\int |\nabla v|^4)^{1/4}\eta^{1/4}\|\nabla w\|_4
\le(4M)^{1/4}\eta^{1/4}\|\nabla w\|_4.
\] This proves \eqref{eq:fixed-eta-recovery-13}.
\end{proof}

\subsection{Extension to Sobolev tangent fields}\label{appsrc-fixed-eta-recovery-legality-of-the-tangent-tests}

To test the equation against differences involving the weak limit, we first extend the residual bound to compactly supported Sobolev tangent fields.

\begin{lemma}[Sobolev tangent tests]\label{lem:recovery-tests}
For each smooth $v_n$, the residual bound in \eqref{eq:fixed-eta-recovery-1} extends to all compactly supported $W^{1,4}$ tangent fields.
\end{lemma}

\begin{proof}
For a fixed smooth base map \(v_n\), any compactly supported \(W^{1,4}\) tangent field can be approximated by smooth compact tangent fields in \(N_\eta\) and in \(B_{\eta,v_n}\). To see this, extend the compact field by zero and mollify it, then multiply by a smooth cutoff equal to one on its original support and finally project by \(\Pi_{v_n}a=a-(v_n\cdot a)v_n\). On the fixed enlarged compact set, mollification converges in \(W^{1,4}\) and uniformly; uniform convergence follows either from the continuous compact representative or from \eqref{eq:fixed-eta-recovery-11} applied to the derivative difference after a fixed extension. Projection preserves smoothness and compact support, converges in these norms by the product rule, and its limit is the original tangent field. For example, \[
|\nabla\Pi_{v_n}a|\le C(|\nabla a|+|\nabla v_n||a|).
\] Convergence in \(L^4\) and uniformly on this compact set also implies the derivative \(L^2\) and height \(L^2\) convergence. Estimate \eqref{eq:fixed-eta-recovery-13} gives convergence in \(B\). The residual functional is continuous under this convergence: its highest derivative coefficient is \(\eta|\nabla v_n|^2\nabla v_n\in L^{4/3}\), and all remaining terms follow from \(L^2\) pairing or uniform convergence on the compact set. Hence \eqref{eq:fixed-eta-recovery-1} extends to all these fields.
\end{proof}

\subsection{Strong local gradients from monotonicity}\label{appsrc-fixed-eta-recovery-strong-local-gradients-from-monotonicity}

The admissible tests can now compare each map with its weak limit. Monotonicity of the regularised gradient operator upgrades weak convergence to strong convergence of the gradients.

\begin{lemma}[Strong convergence of the gradients]\label{lem:recovery-strong}
The subsequence in \cref{lem:recovery-weak} converges strongly in $W^{1,2}_{\rm loc}$ and $W^{1,4}_{\rm loc}$, and satisfies \eqref{eq:fixed-eta-recovery-6}.
\end{lemma}

\begin{proof}
Fix a smooth nonnegative compact cutoff \(\chi\), and write \[
d_n=v_n-u,\quad r_n=v_n\cdot d_n=1-v_n\cdot u=\tfrac12|d_n|^2,\quad
w_n=\chi\Pi_{v_n}d_n=\chi(d_n-r_nv_n).
\] On the support of \(\chi\), \(d_n\to0\) uniformly. The fields \(w_n\) are admissible by \cref{appsrc-fixed-eta-recovery-legality-of-the-tangent-tests}. The projection product estimate shows that their \(N_\eta\) norms are uniformly bounded: \(\nabla d_n\) is bounded in \(L^2\cap L^4\) on the support, \(\nabla v_n\) obeys \eqref{eq:fixed-eta-recovery-12}, and the undifferentiated \(d_n\)'s are bounded there. Their heights have compact support and uniformly bounded supremum. By \eqref{eq:fixed-eta-recovery-13}, \(B_{\eta,v_n}(w_n)\) is uniformly bounded. Therefore \(R_{c_n,\eta,v_n}[w_n]\to0\).

Put \[
\mathcal A(P)=(1+\eta|P|^2)P
\] for a real \(3\)-by-\(2\) matrix \(P\), and \(\mathcal A_n=\mathcal A(\nabla v_n)\). Every column of \(\mathcal A_n\) is orthogonal to \(v_n\). Differentiating the displayed expression for \(w_n\) thus gives the exact identity \begin{equation}
\mathcal A_n:\nabla w_n
=\chi\mathcal A_n:(\nabla v_n-\nabla u)
-\chi r_n(|\nabla v_n|^2+\eta |\nabla v_n|^4)
+\mathcal A_n:\bigl((d_n-r_nv_n)\otimes\nabla\chi\bigr).
\label{eq:fixed-eta-recovery-14}
\end{equation} The term involving \(v_n\otimes\nabla r_n\) is zero by orthogonality.

The integral of the second term in \eqref{eq:fixed-eta-recovery-14} tends to zero because \(\|r_n\|_{\infty,\operatorname{supp}\chi}\to0\) and \(\int(|\nabla v_n|^2+\eta |\nabla v_n|^4)\le6M\). The last term tends to zero because \(d_n\to0\) uniformly and \(\mathcal A_n\) is bounded in \(L^1\) on this compact set; this follows from \eqref{eq:fixed-eta-recovery-12} and Hölder. The height term in the residual tends to zero using \(\|w_n\|_\infty\to0\), the compact support, and the uniform \(L^2\) height bound. Its drift term tends to zero using \(|c_n|\le1\), \(\|w_n\|_\infty\to0\), and the local \(L^1\) bound for \(\nabla v_n\). Equation \eqref{eq:fixed-eta-recovery-14} and the vanishing residual therefore give \begin{equation}
\int\chi\,\mathcal A(\nabla v_n):(\nabla v_n-\nabla u)\longrightarrow0.
\label{eq:fixed-eta-recovery-15}
\end{equation} Weak convergence of the gradients gives \begin{equation}
\int\chi\,\mathcal A(\nabla u):(\nabla v_n-\nabla u)\longrightarrow0,
\label{eq:fixed-eta-recovery-16}
\end{equation} since \(\chi\mathcal A(\nabla u)\in L^{4/3}\) on a compact set. Subtract \eqref{eq:fixed-eta-recovery-16} from \eqref{eq:fixed-eta-recovery-15}.

For any matrices \(P,Q\), set \(m=(P+Q)/2\), \(d=P-Q\). Expansion in their Euclidean matrix inner product gives \[
(|P|^2P-|Q|^2Q):d
=(|m|^2+|d|^2/4)|d|^2+2(m:d)^2
\ge|d|^4/4.
\] Consequently \begin{equation}
(\mathcal A(P)-\mathcal A(Q)):(P-Q)
\ge |P-Q|^2+\frac\eta4|P-Q|^4.
\label{eq:fixed-eta-recovery-17}
\end{equation} The integral in \eqref{eq:fixed-eta-recovery-15} minus \eqref{eq:fixed-eta-recovery-16} tends to zero and dominates the integral of the right side of \eqref{eq:fixed-eta-recovery-17} times \(\chi\). Choosing \(\chi=1\) on any prescribed compact set proves strong local \(L^2\) and \(L^4\) convergence of the gradients. Together with uniform convergence, this is \eqref{eq:fixed-eta-recovery-3}.

For later use, every polynomial expression of total derivative degree at most four with uniformly convergent bounded map coefficients now converges in \(L^1\) on a compact set. In particular \(T_\eta(v_n;B_R)\to T_\eta(u;B_R)\) for every fixed disk. Its density is nonnegative. Hence \(T_\eta(u;B_R)\le\liminf_nT_\eta(u_n)\); increase \(R\) to obtain \eqref{eq:fixed-eta-recovery-6}.
\end{proof}

\subsection{Exact weak criticality and nonconstancy}\label{appsrc-fixed-eta-recovery-exact-weak-criticality-and-nonconstancy}

Strong local convergence now permits passage to the full weak equation, while the height at the origin guarantees nonconstancy.

\begin{proof}[Proof of \cref{prop:fixed-eta-recovery}]
Apply \cref{lem:recovery-weak,lem:recovery-strong}. It remains to pass to the equation and verify nonconstancy.

Let \(\Phi\in C_c^\infty(\mathbb R^2;\mathbb R^3)\) and take \(w_n=\Pi_{v_n}\Phi\), a smooth compact tangent field. Its \(N_\eta\), and then its \(B_{\eta,v_n}\), norm is uniformly bounded by the projection estimate, \eqref{eq:fixed-eta-recovery-12}, and \eqref{eq:fixed-eta-recovery-13}. Thus its residual tends to zero. Using \(\partial_jv_n\cdot v_n=0\) gives \[
\begin{aligned}
R_{c_n,\eta,v_n}[\Pi_{v_n}\Phi]
={}&\int(1+\eta |\nabla v_n|^2)\nabla v_n:\nabla\Phi\\
&-\int\bigl((1+\eta |\nabla v_n|^2)|\nabla v_n|^2+(v_n)_3^2\bigr)v_n\cdot\Phi\\
&+\int(v_n)_3\Phi_3-c_n\int(v_n\times\partial_1 v_n)\cdot\Phi.
\end{aligned}
\] Every integral passes to its counterpart for \(u\) by \eqref{eq:fixed-eta-recovery-3} and Hölder, including the quartic term. Multiplication by minus one gives \eqref{eq:fixed-eta-recovery-4}. Equation \eqref{eq:fixed-eta-recovery-4} also implies zero tangent residual for compact \(W^{1,4}\) tangent tests by the density argument and its local integrability estimates.

A constant sphere map of finite \(E_\eta\) on the plane has third component zero, since otherwise its height square has infinite integral. Our limit has \(|u_3(0)|\ge h_*>0\). It is therefore nonconstant.
\end{proof}

\section{Regularity of the quartic equation}\label{app:weak-regularity}

The weak critical point obtained in the preceding section is
smooth. We prove this local assertion directly for the regularised
system. The constants may depend on the fixed positive parameter
$\eta$; uniform estimates for its removal are proved separately.

Let \(\Omega\subset\mathbb R^2\) be open, and fix constants
\(\eta>0\) and \(c\in\mathbb R\). Suppose \(u\in W^{1,4}_{\mathrm{loc}}(\Omega;S^2)\) satisfies, for every \(\varphi\in C_c^\infty(\Omega;\mathbb R^3)\), \begin{equation}
\int_\Omega(1+\eta|\nabla u|^2)\nabla u:\nabla\varphi
 =\int_\Omega\bigl[(1+\eta|\nabla u|^2)|\nabla u|^2u
       +u_3^2u-u_3e_3+c\,u\times\partial_1u\bigr]\cdot\varphi.
\label{eq:weak-regularity-1}
\end{equation} 

The alternative first-variation formulation is as follows. For every \(\varphi\in C_c^\infty(\Omega;\mathbb R^3)\), set \[
 \xi=\varphi-(u\cdot\varphi)u
   =\left.\frac{d}{dt}\right|_{t=0}\frac{u+t\varphi}{|u+t\varphi|}.
\] Require \begin{equation}
\int_\Omega(1+\eta|\nabla u|^2)\nabla u:\nabla\xi
       +u_3\xi_3+c\,u\cdot(\xi\times\partial_1u)=0.
\label{eq:weak-regularity-2}
\end{equation} All terms in \eqref{eq:weak-regularity-1} and \eqref{eq:weak-regularity-2} are locally integrable.
\begin{proposition}[Interior smoothness]\label{prop:weak-regularity}
The formulations \eqref{eq:weak-regularity-1} and \eqref{eq:weak-regularity-2} are equivalent. Every map satisfying either formulation has a representative in $C^\infty(\Omega;S^2)$.
\end{proposition}

Constants below may depend on $\eta>0$, $|c|$ and the map on a relatively compact neighbourhood.
\subsection{Homogeneous-system estimates}\label{appsrc-weak-regularity-the-external-homogeneous-system-input-and-its-applicability}

We begin with estimates for the homogeneous quartic system, which serves as the local comparison equation.

We use Uhlenbeck's homogeneous-system regularity theorem
\cite[p.~222 and Theorems~3.2 and~5.4]{Uhlenbeck1977} in the following
specialised form.

\begin{lemma}[Homogeneous quartic estimates]\label{lem:regularity-homogeneous}
For \(k>0\), if \(h\in W^{1,4}(B_R;\mathbb R^3)\) solves \begin{equation}
\operatorname{div}((k+|\nabla h|^2)\nabla h)=0,
\label{eq:weak-regularity-7}
\end{equation} there are \(\gamma\in(0,1)\) and \(C<\infty\), depending only on the fixed dimensions and growth constants, such that \begin{equation}
\sup_{B_{R/2}}|\nabla h|^4
       +R^{4\gamma}[\nabla h]_{C^{0,\gamma}(B_{R/2})}^4
 \le C\fint_{B_R}(k^2+|\nabla h|^4).
\label{eq:weak-regularity-8}
\end{equation} The constants are independent of \(k\). Changing the fixed concentric radius ratios changes only \(C\). In particular, for \(0<r\le R/2\), \begin{equation}
\begin{aligned}
 \int_{B_r}|\nabla h|^4
 &\le C(r/R)^2\int_{B_R}(k^2+|\nabla h|^4),\\
 \int_{B_r}|\nabla h-(\nabla h)_{B_r}|^4
 &\le C(r/R)^{2+4\gamma}\int_{B_R}(k^2+|\nabla h|^4).
 \end{aligned}
\label{eq:weak-regularity-9}\end{equation} The notation \(f_B\) in an average is \(|B|^{-1}\int_B f\); \([\cdot]_{C^{0,\gamma}}\) is the Hölder seminorm.
\end{lemma}

\begin{proof}
In Uhlenbeck's notation, the unknown \(\omega\) is the vector-valued one-form \(dh\), its \(Q\) is \(|\nabla h|^2\), its scalar growth exponent \(p\) is \textbf{one} here (the energy exponent is \(2p+2=4\)), and its shift \(C\) is our \(k\). Choose \(\rho(Q)=k+Q\). The source's conditions (1.3) and (1.4) hold with a fixed growth constant, because \[
 k+Q\le\rho(Q)+2Q\rho'(Q)=k+3Q\le3(k+Q),
 \quad |Q_1\rho'(Q_1)-Q_2\rho'(Q_2)|=|Q_1-Q_2|.
\] The complex is the gradient followed by the componentwise exterior derivative; its Laplacian identity (1.2) is the Euclidean one. The source covers arbitrary weak homogeneous solutions in their natural energy space, not merely smooth maps. Theorem 3.2 bounds a quantity comparable to \((Q+k)^2\), and Theorem 5.4 bounds the Hölder seminorm of \(\omega\) by a constant times \((\sup Q+k)^{1/2}\). Applying them on nested fixed-radius disks gives \eqref{eq:weak-regularity-8}; integrating its supremum and oscillation estimates gives \eqref{eq:weak-regularity-9}.
\end{proof}

\subsection{The weak equation and comparison maps}\label{appsrc-weak-regularity-the-weak-equation-test-function-extension-and-comparison-maps}

On each small disk, we compare $u$ with the solution of the homogeneous equation having the same boundary values. The difference between them is controlled by the oscillation of $u$ and the lower-order terms in its equation.

Set \(P=\nabla u\) and \[
 \mathcal A(P)=(1+\eta|P|^2)P,\qquad
 f=(1+\eta|\nabla u|^2)|\nabla u|^2u+u_3^2u-u_3\,e_3+c\,u\times\partial_1u.
\] Thus \eqref{eq:weak-regularity-1} says \(\int\mathcal A(\nabla u):\nabla\varphi=\int f\cdot\varphi\), and \begin{equation}
|f|\le C_{\eta,c}(1+|\nabla u|^4).
\label{eq:weak-regularity-10}
\end{equation} In two dimensions \(W^{1,4}\) embeds into \(C^{0,1/2}\). In particular, \(u\) has a continuous sphere-valued representative. On every relatively compact disk, smooth compactly supported functions are dense in \(W^{1,4}_0\), and convergence in that norm implies uniform convergence. Since \(\mathcal A(\nabla u)\in L^{4/3}\) and \(f\in L^1\), the weak equation extends to every \(W^{1,4}_0\) test on that disk.

Fix a compact subset of \(\Omega\) and a larger relatively compact open neighbourhood. All disks used below have centres in the compact subset and are contained in the neighbourhood. There are finite constants \(H\ge1\) and \(R_*\in(0,1)\) such that \begin{equation}
|u(x)-u(y)|\le H|x-y|^{1/2}
\label{eq:weak-regularity-11}
\end{equation} on each of these disks of radius at most \(R_*\). For such a disk \(B_R=B_R(x_0)\), let \(h\) minimise \begin{equation}
\mathcal F_\eta(v;B_R)=\int_{B_R}\left(\frac12|\nabla v|^2+\frac\eta4|\nabla v|^4\right)
\label{eq:weak-regularity-12}
\end{equation} over \(v\in u+W^{1,4}_0(B_R;\mathbb R^3)\). This is an unconstrained Euclidean comparison map. Coercivity, weak lower semicontinuity and strict convexity give a unique minimiser. Its weak equation is \eqref{eq:weak-regularity-7} with \(k=\eta^{-1}\), after dividing by \(\eta\).

\begin{lemma}[Comparison with the homogeneous solution]\label{lem:regularity-comparison}
The comparison map $h$ satisfies
\begin{equation}
\|u-h\|_{L^\infty(B_R)}\le C H R^{1/2}.
\label{eq:weak-regularity-13}
\end{equation}
It also obeys the comparison and energy bounds
\begin{equation}
\int_{B_R}|\nabla u-\nabla h|^4
 \le C R^{1/2}\Phi_{x_0}(R),\qquad
 \Phi_{x_0}(R)=\int_{B_R(x_0)}(1+|\nabla u|^4),
\label{eq:weak-regularity-15}
\end{equation}
and
\begin{equation}
\int_{B_R}(\eta^{-2}+|\nabla h|^4)
 \le C_\eta\Phi_{x_0}(R).
\label{eq:weak-regularity-16}
\end{equation}
The constants are uniform over the chosen centres and sufficiently small radii; they may depend on $\eta,c,H$.
\end{lemma}

\begin{proof}
Each component of \(h\) stays between the minimum and maximum of the corresponding component of \(u\) on \(\partial B_R\). For example, the positive part of \(h_j-M_j\), where \(M_j=\max_{\partial B_R}u_j\), has zero trace and is an admissible test in its component equation. That equation gives \[
 \int_{\{h_j>M_j\}}(1+\eta|\nabla h|^2)|\nabla h_j|^2=0.
\] The tested function is therefore zero by the zero-trace Poincaré inequality. Testing the negative part gives the lower bound. Continuity up to the disk boundary is supplied by the same \(W^{1,4}\) embedding. Applying \eqref{eq:weak-regularity-11} component by component proves \eqref{eq:weak-regularity-13}.

For arbitrary real matrices \(P,Q\), \begin{equation}
(\mathcal A(P)-\mathcal A(Q)):(P-Q)
 \ge |P-Q|^2+\frac\eta4|P-Q|^4.
\label{eq:weak-regularity-14}
\end{equation} Indeed the cubic part satisfies \[
 (|P|^2P-|Q|^2Q):(P-Q)
 \ge \tfrac12(|P|^2+|Q|^2)|P-Q|^2
 \ge\tfrac14|P-Q|^4.
\] The difference in the first inequality is \((|P|^2-|Q|^2)^2/2\), and the second uses \(|P-Q|^2\le2(|P|^2+|Q|^2)\).

Test both equations with \(u-h\), subtract, and use \eqref{eq:weak-regularity-10}, \eqref{eq:weak-regularity-13}, and \eqref{eq:weak-regularity-14}. With constants uniform over the chosen centres and radii, but allowed to depend on \(\eta,c,H\), this gives \eqref{eq:weak-regularity-15}. Minimality in \eqref{eq:weak-regularity-12} and \(2\eta^{-1}|P|^2\le|P|^4+\eta^{-2}\) also give \eqref{eq:weak-regularity-16}.
\end{proof}

\subsection{Energy decay to every exponent below the dimension}\label{appsrc-weak-regularity-energy-decay-to-every-exponent-below-the-dimension}

Combining the comparison estimate with decay for the homogeneous solution yields decay of the local quartic energy.

\begin{lemma}[Decay of the local quartic energy]\label{lem:regularity-decay}
For each $\delta\in(0,2)$ there are $R_0,C_\delta>0$, uniform over centres in a fixed compact subset, such that $\Phi_{x_0}(R)\le C_\delta R^{2-\delta}$ for $0<R\le R_0$, where $\Phi_{x_0}$ is defined in \eqref{eq:weak-regularity-15}.
\end{lemma}

\begin{proof}
For \(0<r\le R/2\), use \(|P+Q|^4\le8(|P|^4+|Q|^4)\), \eqref{eq:weak-regularity-9}, \eqref{eq:weak-regularity-15}, and \eqref{eq:weak-regularity-16}. The area term in \(\Phi_{x_0}(r)\) is also bounded by \((r/R)^2\Phi_{x_0}(R)\). Hence \begin{equation}
\Phi_{x_0}(r)\le C\left[(r/R)^2+R^{1/2}\right]\Phi_{x_0}(R).
\label{eq:weak-regularity-17}
\end{equation} Fix \(\delta\in(0,2)\). Choose \(\theta\in(0,1/2)\) with \(C\theta^2\le\frac12\theta^{2-\delta}\), then choose \(R_0\in(0,R_*)\) with \(CR_0^{1/2}\le\frac12\theta^{2-\delta}\). Formula \eqref{eq:weak-regularity-17} implies \[
 \Phi_{x_0}(\theta R)\le\theta^{2-\delta}\Phi_{x_0}(R)
 \quad(0<R\le R_0).
\] Iteration at \(R_0,\theta R_0,\theta^2R_0,\ldots\), followed by monotonicity of \(\Phi_{x_0}\) between these radii, proves \begin{equation}
\Phi_{x_0}(R)\le C_\delta R^{2-\delta}\quad(0<R\le R_0).
\label{eq:weak-regularity-18}
\end{equation} The constants are uniform over the chosen centres, since their \(R_0\)-disk energies are bounded by the \(W^{1,4}\) norm on the fixed larger neighbourhood.
\end{proof}

\subsection{Hölder continuity of the gradient}\label{appsrc-weak-regularity-huxf6lder-continuity-of-the-gradient}

We now use the energy decay to estimate the oscillation of the gradient. The resulting bound gives a H\"older continuous representative.

\begin{lemma}[Hölder continuity of the gradient]\label{lem:regularity-holder}
The weak solution belongs to $C^{1,\beta}_{\rm loc}(\Omega)$ for some $\beta>0$.
\end{lemma}

\begin{proof}
For an integrable matrix field \(W\), write \((W)_B=|B|^{-1}\int_BW\). Define \[
 \Psi_{x_0}(r)=\int_{B_r(x_0)}|\nabla u-(\nabla u)_{B_r(x_0)}|^4.
\] Subtracting the corresponding mean for \(\nabla h\), using Jensen's inequality for the difference of the means, and applying \eqref{eq:weak-regularity-9}, \eqref{eq:weak-regularity-15}, \eqref{eq:weak-regularity-16}, and \eqref{eq:weak-regularity-18}, we obtain \begin{equation}
\begin{aligned}
 \Psi_{x_0}(r)
 &\le C\int_{B_R}|\nabla u-\nabla h|^4
       +C\int_{B_r}|\nabla h-(\nabla h)_{B_r}|^4\\
 &\le C\left[R^{2+1/2-\delta}
       +r^{2+4\gamma}R^{-4\gamma-\delta}\right]
 \qquad(0<r\le R/2\le R_0/2). 
 \end{aligned}
\label{eq:weak-regularity-19}\end{equation} Choose \(\delta=\gamma/10\), and for sufficiently small \(r\) set \(R=r^{9/10}\). This choice satisfies \(2r\le R\le R_0\). The two powers of \(r\) on the right of \eqref{eq:weak-regularity-19} are \[
 \tfrac9{10}(\tfrac52-\tfrac\gamma{10})
       =\tfrac94-\tfrac{9\gamma}{100},\qquad
 2+4\gamma-\tfrac9{10}(4\gamma+\tfrac\gamma{10})
       =2+\tfrac{31\gamma}{100}.
\] Since \(0<\gamma<1\), both exceed \(2+\gamma/8\). Therefore \begin{equation}
\Psi_{x_0}(r)\le C r^{2+\gamma/8}=C r^{2+4\beta},
 \qquad \beta=\gamma/32>0.
\label{eq:weak-regularity-20}
\end{equation} This estimate is uniform over centres in the fixed compact subset.

For completeness, \eqref{eq:weak-regularity-20} gives a Hölder representative of \(\nabla u\) by the following mean-value argument. Hölder's inequality bounds the difference between its averages over \(B_r(x)\) and \(B_{r/2}(x)\) by \(Cr^\beta\). Summing this estimate over dyadic radii shows that the averages converge uniformly in \(x\) to a value \(V(x)\), with \(|V(x)-(\nabla u)_{B_r(x)}|\le Cr^\beta\). If \(|x-y|\le r\), compare both radius-\(r\) averages to the average on \(B_{3r}(x)\). Formula \eqref{eq:weak-regularity-20} and the fixed area ratios bound this difference by \(Cr^\beta\). Thus \(|V(x)-V(y)|\le C|x-y|^\beta\). At almost every point Lebesgue differentiation identifies \(V\) with the weak gradient. Local mollification of \(u\) now converges with first derivatives on smaller disks, so \(u\in C^{1,\beta}_{\mathrm{loc}}(\Omega)\).
\end{proof}

\subsection{Smooth bootstrap and equivalence of formulations}\label{appsrc-weak-regularity-smooth-bootstrap-and-equivalence-of-formulations}

With a H\"older continuous gradient, the equation has uniformly elliptic coefficients on compact sets, and standard bootstrapping gives smoothness. We also verify that the two weak formulations stated above are equivalent.

\begin{proof}[Proof of \cref{prop:weak-regularity}]
By \cref{lem:regularity-holder}, the weak solution has a locally Hölder continuous gradient. Once \(u\in C^{1,\beta}\), its gradient is bounded on compact subsets. The derivative of the flux is the symmetric matrix-space operator \[
 D\mathcal A(P)Z=(1+\eta|P|^2)Z+2\eta(P:Z)P.
\] It has eigenvalues between \(1\) and \(1+3\eta|P|^2\). It is locally bounded, uniformly positive, and \(C^{0,\beta}\) after evaluation at \(P=\nabla u\). The forcing \(f\), a smooth expression in \(u\) and \(\nabla u\), is locally \(C^{0,\beta}\).

Difference quotients of the divergence equation first give \(u\in W^{2,2}_{\mathrm{loc}}\). More explicitly, the difference quotient of \(\mathcal A(\nabla u)\) is the integral of \(D\mathcal A\) along the segment joining the two gradients, times their difference quotient. These coefficients have the preceding common ellipticity and boundedness constants. Test the difference equation with a cutoff squared times the difference quotient of \(u\). On the right, write the quotient of \(f\) as a distributional derivative of its bounded translated averages. Cauchy--Schwarz and absorption bound the first derivatives of the quotient locally in \(L^2\); the remaining cutoff terms are bounded by local bounds for \(\nabla u\) and \(f\). Passing to the limit proves the stated \(W^{2,2}\) regularity.

The weak derivative \(v_k=\partial_k u\) then satisfies \begin{equation}
\operatorname{div}\bigl(D\mathcal A(\nabla u)\nabla v_k\bigr)=\partial_k(-f)
\label{eq:weak-regularity-21}
\end{equation} with the sign chosen according to \(\operatorname{div}\mathcal A(\nabla u)=-f\). The right side is the divergence of a matrix field having the \(k\)-th column \(-f\) and the other column zero. The interior divergence-form Schauder estimate for a uniformly strongly elliptic linear system with \(C^{0,\beta}\) coefficients and \(C^{0,\beta}\) divergence data gives \(v_k\in C^{1,\beta}\), hence \(u\in C^{2,\beta}\). We then differentiate \eqref{eq:weak-regularity-21} repeatedly and apply the same interior linear estimates, with one higher order of coefficient and data regularity at each step. The coefficients and \(f\) are smooth functions of their arguments and the ellipticity lower bound is one throughout. This proves \(u\in C^\infty(\Omega)\).

It remains to check that variationally supplied maps fit \eqref{eq:weak-regularity-1}. Since \(|u|^2=1\) almost everywhere, the Sobolev chain rule gives \(u\cdot\partial_i u=0\). For \(\xi=\varphi-(u\cdot\varphi)u\), direct differentiation gives \[
 \nabla u:\nabla\xi
   =\nabla u:\nabla\varphi-|\nabla u|^2(u\cdot\varphi),
 \quad u_3\xi_3=u_3\varphi_3-u_3^2(u\cdot\varphi),
\] \[
 u\cdot(\xi\times\partial_1u)
   =u\cdot(\varphi\times\partial_1u)
   =-(u\times\partial_1u)\cdot\varphi.
\] Substitution shows that \eqref{eq:weak-regularity-2} is exactly \eqref{eq:weak-regularity-1}. These calculations are valid in the stated Sobolev spaces; the only quartic derivative term is locally \(L^1\), while the undifferentiated test factors are bounded. Normalised compact variations are legitimate since \(|u+t\varphi|\ge1-|t|\|\varphi\|_\infty>0\) for sufficiently small \(|t|\), and their first derivative in \(W^{1,4}\) is the displayed \(\xi\).
\end{proof}

\section{Decay of the quartic regularisation}\label{app:regularizer-decay}

We next let the regularisation parameter tend to zero along
exact smooth waves with bounded regularised energy. A scalar
differential inequality controls the coefficient seen after
rescaling. Monotonicity then gives strong convergence, and the
limiting stress identity rules out a nonzero quartic contribution.

\begin{lemma}[A scalar estimate at the concentration scale]\label{lem:scalar-gradient}
Let a smooth sphere map $v$ solve
\begin{equation}\label{eq:concentration-compactness-19}
\operatorname{div}((1+\alpha w)\nabla v)
=-(1+\alpha w)wv-\lambda^2v_3^2v+\lambda^2v_3\,e_3-dv\times\partial_1v,
\end{equation}
where $w=|\nabla v|^2$ and $\alpha,\lambda,d\ge0$.
Set $a=1+\alpha w$, $g_{ij}=\partial_i v\cdot\partial_jv$,
$B=aI+2\alpha g$, $C=B/a$, and $F=w+\alpha w^2/2$.
The eigenvalues of $C$ lie in $[1,3]$, and
\begin{equation}\label{eq:concentration-compactness-20}
-\operatorname{div}(C\nabla F)\le(4w+2\lambda^2+d^2)F.
\end{equation}
In particular, if $w\le4$ and $\lambda,d\le1$ on $B_1$, then
\begin{equation}\label{eq:concentration-compactness-22}
F(0)\le C_0\int_{B_1}F
\end{equation}
for an absolute constant $C_0$, independent of $\alpha$.
\end{lemma}

\begin{proof}
The eigenvalue claim follows from $0\le g\le wI$. Differentiation,
using $v\cdot\partial_kv=0$, gives
\begin{equation}\label{eq:concentration-compactness-21}
\begin{aligned}
\tfrac12\operatorname{div}(B\nabla w)
={}&a|D^2v|^2+\tfrac\alpha2|\nabla w|^2-aw^2
 -\lambda^2v_3^2w+\lambda^2|\nabla v_3|^2\\
&-d\sum_k\partial_kv\cdot(v\times\partial_{1k}v).
\end{aligned}
\end{equation}
The last term has absolute value at most
$a|D^2v|^2/2+d^2w/(2a)$. Discard the remaining nonnegative terms and use
$B\nabla w=C\nabla F$, $w\le F$, and $aw\le2F$ to obtain
\eqref{eq:concentration-compactness-20}.

For completeness, the mean-value estimate used here requires only the
displayed ellipticity and $-\operatorname{div}(C\nabla F)\le19F$.
Testing with $\zeta^2(F+\tau)^{p-1}$, $p\ge2$ and $\tau>0$, gives
\[
\int|\nabla(\zeta(F+\tau)^{p/2})|^2
\le C p^2\int(\zeta^2+|\nabla\zeta|^2)(F+\tau)^p.
\]
The disk Sobolev inequality and iteration at $p=2,4,8,\ldots$ yield
$\sup_{B_r}(F+\tau)\le C(R-r)^{-2}\|F+\tau\|_{L^2(B_R)}$
for $1/2\le r<R\le1$. Interpolation gives
\[
m(r)\le\theta m(R)+C_\theta(R-r)^{-4}\int_{B_1}(F+\tau),
\qquad m(r)=\sup_{B_r}(F+\tau).
\]
Iterate on radii increasing from $1/2$ to $3/4$ with gaps proportional
to $2^{-j}$ and choose $\theta<1/16$. The resulting geometric series
converges. Letting $\tau\downarrow0$ proves
\eqref{eq:concentration-compactness-22}.
\end{proof}

We now apply the scalar estimate to a sequence of exact regularised waves. The point is to exclude concentration at a scale where the quartic coefficient remains significant.

Let $\eta_n>0$ tend to zero and let smooth $u_n:\R^2\to S^2$ solve
\begin{equation}\label{eq:regularizer-decay-1}
\operatorname{div}((1+\eta_n|\nabla u_n|^2)\nabla u_n)
+(1+\eta_n|\nabla u_n|^2)|\nabla u_n|^2u_n
+(u_n)_3^2u_n-(u_n)_3e_3+c_nu_n\times\partial_1u_n=0,
\end{equation}
where $c_n\in[0,1]$. Assume
\begin{equation}\label{eq:regularizer-decay-2}
E(u_n)\le M,\qquad \eta_n\int|\nabla u_n|^4\le K.
\end{equation}

\begin{proposition}[Vanishing of the quartic term]\label{prop:regularizer-decay}
Under these assumptions,
\begin{equation}\label{eq:regularizer-decay-3}
\sup_{\R^2}\eta_n|\nabla u_n|^2\longrightarrow0,
\qquad \eta_n\int|\nabla u_n|^4\longrightarrow0.
\end{equation}
This conclusion uses exact smooth criticality, not just bounded energy.
\end{proposition}

\subsection{Rescaling and the effective coefficient}

Suppose the quantity $\eta_n|\nabla u_n|^2$ does not vanish uniformly. Rescaling at a point of large gradient normalises the gradient, and the scalar estimate bounds the coefficient of the quartic term in the rescaled equation.

\begin{lemma}[A bounded effective coefficient]\label{lem:decay-rescaling}
If $\sup\eta_n|\nabla u_n|^2$ does not tend to zero, a subsequence admits rescalings $v_n(y)=u_n(y_n+y/L_n)$ with $L_n\to\infty$, $|\nabla v_n(0)|=1$ and $|\nabla v_n|\le2$ on disks exhausting $\mathbb R^2$. Set
\[
\begin{gathered}
\alpha_n=\eta_nL_n^2,\qquad \lambda_n=L_n^{-1},\qquad d_n=c_n/L_n,\\
w_n=|\nabla v_n|^2,\qquad F_n=w_n+\alpha_nw_n^2/2.
\end{gathered}
\]
For some $\delta>0$, the coefficients $\alpha_n\ge\delta$ converge to a finite $\alpha>0$. The rescaled maps solve
\begin{equation}\label{eq:regularizer-decay-7}
\operatorname{div}((1+\alpha_nw_n)\nabla v_n)
=-(1+\alpha_nw_n)w_nv_n-\lambda_n^2(v_n)_3^2v_n
 +\lambda_n^2(v_n)_3e_3-d_nv_n\times\partial_1v_n=:f_n,
\end{equation}
and satisfy, with $C_0$ from \cref{lem:scalar-gradient},
\begin{equation}\label{eq:regularizer-decay-9}
\alpha_n\int_{B_1}w_n^2\le K,\qquad
1+\alpha_n/2\le C_0\int_{B_1}F_n
\le C_0\left((\pi K/\alpha_n)^{1/2}+K/2\right).
\end{equation}
\end{lemma}

\begin{proof}
Suppose the first limit in \eqref{eq:regularizer-decay-3} fails. After
extraction there are $x_n\in\R^2$ and $\delta>0$ with
$\eta_n|\nabla u_n(x_n)|^2\ge\delta$.
On $\overline{B_{1/2}(x_n)}$ maximise
$(1/2-|x-x_n|)|\nabla u_n(x)|$ at an interior point $y_n$, and set
\[
t_n=1/2-|y_n-x_n|,\qquad L_n=|\nabla u_n(y_n)|.
\]
Then $t_nL_n\ge|\nabla u_n(x_n)|/2\to\infty$,
$L_n\ge|\nabla u_n(x_n)|$, and
$|\nabla u_n|\le2L_n$ on $B_{t_n/2}(y_n)$. Use the rescalings and coefficients defined in the statement; in particular,
$\alpha_n=\eta_nL_n^2\ge\delta$.
On disks of radii $t_nL_n/2\to\infty$ these maps have
$|\nabla v_n(0)|=1$, $|\nabla v_n|\le2$, and solve
\eqref{eq:regularizer-decay-7},
with $w_n\le4$.
Planar scaling and \cref{lem:scalar-gradient} give
\eqref{eq:regularizer-decay-9}. For $\alpha_n\ge1$, its last expression is at most
$C_0(\sqrt{\pi K}+K/2)$. Thus $\alpha_n$ is bounded, and a subsequence
converges to $\alpha\in[\delta,A_K]$ for some finite $A_K$.
In particular $f_n$ is uniformly bounded on every fixed disk.
\end{proof}

\subsection{Strong convergence by monotonicity}

The bounds on the rescaled maps give a weak limit. The same monotonicity argument used at fixed regularisation yields strong convergence, while the scalar estimate ensures that the limit is nonconstant.

\begin{lemma}[A nonconstant rescaled limit]\label{lem:decay-limit}
The rescalings in \cref{lem:decay-rescaling} have a subsequence converging locally uniformly and strongly in $W^{1,2}_{\rm loc}\cap W^{1,4}_{\rm loc}$ to a nonconstant Lipschitz sphere map $Q$. With $\alpha$ as in that lemma,
\begin{equation}\label{eq:regularizer-decay-11}
\frac12\int|\nabla Q|^2\le M,\qquad
\alpha\int|\nabla Q|^4\le K,
\end{equation}
and
\begin{equation}\label{eq:regularizer-positive-limit}
0<\frac{1+\alpha/2}{C_0}
\le\int_{B_1}\left(|\nabla Q|^2+\frac\alpha2|\nabla Q|^4\right).
\end{equation}
\end{lemma}

\begin{proof}
The gradient bound and sphere constraint give, after diagonal extraction,
local uniform and weak $W^{1,4}_{\rm loc}$ convergence to a Lipschitz
sphere map $Q$, with $|\nabla Q|\le2$.
Put $\mathcal A_n(P)=(1+\alpha_n|P|^2)P$. For any nonnegative smooth
compact cutoff $\chi$, test \eqref{eq:regularizer-decay-7} by
$\chi(v_n-Q)$. This compact Lipschitz test is legitimate by smoothing
in $W^{1,4}$. The right side and the terms differentiating $\chi$ tend
to zero: $v_n-Q\to0$ uniformly there, while $f_n$ and
$\mathcal A_n(\nabla v_n)$ are bounded. Consequently
\[
\int\chi\mathcal A_n(\nabla v_n):
                  (\nabla v_n-\nabla Q)\longrightarrow0.
\]
Also $\mathcal A_n(\nabla Q)\to\mathcal A_\alpha(\nabla Q)$ strongly
in local $L^{4/3}$, so weak gradient convergence gives the same limit
after replacing $\mathcal A_n(\nabla v_n)$ by
$\mathcal A_n(\nabla Q)$. Subtract and use the elementary monotonicity
inequality \eqref{eq:fixed-eta-recovery-17}, with $\eta=\alpha_n\ge\delta$:
\[
\int\chi\left(|\nabla v_n-\nabla Q|^2
            +\frac\delta4|\nabla v_n-\nabla Q|^4\right)\longrightarrow0.
\]
Thus gradients converge strongly in local $L^2$ and $L^4$. Scaling the
global bounds and exhausting the plane gives
\eqref{eq:regularizer-decay-11}.
Passing to the limit in the integral in
\eqref{eq:regularizer-decay-9} gives
\eqref{eq:regularizer-positive-limit}, so $Q$ is nonconstant.
\end{proof}

\subsection{The limiting stress identity}

The rescaled drift and anisotropy terms disappear in the limit. Passing the stress identities through the strong convergence then forces the remaining quartic energy to vanish, contradicting nonconstancy when its coefficient is positive.

\begin{lemma}[The limiting stress constraint]\label{lem:decay-stress}
For the limit in \cref{lem:decay-limit}, the limiting stress is integrable and divergence free, and $\alpha\int|\nabla Q|^4=0$.
\end{lemma}

\begin{proof}
The exact rescaled stress is
\[
S^n_{ij}=(1+\alpha_nw_n)\partial_i v_n\cdot\partial_jv_n
 -(w_n/2+\alpha_nw_n^2/4+\lambda_n^2(v_n)_3^2/2)\delta_{ij}.
\]
Differentiating and using \eqref{eq:regularizer-decay-7} gives
\[
\partial_iS^n_{ij}
 =-d_n\delta_{j2}\,v_n\cdot(\partial_1v_n\times\partial_2v_n).
\]
The right side tends to zero locally in $L^1$. Strong local $L^4$
convergence of gradients, $\alpha_n\to\alpha$, and $\lambda_n\to0$
give local $L^1$ convergence of the stresses to
\[
S_{ij}=(1+\alpha|\nabla Q|^2)\partial_iQ\cdot\partial_jQ
 -( |\nabla Q|^2/2+\alpha|\nabla Q|^4/4)\delta_{ij}.
\]
Hence $\partial_iS_{ij}=0$ distributionally. This passage uses the
identities for the smooth approximating maps; no regularity or
stationarity theorem for $Q$ is needed. By
\eqref{eq:regularizer-decay-11}, $S\in L^1(\R^2)$ and
\begin{equation}\label{eq:regularizer-decay-12}
\operatorname{tr}S=\alpha|\nabla Q|^4/2\ge0.
\end{equation}
Test its $j$-th divergence with $x_j\chi_R$, sum in $j$, and let
$R\to\infty$, where $\chi_R$ is one on $B_R$, zero outside $B_{2R}$,
and $|\nabla\chi_R|\le C/R$. The identity
\[
\int\chi_R\operatorname{tr}S
=-\int S_{ij}x_j\partial_i\chi_R
\]
has a right-hand side tending to zero by the integrable tail of $S$.
Thus $\alpha\int|\nabla Q|^4=0$.
\end{proof}

We can now complete the proof of quartic decay by combining the nonconstant limit with the stress constraint.

\begin{proof}[Proof of \cref{prop:regularizer-decay}]
If the first limit failed, \cref{lem:decay-rescaling,lem:decay-limit} would give a nonconstant $Q$ with $\alpha>0$. This contradicts \cref{lem:decay-stress}. Hence the first limit holds. Finally,
\[
\eta_n\int|\nabla u_n|^4
\le\left(\sup_{\R^2}\eta_n|\nabla u_n|^2\right)
       \int|\nabla u_n|^2
\le2M\sup_{\R^2}\eta_n|\nabla u_n|^2\longrightarrow0.
\]
\end{proof}

\section{Planar concentration and compactness}\label{app:concentration-compactness}

Once the quartic term has disappeared, every possible gradient
concentration has a harmonic rescaling limit. The stress identity derived from the travelling-wave equation
forces zero net local area at a concentration point, doubling the
usual harmonic-map energy cost. A strict bound on transverse energy
therefore gives compactness after arbitrary translations.

Fix $c\in(0,1)$. Let $\eta_n\ge0$ tend to zero and let smooth maps
$u_n:\R^2\to S^2$ solve
\begin{equation}\label{eq:concentration-compactness-1}
\operatorname{div}((1+\eta_n|\nabla u_n|^2)\nabla u_n)
+(1+\eta_n|\nabla u_n|^2)|\nabla u_n|^2u_n
+(u_n)_3^2u_n-(u_n)_3e_3+c\,u_n\times\partial_1u_n=0.
\end{equation}
Assume
\begin{equation}\label{eq:concentration-compactness-2}
E(u_n)\le M,\qquad k_n:=\eta_n\int|\nabla u_n|^4\longrightarrow0.
\end{equation}
When $\eta_n=0$, the density defining $k_n$ is identically zero, so no
global $L^4$ assumption is imposed on that term.
Put $w(v)=|\nabla v|^2$, $e(v)=(w(v)+v_3^2)/2$, and
$T_{ij}(v)=\partial_iv\cdot\partial_jv-e(v)\delta_{ij}$.

\subsection{Harmonic spheres and differential identities}\label{appsrc-concentration-compactness-inputs-and-the-regularized-differential-identities}

We first record the harmonic-map energy bound and the differential identities that relate the travelling-wave equation to its vorticity and stress.

We use \cite[Proposition~1.4]{delaire2014} for the
smoothness of a finite-energy planar distributional solution, after
establishing its full weak equation. In dimension two that result
requires no prior uniform continuity or stationarity assumption.

We also use the classification of finite-energy harmonic spheres: if \(Q:\mathbb R^2\to S^2\) is smooth and nonconstant, \(\Delta Q+|\nabla Q|^2Q=0\), and \(D(Q):=\frac12\int|\nabla Q|^2<\infty\), then \begin{equation}
D(Q)=\left|\int Q\cdot(\partial_1Q\times\partial_2Q)\right|=4\pi|d|\ge4\pi
\label{eq:concentration-compactness-14}
\end{equation} for a nonzero integer \(d\). This is the finite-energy harmonic-plane classification
in \cite[Proposition~1.1 and Lemmas~1.2 and~2.1]{deng2024}. Its Dirichlet
energy has the same factor $1/2$; its opposite orientation for the area
does not affect the absolute identity.

For one smooth solution \(v\) of the regularised equation write \(w=|\nabla v|^2\) and \(a=1+\eta w\). Its equation is \begin{equation}
\operatorname{div}(a\nabla v)=-aw v-v_3^2v+v_3\,e_3-cv\times\partial_1v.
\label{eq:concentration-compactness-15}
\end{equation} Taking the target cross product with \(v\) gives \begin{equation}
\sum_i\partial_i(a v\times\partial_i v)=v_3(v\times e_3)+c\partial_1v.
\label{eq:concentration-compactness-16}
\end{equation} The terms \(\partial_i v\times\partial_i v\) vanish, as does the component of \eqref{eq:concentration-compactness-15} parallel to \(v\).

Differentiating the regularised stress tensor $T^\eta(v)$ from
\eqref{eq:energy-absorption-W1} gives \begin{equation}
\sum_i\partial_iT^\eta_{ij}
 =\operatorname{div}(a\nabla v)\cdot\partial_jv-v_3\partial_j v_3
 =-c\delta_{j2}q(v).
\label{eq:concentration-compactness-17}
\end{equation} Indeed the derivative of \(w/2+\eta w^2/4\) is \((a/2)\partial_jw\), which cancels \(a\sum_i\partial_i v\cdot\partial_{ij}v\); all terms of \eqref{eq:concentration-compactness-15} parallel to \(v\) have zero dot product with \(\partial_jv\). For the matrix difference \(H=T^\eta-T\), \begin{equation}
H_{ij}=\eta\big(w\partial_i v\cdot\partial_j v-w^2\delta_{ij}/4\big).
\label{eq:concentration-compactness-18}
\end{equation} The Euclidean length of each column of the Gram matrix \((\partial_i v\cdot\partial_jv)_{ij}\) is at most its trace \(w\). Thus \(|(H_{1j},H_{2j})|\le(5/4)\eta w^2\le2\eta w^2\). Thus $\|T^\eta-T\|_{L^1}\le C\eta\int w^2$.

\subsection{Harmonic limits at a concentration scale}\label{appsrc-concentration-compactness-every-gradient-blowup-has-a-harmonic-rescaling-limit}

The quartic decay proved above makes the rescaled principal part approach the Laplacian. We use this to identify any gradient concentration with a nonconstant harmonic sphere, costing at least $4\pi$ in Dirichlet energy.

\begin{lemma}[Harmonic rescaling limits]\label{lem:compactness-rescaling}
Every sequence of unbounded gradients in a fixed disk admits a rescaling converging smoothly locally to a nonconstant finite-energy harmonic sphere, of Dirichlet energy at least $4\pi$. Consequently, an eventual energy bound strictly below $4\pi$ on a disk gives uniform bounds on every derivative on smaller disks.
\end{lemma}

\begin{proof}
Consider points with unbounded gradients in a fixed disk. The weighted maximum construction is as follows. On a slightly larger closed disk \(\overline{B_s(x_0)}\) maximise \[
 (s-|x-x_0|)|\nabla u_n(x)|
\] at an interior point \(y_n\). If the original unbounded-gradient points stay in a smaller disk, the maximum tends to infinity. Put \(t_n=s-|y_n-x_0|\) and \(L_n=|\nabla u_n(y_n)|\). Then \(t_nL_n\to\infty\), \(L_n\to\infty\), and \(|\nabla u_n|\le2L_n\) on \(B_{t_n/2}(y_n)\). Define \[
 v_n(y)=u_n(y_n+y/L_n),\quad
 \alpha_n=\eta_nL_n^2,\quad \lambda_n=L_n^{-1},\quad d_n=c/L_n.
\] These maps solve \eqref{eq:concentration-compactness-19}, have \(|\nabla v_n(0)|=1\), and satisfy \(w(v_n)\le4\) on disks of radii \(t_nL_n/2\to\infty\). By \cref{prop:regularizer-decay},
\begin{equation}\label{eq:concentration-compactness-25}
0\le\alpha_n=\eta_nL_n^2
\le\sup_{\R^2}\eta_n|\nabla u_n|^2\longrightarrow0.
\end{equation}
Indeed, the subsequence with $\eta_n>0$ has bounded quartic energy
by \eqref{eq:concentration-compactness-2}, so that proposition applies;
on the remaining indices $\alpha_n=0$ identically.

Since $\alpha_n\to0$, we obtain uniform local regularity by treating the rescaled operator as a perturbation of the Laplacian. Divide \eqref{eq:concentration-compactness-19} by \(a_n=1+\alpha_nw(v_n)\). In components the second-order part is \begin{equation}
\Delta v_n^\ell+
 \frac{2\alpha_n}{a_n}\sum_{i,j,m}
 (\partial_i v_n^\ell)(\partial_jv_n^m)\partial_{ij}v_n^m.
\label{eq:concentration-compactness-26}
\end{equation} Here \(\ell,m\in\{1,2,3\}\), \(i,j\in\{1,2\}\). Its difference from the componentwise Laplacian has coefficient norm at most a fixed multiple of \(\alpha_n\) on any fixed disk, since \(|\nabla v_n|\le2\). Its right side after division is uniformly bounded. For any fixed \(p\in(1,\infty)\), multiply \(v_n\) by a smooth cutoff supported in a slightly larger fixed disk and apply the componentwise Laplacian \(W^{2,p}\) estimate to that product. The term involving \(D^2v_n\) in \eqref{eq:concentration-compactness-26} is bounded by \(C_p\alpha_n\) times the second-derivative norm of the product, plus bounded terms involving the cutoff, \(v_n\) and \(\nabla v_n\). For large \(n\), \(C_p\alpha_n<1/2\), so it is absorbed on the left. This gives uniform local \(W^{2,p}\) estimates.

Choose \(p>2\) to obtain local \(C^{1,\beta}\) bounds for some \(\beta>0\). The coefficients and right side in \eqref{eq:concentration-compactness-26} are smooth functions of \(v_n,\nabla v_n\) and the bounded parameters, so they have uniform \(C^{0,\beta}\) bounds. They are uniformly strongly elliptic. Interior Hölder estimates for this system then give \(C^{2,\beta}\) bounds. Differentiation and iteration give all higher local bounds.

A diagonal subsequence converges smoothly on every compact subset of the plane to a sphere-valued harmonic map \(Q\) with \(|\nabla Q(0)|=1\). It is nonconstant. The scale invariance of the two-dimensional Dirichlet integral and \eqref{eq:concentration-compactness-2} give \(D(Q)\le M\), so \eqref{eq:concentration-compactness-14} applies. In particular every such blowup carries Dirichlet energy at least \(4\pi\).

It follows that on any fixed disk with an eventual energy bound \(m_0<4\pi\), the gradients on a smaller disk are uniformly bounded along the sequence. Otherwise the construction just given, performed within that disk, would give a harmonic \(Q\) with \(D(Q)\le m_0<4\pi\). Once the gradients are bounded, the same absorption argument, now with parameter \(\eta_n\) at the original scale, supplies uniform bounds of all derivatives on smaller disks. This is the local sequential regularity assertion we need.
\end{proof}

\subsection{The weak limit and the concentration set}\label{appsrc-concentration-compactness-weak-equation-closure-and-the-finite-concentration-set}

We now pass to the original-scale weak limit. The $4\pi$ concentration bound confines any failure of smooth convergence to finitely many points.

\begin{lemma}[Weak closure and isolated concentration]\label{lem:compactness-weak}
There are a subsequence, a smooth finite-energy solution $u$ of
\eqref{eq:main}, and a finite set $S=\{p_1,\ldots,p_N\}$ such that
\begin{equation}\label{eq:concentration-compactness-3}
u_n\rightharpoonup u\text{ in }H^1_{\rm loc},\qquad
u_n\to u\text{ in }L^2_{\rm loc}\text{ and a.e.},\qquad
u_n\to u\text{ in }C^m_{\rm loc}(\R^2\setminus S)
\end{equation}
for every integer $m\ge0$. In particular the limit obeys
\begin{equation}\label{eq:concentration-compactness-4}
-\int\nabla u:\nabla\varphi
+\int(|\nabla u|^2+u_3^2)u\cdot\varphi
-\int u_3e_3\cdot\varphi+c\int(u\times\partial_1u)\cdot\varphi=0.
\end{equation}
Against compactly supported continuous test functions, $e(u_n)dx\rightharpoonup\mu$, where
\begin{equation}
\mu=e(u)dx+\sum_{p\in S}\gamma_p\delta_p,\qquad\gamma_p\ge4\pi.
\label{eq:concentration-compactness-29}
\end{equation}
For real coefficients $\beta_p$ and $A_{ij,p}$, the vorticity and stress measures satisfy
\begin{equation}
q(u_n)dx\rightharpoonup q(u)dx+\sum_{p\in S}\beta_p\delta_p,\quad
 T_{ij}(u_n)dx\rightharpoonup T_{ij}(u)dx+\sum_{p\in S}A_{ij,p}\delta_p.
\label{eq:concentration-compactness-30}
\end{equation}
\end{lemma}

\begin{proof}
The target and energy bounds give a subsequence with a limit \(u\in H^1_{\rm loc}(\mathbb R^2;S^2)\), weak in \(H^1_{\rm loc}\) and strong in \(L^2_{\rm loc}\) and a.e.; lower semicontinuity on increasing disks gives \(E(u)\le M\). The regularising part of the current in \eqref{eq:concentration-compactness-16} tends to zero even in \(L^1\) on its domain, because \begin{equation}
\int_{\mathbb R^2}\eta_n|\nabla u_n|^3
 \le\left(\eta_n\int_{\mathbb R^2}|\nabla u_n|^4\right)^{1/2}
       \left(\eta_n\int_{\mathbb R^2}|\nabla u_n|^2\right)^{1/2}
 \le\sqrt{2M\eta_n k_n}\longrightarrow0.
\label{eq:concentration-compactness-27}
\end{equation} For each fixed smooth compact test, the unregularised current \(u_n\times\partial_i u_n\) passes by strong convergence of \(u_n\) and weak convergence of its derivatives in \(L^2\). The lower terms in \eqref{eq:concentration-compactness-16} also pass, so \begin{equation}
\sum_i\partial_i(u\times\partial_i u)=u_3(u\times e_3)+c\partial_1u.
\label{eq:concentration-compactness-28}
\end{equation} Both its current and right side are locally \(L^2\). Therefore the weak identity extends to compactly supported \(H^1\) vector tests. For a given smooth compact \(\varphi\), insert the allowed test \(\psi=u\times\varphi\). The sphere constraint gives \[
 (u\times\partial_i u)\cdot\partial_i(u\times\varphi)
 =\partial_i u\cdot\partial_i\varphi-|\partial_i u|^2u\cdot\varphi,
\] \[
 u_3(u\times e_3)\cdot(u\times\varphi)
 =u_3e_3\cdot\varphi-u_3^2u\cdot\varphi,
 \qquad
 \partial_1u\cdot(u\times\varphi)=-(u\times\partial_1u)\cdot\varphi.
\] These identities turn \eqref{eq:concentration-compactness-28} into \eqref{eq:concentration-compactness-4}, with all terms integrable. \cite[Proposition~1.4]{delaire2014} now applies to this full finite-energy distributional solution and yields smoothness.

After further extraction let \(e(u_n)dx\) converge against compact continuous functions to a finite nonnegative measure \(\mu\), of total mass at most \(M\). Lower semicontinuity gives \(\mu\ge e(u)dx\). Define \[
 S=\{p:\mu(\{p\})\ge4\pi\}.
\] This is finite. If \(p\notin S\), choose a disk about \(p\) with limiting energy strictly below \(4\pi\). Using a continuous cutoff supported in a slightly larger such disk gives a uniform bound \(m_0<4\pi\) for the energies on a smaller disk. \Cref{appsrc-concentration-compactness-every-gradient-blowup-has-a-harmonic-rescaling-limit} gives local smooth convergence there after extraction. Diagonal extraction on a countable cover proves \eqref{eq:concentration-compactness-3} off \(S\) and \eqref{eq:concentration-compactness-29}. The potential terms converge strongly locally, so these atoms are entirely Dirichlet-energy defects. The bound \(|q(v)|\le|\nabla v|^2/2\le e(v)\) and analogous bounds for \(T\) permit further measure extraction. Smooth convergence off $S$ gives \eqref{eq:concentration-compactness-30} for real coefficients $\beta_p$ and $A_{ij,p}$. No diffuse defect remains away from \(S\) because convergence there is smooth.
\end{proof}

\subsection{Local neutrality and the concentration threshold}\label{appsrc-concentration-compactness-local-neutrality-and-the-doubled-cost}

The stress identity now strengthens the concentration bound. It forces the net vorticity defect at each concentration point to vanish, so the energy of a harmonic bubble must be supplemented by the cost of cancelling its signed area. Together these contributions cost at least $8\pi$.

\begin{proposition}[Local neutrality and the doubled energy cost]\label{prop:compactness-neutrality}
For the subsequence and concentration set in \cref{lem:compactness-weak}, the following measure convergences hold, with numbers $\gamma_j\ge8\pi$, against compact continuous functions:
\begin{equation}\label{eq:concentration-compactness-5}
\begin{gathered}
e(u_n)\,dx\rightharpoonup e(u)\,dx+\sum_j\gamma_j\delta_{p_j},\\
q(u_n)\,dx\rightharpoonup q(u)\,dx,\qquad
T_{il}(u_n)\,dx\rightharpoonup T_{il}(u)\,dx,
\end{gathered}
\end{equation}
\begin{equation}\label{eq:concentration-compactness-6}
\begin{gathered}
(\partial_i u_n\cdot\partial_lu_n)\,dx\rightharpoonup
(\partial_i u\cdot\partial_lu)\,dx+\delta_{il}\sum_j\gamma_j\delta_{p_j},\\
E(u)+\sum_j\gamma_j\le M.
\end{gathered}
\end{equation}
\end{proposition}

\begin{proof}
By \eqref{eq:concentration-compactness-18} and \eqref{eq:concentration-compactness-2}, \(T^{\eta_n}(u_n)-T(u_n)\to0\) in \(L^1\). Thus \eqref{eq:concentration-compactness-17} passes to distributions using \eqref{eq:concentration-compactness-30}. Subtract the exact stress identity for the smooth limit \(u\), which follows from \eqref{eq:concentration-compactness-4} by differentiation. For each \(j\in\{1,2\}\) this gives \begin{equation}
\sum_{p\in S}\sum_i A_{ij,p}\partial_i\delta_p
 =-c\delta_{j2}\sum_{p\in S}\beta_p\delta_p.
\label{eq:concentration-compactness-31}
\end{equation} Fix \(p\in S\) and test in a disk containing no other point of \(S\). A smooth test with value zero and arbitrary gradient at \(p\) forces \(A_{ij,p}=0\) for each \(i,j\). A test equal to one near \(p\), in the \(j=2\) equation, then forces \(\beta_p=0\), since $c>0$. This proves the last two measure convergences in \eqref{eq:concentration-compactness-5}. Adding the energy measure times \(\delta_{ij}\) proves \eqref{eq:concentration-compactness-6}'s directional convergence, provisionally with \(\gamma_p\ge4\pi\).

Fix such a \(p\) and a small fixed \(r>0\) whose closed disk contains no other point of \(S\). The boundary circle has zero \(\mu\) mass. Approximation of its indicator by continuous cutoffs and the bound \(|q(u_n)|\le e(u_n)\) therefore give \begin{equation}
E(u_n;B_r(p))\to\mu(B_r(p)),\qquad
 \int_{B_r(p)}q(u_n)\to\int_{B_r(p)}q(u).
\label{eq:concentration-compactness-32}
\end{equation} The positive energy atom implies unbounded gradients in \(B_{r/4}(p)\). Apply the weighted maximum construction of \cref{appsrc-concentration-compactness-every-gradient-blowup-has-a-harmonic-rescaling-limit} in \(B_{r/2}(p)\). Its centres \(y_n\) tend to \(p\), because the gradients are bounded on every compact set away from \(S\). Its scales \(L_n\to\infty\) give a smooth harmonic limit \(Q\), and for fixed \(R>0\) the core disk \(D_n=B_{R/L_n}(y_n)\) lies in \(B_r(p)\) eventually. The area-energy inequality on the complement of this core gives \begin{equation}
E(u_n;B_r(p))\ge\frac12\int_{D_n}|\nabla u_n|^2
  +\left|\int_{B_r(p)}q(u_n)-\int_{D_n}q(u_n)\right|.
\label{eq:concentration-compactness-33}
\end{equation} Planar positive dilation preserves both core integrals. First pass \(n\to\infty\) using \eqref{eq:concentration-compactness-32} and smooth convergence on \(B_R\), then let \(R\to\infty\). Equation \eqref{eq:concentration-compactness-14} yields \[
 \mu(B_r(p))\ge D(Q)+\left|\int_{B_r(p)}q(u)-\int q(Q)\right|
 \ge8\pi-\left|\int_{B_r(p)}q(u)\right|.
\] Let \(r\downarrow0\). Smoothness of \(u\) implies the last integral tends to zero, so \(\gamma_p\ge8\pi\). The total mass bound of \(\mu\) proves the energy inequality in \eqref{eq:concentration-compactness-6}. This completes the proof.
\end{proof}

\subsection{The transverse cap and translated limits}

The strict transverse-energy cap excludes every concentration point, giving smooth local convergence even after arbitrary translations.

\begin{corollary}[Compactness below the transverse threshold]\label{prop:concentration-compactness}
Under the standing hypotheses \eqref{eq:concentration-compactness-1}--\eqref{eq:concentration-compactness-2}, use the subsequence and limit of \cref{lem:compactness-weak}. If, for some $\varepsilon>0$,
\begin{equation}\label{eq:concentration-compactness-8}
T_0(u_n)=\int|\partial_2u_n|^2\le8\pi-\varepsilon,
\end{equation}
then $S$ is empty and $T_0(u)\le8\pi-\varepsilon$. These conclusions
apply after any translations of the sequence, and
\begin{equation}\label{eq:concentration-compactness-9}
\limsup_{n\to\infty}\|\nabla u_n\|_\infty<\infty.
\end{equation}
The limit may be constant; no global energy tightness is asserted.
\end{corollary}

\begin{proof}
Under \eqref{eq:concentration-compactness-8}, the directional measure
convergence in \eqref{eq:concentration-compactness-6} excludes every
atom: a compact nonnegative test bounded by one and equal to one at
an atom would have limiting transverse mass at least $8\pi$.
Thus $S=\varnothing$, and exhaustion by disks gives
$T_0(u)\le8\pi-\varepsilon$.

The equation and all bounds are translation invariant, so the same
argument applies to $u_n(\cdot+b_n)$ for arbitrary centres $b_n$.
If \eqref{eq:concentration-compactness-9} failed, there would be a
subsequence and centres with $|\nabla u_n(b_n)|\to\infty$.
Apply the already proved local compactness conclusion to these
translated maps. The transverse cap leaves no concentration point,
so a further subsequence converges in $C^1$ near zero, contradicting
the chosen gradients. This proves the gradient bound and completes the compactness argument.
\end{proof}

\section{Removal of the regularisation and the planar proof}\label{app:quartic-compactness}

A nonconstant regularised wave has a positive height determined
by its speed. This estimate prevents the compact limits of the
preceding section from becoming constant. We then assemble the
construction to prove the planar part of the main theorem.

\subsection{The height estimate}\label{appsrc-quartic-compactness-the-height-estimate-on-the-plane}

To preserve nonconstancy, we need a positive height that survives as $\eta\to0$. The following estimate depends only on the speed and is obtained from the phase and height equations.

\begin{lemma}[A positive height for nonconstant waves]\label{lem:quartic-height}
Fix $c\in(0,1)$. Every nonconstant smooth finite-energy solution of
\begin{equation}\label{eq:quartic-compactness-1}
\operatorname{div}((1+\eta|\nabla v|^2)\nabla v)
+(1+\eta|\nabla v|^2)|\nabla v|^2v
+v_3^2v-v_3e_3+c\,v\times\partial_1v=0,
\qquad\eta>0,
\end{equation}
with $E_\eta(v)<\infty$ satisfies
\begin{equation}\label{eq:quartic-compactness-2}
\|v_3\|_\infty\ge\sqrt{1-c}.
\end{equation}
\end{lemma}

\begin{proof}
Let \(v\) be a nonconstant wave as in \eqref{eq:quartic-compactness-1}, fix its parameters \(\eta>0\) and \(c\in(0,1)\), and put \(w=|\nabla v|^2\), \(b=1+\eta w\), and \(\delta=\|v_3\|_\infty\). If \(\delta=1\), \eqref{eq:quartic-compactness-2} follows. Suppose \(\delta<1\). The sphere map has a global smooth phase \(\theta\) on the simply connected plane such that \begin{equation}
v=(\rho\cos\theta,\rho\sin\theta,v_3),\qquad \rho=\sqrt{1-v_3^2},\qquad
 w=\frac{|\nabla v_3|^2}{1-v_3^2}+(1-v_3^2)|\nabla\theta|^2.
\label{eq:quartic-compactness-8}
\end{equation} Finite \(E_\eta\) implies \(\nabla v\in L^2\cap L^4\), \(v_3\in H^1\cap L^4\), and, since \(1-v_3^2\ge1-\delta^2>0\), \(\nabla\theta\in L^2\cap L^4\). The \(L^4\) bound on \(v_3\) also follows directly from \(|v_3|\le1\) and \(v_3\in L^2\).

The angular component of \eqref{eq:quartic-compactness-1} gives \begin{equation}
\operatorname{div}\big(b(1-v_3^2)\nabla\theta-cv_3(1,0)\big)=0.
\label{eq:quartic-compactness-9}
\end{equation} Let \(\chi_R\) equal one on \(B_R\), vanish outside \(B_{2R}\), and satisfy \(|\nabla\chi_R|\le C/R\), obtained by rescaling a fixed smooth cutoff. Let \(m_R\) be the average of \(\theta\) on the connected annulus \(D_R=B_{2R}\setminus B_R\). Test \eqref{eq:quartic-compactness-9} by \(\chi_R(\theta-m_R)\). Split its current as \[
 F_0=(1-v_3^2)\nabla\theta-cv_3(1,0)\in L^2,\qquad
 F_1=\eta w(1-v_3^2)\nabla\theta\in L^{4/3}.
\] The second membership follows from \(w\in L^2\) and \(\nabla\theta\in L^4\). Annulus Poincaré inequalities, scaled from a fixed annulus and using the same average \(m_R\), give a bound for the cutoff error by \[
 C\big(\|F_0\|_{L^2(D_R)}\|\nabla\theta\|_{L^2(D_R)}
       +\|F_1\|_{L^{4/3}(D_R)}\|\nabla\theta\|_{L^4(D_R)}\big)\longrightarrow0.
\] Every main integrand is in \(L^1\). Thus \begin{equation}
Q:=\int b(1-v_3^2)|\nabla\theta|^2=c\int v_3\partial_1\theta.
\label{eq:quartic-compactness-10}
\end{equation} This argument uses the two available integrability exponents separately; no bounded-gradient assumption or absolute first area moment is used.

The smooth PDE implies, for every compact smooth tangent field \(\xi\) along \(v\), \begin{equation}
\int b\sum_i\partial_i v\cdot\partial_i\xi+v_3\xi_3
             +c\,v\cdot(\xi\times\partial_1v)=0.
\label{eq:quartic-compactness-11}
\end{equation} Take \(\xi=\chi_R v_3(e_3-v_3\,v)/(1-v_3^2)\), the height-scaling variation with cutoff. In the coordinates \eqref{eq:quartic-compactness-8}, expansion of \eqref{eq:quartic-compactness-11} gives \begin{equation}
\int\chi_R\left(\frac{b|\nabla v_3|^2}{(1-v_3^2)^2}
              -bv_3^2|\nabla\theta|^2+v_3^2-cv_3\partial_1\theta\right)
 +\int\frac{bv_3\nabla v_3\cdot\nabla\chi_R}{1-v_3^2}=0.
\label{eq:quartic-compactness-12}
\end{equation} The absolute value of the second term is at most \[
 \frac{C_\delta}{R}\left(\|v_3\|_2\|\nabla v_3\|_2
          +\eta\|v_3\|_4\|\nabla v\|_4^3\right)\longrightarrow0.
\] All main terms are dominated by integrable multiples of \(w+\eta w^2+v_3^2+|v_3\partial_1\theta|\). Therefore passage to \(R\to\infty\), followed by \eqref{eq:quartic-compactness-10}, yields \begin{equation}
H+Z=\int b|\nabla\theta|^2,\qquad
 H=\int\frac{b|\nabla v_3|^2}{(1-v_3^2)^2}\ge0,
 \qquad Z=\int v_3^2.
\label{eq:quartic-compactness-13}
\end{equation} If \(Z=0\), then \(v_3=0\) and \eqref{eq:quartic-compactness-10} gives \(\nabla\theta=0\), contradicting nonconstancy. Thus \(Z>0\), and \eqref{eq:quartic-compactness-13} gives \(Q>0\). Since \(b\ge1\), Cauchy--Schwarz and \eqref{eq:quartic-compactness-10} imply \[
 Q\le c\sqrt Z\left(\int|\partial_1\theta|^2\right)^{1/2}
 \le c\sqrt{ZQ/(1-\delta^2)},
 \qquad Q\le\frac{c^2Z}{1-\delta^2}.
\] Equation \eqref{eq:quartic-compactness-13} gives \(Z\le Q/(1-\delta^2)\). Combining these inequalities gives \((1-\delta^2)^2\le c^2\), hence \(\delta^2\ge1-c\). This proves \eqref{eq:quartic-compactness-2} in its full stated class.
\end{proof}

\subsection{Compactness after translation}

We centre each wave at a point of positive height and apply the compactness result just established.

\begin{proposition}[Removal of the regularisation]\label{prop:quartic-compactness}
Fix $c\in(0,1)$. If $\eta_n\downarrow0$ and nonconstant smooth solutions
$u_n$ of \eqref{eq:quartic-compactness-1} obey
\begin{equation}\label{eq:quartic-compactness-3}
E_{\eta_n}(u_n)\le M,\qquad T_0(u_n)\le8\pi-\varepsilon
\end{equation}
for fixed $M<\infty$ and $\varepsilon>0$, then some translations
$u_n(\cdot+b_n)$ converge along a subsequence in every local $C^m$ norm
to a nonconstant finite-energy smooth solution of \eqref{eq:main}
at the same speed. The limit has $E(u)\le M$ and
$T_0(u)\le8\pi-\varepsilon$.
\end{proposition}

\begin{proof}
By \cref{prop:regularizer-decay}, the first bound in
\eqref{eq:quartic-compactness-3} implies
$\eta_n\int|\nabla u_n|^4\to0$. Choose centres $b_n$ with
$|(u_n)_3(b_n)|\ge\sqrt{1-c}/2$, using
\eqref{eq:quartic-compactness-2}. Translation preserves every hypothesis
of \cref{prop:concentration-compactness}. Its transverse-cap conclusion
therefore gives a subsequence of $u_n(\cdot+b_n)$ converging smoothly on
compact sets to a solution of \eqref{eq:main}, with $E(u)\le M$ and
$T_0(u)\le8\pi-\varepsilon$. The height at the origin is at least
$\sqrt{1-c}/2$. A constant finite-energy sphere map has zero height,
so the limit is nonconstant.
\end{proof}

\subsection{Conclusion of the planar proof}\label{sec:planar-conclusion}

The ingredients for the planar construction are now in place. Selection and the energy estimate produce a bounded sequence at fixed $\eta$, and the two compactness steps give a nonconstant unregularised wave.

\begin{proof}[Proof of \cref{thm:all-dimensional} in dimension two]
Fix $c\in(0,1)$ and a compact interval $K\Subset(0,1)$ containing $c$.
The path construction and height-preserving selection in
\cref{prop:planar-path,prop:planar-window,prop:planar-selector} give, for
each sufficiently small fixed $\eta>0$, smooth compact maps $u_n$ with
\[
J_{c,\eta}(u_n)\to m_\eta(c),\qquad
0<b_K\le m_\eta(c)\le8\pi-\zeta_K,\qquad
\|(u_n)_3\|_\infty\ge h_K>0.
\]
For $\lambda_n\downarrow0$, the same maps have residual at most
$3\lambda_nB_{\eta,u_n}$, threshold index $i_{\lambda_n}\le1$, and
$|D_1|+|D_2|\le6\lambda_nW_\eta(u_n)$.
The energy estimate in \cref{prop:energy-absorption} absorbs these
weighted scalar errors and gives $E_\eta(u_n)\le M_K$, uniformly also in
small $\eta$. Hence $P(u_n)$ and $W_\eta(u_n)$ are bounded and $D_2(u_n)\to0$.
By \eqref{eq:transverse-action},
\[
T_\eta(u_n)\longrightarrow m_\eta(c)\le8\pi-\zeta_K.
\]

Translate a height point to the origin. \Cref{prop:fixed-eta-recovery}
gives a nonconstant weak critical map $u_\eta$, with strong local
$W^{1,4}$ convergence, $E_\eta(u_\eta)\le M_K$, and
$T_\eta(u_\eta)\le8\pi-\zeta_K$. The independent local regularity result
\cref{prop:weak-regularity} makes $u_\eta$ smooth. Its stress and cutoff
Pohozaev identities now follow directly from the equation, as proved in
\cref{prop:energy-stress}.

Finally choose $\eta_j\downarrow0$. These are exact smooth regularised
waves with bounded $E_{\eta_j}$ and
$T_0(u_{\eta_j})\le T_{\eta_j}(u_{\eta_j})\le8\pi-\zeta_K$.
\Cref{prop:quartic-compactness} gives, after translations and extraction,
a nonconstant finite-energy smooth solution of \eqref{eq:main} at speed
$c$. The fixed-parameter limit has produced exact waves before
$\eta_j$ tends to zero, as required by the compactness proposition.
\end{proof}

\section{Compactly supported target rotations}\label{sec:rotations}

We turn to dimensions $d\ge3$. The first step is a variational
principle that passes to a weak limit using compactly supported target rotations.
Its energy increments contain only products of weakly convergent
gradients with strongly convergent maps, so no limiting momentum
needs to be assigned.

Write $\mathcal X_d=\{u\in H^1_{\rm loc}(\R^d;S^2):E(u)<\infty\}$ and
$F_{ij}(u)=u\cdot(\partial_i u\times\partial_j u)$; the two-form
$F=\sum_{i<j}F_{ij}\,dx_i\wedge dx_j$ is locally integrable.
The following proposition isolates the hypotheses under which minimisation with
an energy cap yields the weak travelling-wave equation.

\begin{proposition}[Compactly supported target rotations]\label{prop:rotation-proof}
Fix $c\in(0,1)$ and a finite group $G\subset O(d)$ fixing
$f_1=(1,0,\ldots,0)$. Let $\rho:G\to SO(3)$ be a homomorphism satisfying
$\rho(g)e_3\in\{e_3,-e_3\}$. Let $\mathcal A\subset\mathcal X_d$ consist
of maps with $v(gx)=\rho(g)v(x)$, and let $P:\mathcal A\to\R$.
For every compact smooth field $\psi$ with
$\psi(gx)=\rho(g)\psi(x)$, set $Q_t=\exp(t[\psi\times])$.
Assume $Q_tv\in\mathcal A$ and
\begin{equation}\label{eq:rotation-increment}
P(Q_tv)-P(v)=-\int_0^t\!\int(\partial_1\psi)\cdot Q_rv\,dx\,dr
\end{equation}
for every $v\in\mathcal A$ and $t\in\R$.
Suppose the capped infimum
\[
I=\inf\{E(v)+cP(v):v\in\mathcal A,\ E(v)\le M\}
\]
is finite, and a minimising sequence $v_n$ satisfies
$E(v_n)\le M-\delta$ eventually, for some $\delta>0$.
If $v_n\rightharpoonup u$ in $H^1_{\rm loc}$ and $v_n\to u$ in
$L^2_{\rm loc}$, then $u\in\mathcal X_d$ satisfies \eqref{eq:weak}
against every compact smooth vector field, and
$E(u)\le\liminf_nE(v_n)$.

The rotation hypotheses hold for classes defined by equivariance,
$v-e_1\in W^{1,p}$, a compact exterior constant $e_1$ fixed by every
$\rho(g)$, and a prescribed distributional $dF(v)$, when
$P(v)=\int x_kF_{1k}(v)$ for one $k\ne1$.
\end{proposition}

\begin{proof}
We first prove that compact rotations preserve the prescribed
current and give the stated momentum increment. We then pass the
energy increments to the weak limit and remove the symmetry
restriction on the test fields.

 The estimates
\[
|Q_t-\mathrm{Id}|\le |t||\psi|,\qquad
|\partial_iQ_t|\le |t||\partial_i\psi|,\qquad
\partial_i(Q_tv)=Q_t\partial_iv+(\partial_iQ_t)v
\]
follow from the matrix exponential and orthogonality. They preserve the
Sobolev and exterior conditions. Since target rotations preserve cross
products, $Q_t(gx)=\rho(g)Q_t(x)\rho(g)^{-1}$, so equivariance is preserved.
Differentiating the area density gives, in distributions,
\[
\partial_tF_{ij}(Q_tv)
=(\partial_i\psi)\cdot\partial_j(Q_tv)
 -(\partial_j\psi)\cdot\partial_i(Q_tv).
\]
The differentiation is valid in $L^1_{\rm loc}$ by the
$L^\infty$--$L^2$--$L^2$ product estimate. Undifferentiated $\psi$ terms
cancel by rotational invariance of scalar triple products. Consequently
\begin{equation}\label{eq:rotation-current}
F_{ij}(Q_tv)-F_{ij}(v)=\partial_iB_j-\partial_jB_i,
\qquad B_i=\int_0^t\psi\cdot\partial_i(Q_rv)\,dr.
\end{equation}
Mixed distributional derivatives commute; hence $dF(Q_tv)=dF(v)$.
The fields $B_i$ have compact support. Multiplying
\eqref{eq:rotation-current} by $x_k$, integrating with a cutoff equal to
one on that support, and using $k\ne1$ gives
\[
P(Q_tv)-P(v)=\int B_1
=-\int_0^t\!\int(\partial_1\psi)\cdot Q_rv,
\]
as required. These moments exist absolutely for the exterior-constant
maps, since $|F_{1k}(v)|\le|\nabla v|^2/2$ and its support is compact.

The admissibility and momentum-increment identities are now established.
We use the energy margin to vary a minimising sequence in both directions
and pass these local increments to its weak limit.

Now fix an equivariant $\psi$, and write $J=E+cP$ on $\mathcal A$.
The energy cap and the preceding bounds give
$|E(Q_tv_n)-E(v_n)|\le C_\psi|t|$ for $|t|\le1$, uniformly in $n$.
Thus both signs of $t$ are admissible on one fixed interval about zero,
and
\[
0\le J(v_n)-I+J(Q_tv_n)-J(v_n).
\]
Orthogonality gives
\begin{equation}\label{eq:rotation-energy}
\begin{aligned}
E(Q_tv_n)-E(v_n)
={}&\sum_i\int\partial_iv_n\cdot Q_t^T(\partial_iQ_t)v_n
 +\frac12\sum_i\int| (\partial_iQ_t)v_n|^2\\
&+\frac12\int\bigl[(e_3\cdot Q_tv_n)^2-(v_n)_3^2\bigr].
\end{aligned}
\end{equation}
All terms are local. Strong convergence of the bounded maps and weak
convergence of their gradients pass both \eqref{eq:rotation-energy}
and \eqref{eq:rotation-increment} to the limit. The limiting increment
is nonnegative on both sides of zero, so its derivative vanishes:
\begin{equation}\label{eq:rotation-first-variation}
\mathcal L_u(\psi):=
\int\sum_i(u\times\partial_i u)\cdot\partial_i\psi
 +u_3(u\times e_3)\cdot\psi+c\,\partial_1u\cdot\psi=0.
\end{equation}

The identity so far holds for equivariant test fields. Averaging over
the symmetry group removes this restriction.

The limit is equivariant by strong local convergence. For an arbitrary
compact smooth $\psi$, put
$\psi^G(x)=|G|^{-1}\sum_{g\in G}\rho(g)^{-1}\psi(gx)$.
This field is equivariant. A change of variables in
\eqref{eq:rotation-first-variation} shows
$\mathcal L_u(\rho(g)^{-1}\psi\circ g)=\mathcal L_u(\psi)$:
the gradient term uses orthogonality and the cross-product rule, the
potential term uses $\rho(g)e_3=\pm e_3$, and the travel term uses
$gf_1=f_1$. Thus $\mathcal L_u(\psi)=\mathcal L_u(\psi^G)=0$.
Its coefficients are locally $L^2$, so density extends the identity
to compactly supported $H^1$ fields. Taking $\psi=u\times\varphi$ gives
\eqref{eq:weak}, because
\[
\begin{aligned}
(u\times\partial_i u)\cdot\partial_i(u\times\varphi)
 &=\partial_i u\cdot\partial_i\varphi-|\partial_i u|^2u\cdot\varphi,\\
u_3(u\times e_3)\cdot(u\times\varphi)
 &=u_3e_3\cdot\varphi-u_3^2u\cdot\varphi,\\
\partial_1u\cdot(u\times\varphi)&=-(u\times\partial_1u)\cdot\varphi.
\end{aligned}
\]
Finally, $|u|=1$ follows from strong local convergence, and lower
semicontinuity on increasing balls gives the energy assertion.
\end{proof}

\section{A three-dimensional comparison profile}\label{app:charged-comparison}

We construct a compact three-dimensional profile carrying four
point charges. Its Dirichlet energy lies strictly below $16\pi$,
which will give the comparison gap in every dimension $d\ge3$. In this section $x=(x_1,x_2,x_3)\in\R^3$ and
\[
\begin{gathered}
D(v)=\bigl(v\cdot(\partial_2 v\times \partial_3 v),\,
           v\cdot(\partial_3 v\times \partial_1 v),\,
           v\cdot(\partial_1 v\times \partial_2 v)\bigr),\\
T(v)=\tfrac12\int_{\R^3}|\nabla v|^2,\qquad
V(v)=\tfrac12\int_{\R^3}v_3^2.
\end{gathered}
\]
Set $g(x_1,x_2,x_3)=(x_1,-x_2,-x_3)$ and
$\mathcal R=\operatorname{diag}(1,-1,-1)$ in the target.
The standard basis vectors $e_i$ also specify points in the
three-dimensional domain when they occur as subscripts of point masses.
With the area form $F$ from \cref{sec:rotations}, one has
$dF(v)=(\operatorname{div}D(v))\dd x_1\wedge \dd x_2\wedge \dd x_3$.

\begin{proposition}\label{thm:strict-comparison}
There is a sphere-valued map $w_*:\R^3\to S^2$ and $\eta_*>0$ such that
\[
w_*=e_1\quad\hbox{outside the closed unit ball},\qquad
w_*-e_1\in W^{1,p}(\R^3;\R^3)\quad(2<p<3),
\]
\[
w_*(gx)=\mathcal Rw_*(x),\qquad
\operatorname{div}D(w_*)=
4\pi(\delta_{e_2}+\delta_{-e_2}-\delta_{e_1}-\delta_{-e_1}),
\]
and
\[
0<T_*:=T(w_*)\le16\pi-\eta_*,
\qquad V_*:=V(w_*)<\infty.
\]
The divergence identity is distributional. The same map works for every
$p\in(2,3)$.
\end{proposition}

\begin{proof}
The profile is obtained by composing a normalised vector field
with a folded sphere map. We compute its charges and then estimate
the energy to retain a strict saving below $16\pi$.

\emph{Construction and Sobolev regularity.}
The zeros of the vector field specify the defects, while the fold makes
the map constant outside the ball.
For $\varepsilon>0$, define
\[
H_\varepsilon(x)=
\bigl(\varepsilon(|x|^2-1),x_1x_2,-|x|x_3\bigr),
\qquad
Q(\xi)=
\begin{cases}
(1-2\xi_1^2,-2\xi_1\xi_2,-2\xi_1\xi_3),&\xi_1<0,\\
e_1,&\xi_1\ge0,
\end{cases}
\quad \xi\in\Stwo.
\]
The map $Q$ takes values in $\Stwo$. Its two formulas agree on
$\{\xi_1=0\}$, and their bounded derivatives make $Q$ Lipschitz.
Set
\[
w_\varepsilon(x)=Q\left(\frac{H_\varepsilon(x)}{|H_\varepsilon(x)|}\right)
\quad\text{for }x\notin\{\pm e_1,\pm e_2\},
\]
and assign the value $e_1$ at the four excluded points.
The function $H_\varepsilon$ is $C^1$ on $\Rthree$, including at zero,
where the derivative of $|x|x_3$ is zero. Its zeros are precisely
$\pm e_1,\pm e_2$. At each such zero $b$, the derivative matrix
\[
B_{\varepsilon,b}=
\bigl(\partial_1H_\varepsilon(b),\partial_2H_\varepsilon(b),
\partial_3H_\varepsilon(b)\bigr)
\]
satisfies
\[
\det B_{\varepsilon,b}=2\varepsilon(b_2^2-b_1^2)\ne0.
\]
The first-order expansion at $b$ and the derivative of normalisation
therefore give, near each zero,
\[
|\nabla w_\varepsilon(x)|\le C_\varepsilon|x-b|^{-1}
\quad\text{almost everywhere}.
\]
Elsewhere the map is locally Lipschitz. For $|x|\ge1$, the first
component of $H_\varepsilon$ is nonnegative, so $w_\varepsilon=e_1$.

The displayed derivative bound is locally in $L^p$ for every $1\le p<3$.
To extend the weak derivatives across the four points, integrate by parts
outside balls of radius $\rho$ about them. Since $w_\varepsilon$ is
bounded, the boundary terms against a smooth test function are
$O(\rho^2)$ and tend to zero. Compact support of $w_\varepsilon-e_1$
then gives
\[
w_\varepsilon-e_1\in W^{1,p}(\Rthree;\Rthree)
\quad(1\le p<3),\qquad
T(w_\varepsilon)<\infty,\qquad V(w_\varepsilon)<\infty.
\]

\emph{The four charges.}
The four possible singularities are now fixed. We compute their local
degrees to identify the signs and strengths of the charges. Parametrise the target sphere
by
\[
\xi=(\tau,\sqrt{1-\tau^2}\cos\theta,
\sqrt{1-\tau^2}\sin\theta).
\]
Its outward area form is $d\theta\wedge d\tau$.
On $\tau<0$, the map $Q$ preserves $\theta$ and replaces $\tau$ by
$1-2\tau^2$; on $\tau>0$, it is constant. Thus its pulled-back area
integrates to
\[
2\pi\int_{-1}^0(-4\tau)\,d\tau=4\pi,
\]
so $Q$ has degree one. For sufficiently small $\rho$, the uniform first-order expansion of $H_\varepsilon$ gives a homotopy on $\partial B_\rho(b)$
between its normalisation and
\[
\omega\longmapsto
\frac{B_{\varepsilon,b}\omega}{|B_{\varepsilon,b}\omega|},
\qquad \omega\in\Stwo.
\]
Polar decomposition of $B_{\varepsilon,b}$ shows that the latter map has
degree $\operatorname{sgn}\det B_{\varepsilon,b}$. Multiplicativity of
degree and the area formula consequently give
\[
\int_{\partial B_\rho(b)}\Jac(w_\varepsilon)\cdot\nu
=4\pi\operatorname{sgn}\det B_{\varepsilon,b},
\]
where $\nu$ is the outward unit normal of the ball.
The area formula applies to these Lipschitz sphere maps: smooth
approximation followed by normalisation converges uniformly and strongly
in $W^{1,p}$ for each finite $p>2$. Uniform closeness preserves degree,
and strong convergence of the derivatives passes the area integral to the
limit.

For a smooth sphere-valued map, the divergence of $\Jac$ is zero, since
its three derivatives lie in a two-dimensional tangent plane.
Local smooth approximation and normalisation give the same distributional
identity for $w_\varepsilon$ away from the four zeros. Near $b$,
\[
|\Jac(w_\varepsilon)(x)|
\le\frac12|\nabla w_\varepsilon(x)|^2
\le\frac{C_\varepsilon^2}{2|x-b|^2}.
\]
If $\psi\in C_c^\infty(\Rthree)$, replacing $\psi$ by $\psi(b)$ in the
flux over $\partial B_\rho(b)$ has error
$O_\varepsilon(\rho\|\nabla\psi\|_\infty)$. Integration by parts outside
the four balls and passage to $\rho=0$ now give
\[
\operatorname{div}\Jac(w_\varepsilon)
=4\pi(\delta_{e_2}+\delta_{-e_2}-\delta_{e_1}-\delta_{-e_1})
\quad\text{in distributions}.
\]

\emph{Equivariance.}
The identities
\[
H_\varepsilon(gx)=\mathcal RH_\varepsilon(x),\qquad
Q(\mathcal R\xi)=\mathcal RQ(\xi)
\]
give $w_\varepsilon(gx)=\mathcal Rw_\varepsilon(x)$.

\emph{An exact energy identity.}
It remains to obtain the strict energy bound. We first express the energy
as $16\pi$ plus an explicit signed correction.
Write $d\omega$ for surface measure on $\Stwo$, and
$\nabla_{\Stwo}$ and $\Delta_{\Stwo}$ for the spherical gradient and
Laplace--Beltrami operator. For $\omega\in\Stwo$, put
\[
f(\omega)=\omega_1\omega_2,\qquad h(\omega)=\omega_3,\qquad
A(\omega)=f(\omega)^2+h(\omega)^2,
\qquad K_0=\int_{\Stwo}A^{-1/2}\,d\omega.
\]
The identity
\[
A=(1-\omega_1^2)(1-\omega_2^2)
\]
shows that $A$ vanishes only at $\pm e_1,\pm e_2$.
Near each of these points, one factor is bounded away from zero and the
other is comparable to the squared distance. Thus $A^{-1/2}$ is bounded
by a constant times the inverse distance, which is integrable on the
two-dimensional sphere. In particular, $0<K_0<\infty$.

For $0<r<1$, define $k=k(r)=\varepsilon(1-r^2)/r^2$.
Substitution in the formula for $w_\varepsilon$ gives
\begin{equation}\label{eq:comparison-radial-profile}
w_\varepsilon(r\omega)=\frac{(A-k^2,2kf,-2kh)}{A+k^2}.
\end{equation}
Differentiating \eqref{eq:comparison-radial-profile} yields
\[
|\partial_rw_\varepsilon|^2=\frac{4k'^2A}{(A+k^2)^2},\qquad
|\nabla_{\Stwo}w_\varepsilon|^2
=\frac{4k^2(|\nabla_{\Stwo}f|^2+|\nabla_{\Stwo}h|^2)}{(A+k^2)^2}.
\]
Projection of the Euclidean gradients onto the tangent plane gives
\[
\begin{gathered}
|\nabla_{\Stwo}h|^2=1-h^2,\qquad
|\nabla_{\Stwo}f|^2=1-h^2-4f^2,\qquad
\nabla_{\Stwo}f\cdot\nabla_{\Stwo}h=-2fh,\\
\Delta_{\Stwo}f=-6f,\qquad \Delta_{\Stwo}h=-2h.
\end{gathered}
\]
The product rule then gives
\[
\Delta_{\Stwo}A=4-8A-12f^2,\qquad
|\nabla_{\Stwo}A|^2=4A(1-A-3f^2),
\]
and therefore
\[
\Delta_{\Stwo}\log(A+k^2)
=\frac{4k^2(1-3f^2+k^2)}{(A+k^2)^2}-4.
\]
The radial identities needed below are
\[
rk'+2k=-2\varepsilon,\qquad
16\varepsilon^2-16rkk'=4r^2k'^2+16k^2,
\]
\[
8\frac{d}{dr}\frac{rk^2}{A+k^2}
=\frac{8k^2}{A+k^2}+\frac{16rkk'A}{(A+k^2)^2}.
\]
Combining these formulas over the denominator $(A+k^2)^2$ gives
\begin{equation}\label{eq:comparison-energy-density}
\begin{aligned}
r^2|\partial_rw_\varepsilon|^2+|\nabla_{\Stwo}w_\varepsilon|^2
&=\frac{4(r^2k'^2-4k^2)A+8k^2(1+h^2)}{(A+k^2)^2}\\
&=8+2\Delta_{\Stwo}\log(A+k^2)
-8\frac{d}{dr}\frac{rk^2}{A+k^2}
+16\frac{\varepsilon^2A-k^2h^2}{(A+k^2)^2}.
\end{aligned}
\end{equation}

For fixed $r$, the spherical Laplacian in
\eqref{eq:comparison-energy-density} integrates to zero because $k>0$.
The integrated radial boundary term also vanishes at both endpoints.
At zero,
\[
0\le\int_{\Stwo}\frac{rk^2}{A+k^2}\,d\omega\le4\pi r.
\]
At one, its integrand tends to zero almost everywhere and is bounded by
one, so dominated convergence applies.
The remaining terms are integrable. Indeed, $-k'=2\varepsilon/r^3
\ge2\varepsilon$, so the substitution $\lambda=k(r)$ and Tonelli's
theorem give
\begin{equation}\label{eq:comparison-positive-error}
\begin{aligned}
16\varepsilon^2\int_0^1\int_{\Stwo}
\frac{A}{(A+k^2)^2}\,d\omega\,dr
&\le8\varepsilon\int_{\Stwo}\int_0^\infty
\frac{A}{(A+\lambda^2)^2}\,d\lambda\,d\omega\\
&=2\pi\varepsilon K_0.
\end{aligned}
\end{equation}
Here the inner integral is $\pi/(4\sqrt A)$ for $A>0$.
For the negative term, we have
\[
0\le\frac{k^2h^2}{(A+k^2)^2}
\le\frac{k^2A}{(A+k^2)^2}\le\frac14.
\]
Integrating \eqref{eq:comparison-energy-density} first on compact radial
intervals and then taking their endpoints to zero and one proves
\begin{equation}\label{eq:comparison-energy-identity}
2T(w_\varepsilon)=32\pi+
16\int_0^1\int_{\Stwo}
\frac{\varepsilon^2A-k^2h^2}{(A+k^2)^2}\,d\omega\,dr.
\end{equation}

\emph{The strict saving.}
In the energy identity, the positive correction is $O(\varepsilon)$.
We now find a negative contribution of order $\sqrt\varepsilon$, which
dominates for small $\varepsilon$.
Suppose $0<\varepsilon\le1/8$.
On $\sqrt\varepsilon\le r\le2\sqrt\varepsilon$, the decreasing function
$k$ lies between $1/4-\varepsilon$ and $1-\varepsilon$, so
$1/8\le k\le1$. Since $A\le1$, this gives
\[
16\frac{k^2h^2}{(A+k^2)^2}\ge\frac{h^2}{16}
\quad\text{on that interval}.
\]
The interval has length $\sqrt\varepsilon$, and
$\int_{\Stwo}h^2\,d\omega=4\pi/3$. Its contribution to the negative
term in \eqref{eq:comparison-energy-identity}, together with the bound
\eqref{eq:comparison-positive-error} for the positive term, yields
\[
2T(w_\varepsilon)
\le32\pi-\frac{\pi}{12}\sqrt\varepsilon+2\pi\varepsilon K_0.
\]
Choose once and for all
\[
0<\varepsilon_*<\min\{1/8,(48K_0)^{-2}\},\qquad
w_*=w_{\varepsilon_*},\qquad
\eta_*=\frac{\pi}{48}\sqrt{\varepsilon_*}.
\]
After division by two, the preceding bound gives
$T_*=T(w_*)\le16\pi-\eta_*$. The preceding construction gives $V_*<\infty$, the Sobolev
memberships, the charge identity and equivariance. The nonzero charge
identity implies $T_*>0$. The choice of $\varepsilon_*$ is independent
of $p$.
\end{proof}

\section{Comparison and noncollapse in dimensions \texorpdfstring{$d\ge3$}{d >= 3}}\label{app:shell-proof}

We use the three-dimensional comparison profile to construct
charged maps in every dimension $d\ge3$. The signed configuration
has zero projected charge, which permits a common momentum estimate.
A local lower bound for collapse is strictly larger than the
comparison action. This gives a nonconstant weak limit to which the
compact-rotation argument applies.

Let \(d\ge3\) be an integer, \(0<c<1\), \(k=d-3\), \(N=d-1\), \(q=N/(N-1)\), and let \(\sigma_k\) be the area of the unit sphere \(S^k\).
For $k=0$, use $S^0=\{-1,1\}$ with counting measure, so $\sigma_0=2$.
All integrals over $S^0$ below are sums over these two points. Write \[
X=(x_1,x_4,\ldots,x_d)\in\mathbb R^{k+1},\qquad
r=|X|,\qquad x'=(x_2,\ldots,x_d).
\] Thus \(x_1\) is the travel coordinate, and \(x_2,x_3\) are the
two remaining coordinates outside the block \(X\).
When $d=3$, these conventions mean $X=x_1$, $r=|x_1|$, and $x'=(x_2,x_3)$. The vectors \(e_1,e_2,e_3\) below are target vectors in \(\mathbb R^3\), and \(S^2=\{\xi\in\mathbb R^3:|\xi|=1\}\). For a sphere-valued Sobolev map $v$, use the area form $F(v)$ from
\cref{sec:rotations} and set \[
\begin{aligned}
T_d(v)&=\frac12\int_{\mathbb R^d}|\nabla v|^2,\\
V_d(v)&=\frac12\int_{\mathbb R^d}v_3^2,\\
E&=T_d+V_d,\\
P_d(v)&=\int_{\mathbb R^d}x_2\,F_{12}(v),\\
J_{c,d}(v)&=E(v)+cP_d(v).
\end{aligned}
\] All momentum integrals below are absolutely convergent because each admissible map equals \(e_1\) outside a compact set.

Fix the map \(w_*:\mathbb R^3\to S^2\) supplied by \cref{thm:strict-comparison}, retaining the notation $\mathcal R$ and $T_*$ from that proposition. With coordinates \((\xi_1,\xi_2,\xi_3)\) on its domain, that proposition gives \[
w_*=e_1\quad\hbox{outside }B_1^3,\qquad
w_*-e_1\in W^{1,p}(\mathbb R^3)\quad(2<p<3),
\] \[
w_*(\xi_1,-\xi_2,-\xi_3)=\mathcal Rw_*(\xi_1,\xi_2,\xi_3),
\] \[
\begin{gathered}
dF(w_*)=4\pi[
\delta_{(0,1,0)}+\delta_{(0,-1,0)}
-\delta_{(1,0,0)}-\delta_{(-1,0,0)}]\dd \xi_1\wedge\dd \xi_2\wedge\dd \xi_3,\\
0<T_*<16\pi.
\end{gathered}
\] Here and below \(d\) applied to a differential form is its distributional exterior derivative.

The constants in the next inequality compare the two Dirichlet-energy
costs at scale one. The constant $B_{d,R}$ bounds the energy of the
comparison map, while $C_{d,R,b}$ is a lower bound on the limiting
energy near the prescribed defects for sequences converging locally
to $e_1$. We choose the geometry so that these costs are separated:
take \(b>2\) and \(R>2b\) such that \begin{equation}
C_{d,R,b}:=
\frac{4\pi\sigma_k\displaystyle\int_{R-1}^{R+1}r^k\,dr
      +8\pi\sigma_k(R-b)^k}
     {1+2\sqrt3/b}
>
B_{d,R}:=\sigma_k(R+1)^kT_*.
\label{eq:shell-proof-G}
\end{equation} Such choices exist for every \(d\ge3\), independently of \(c\) and \(p\). Set \(\Delta=C_{d,R,b}-B_{d,R}>0\).

For \(\lambda>0\), define the compact distributional three-form \begin{equation}
\begin{aligned}
\mathcal J_{\lambda,R}=4\pi\big[&
 \delta(r-\lambda R)\delta(x_2-\lambda)\delta(x_3)
+\delta(r-\lambda R)\delta(x_2+\lambda)\delta(x_3)\\
&-\delta(r-\lambda(R+1))\delta(x_2)\delta(x_3)
-\delta(r-\lambda(R-1))\delta(x_2)\delta(x_3)\big]\,
dr\wedge \dd x_2\wedge \dd x_3.
\end{aligned}
\label{eq:shell-proof-S}\end{equation} This notation means the following precise distribution: in polar coordinates \(X=r\theta\), \(r>0\), the pullback is the displayed distributional form, independent of \(\theta\in S^k\); it is extended by zero near \(r=0\). Equivalently, the scalar distribution \(\delta(r-a)\) against Euclidean volume integrates a test function over \(\{|X|=a\}\), and \(dr=\sum_{i=1}^{k+1}(X_i/r)\,dX_i\) on its support. Its support is the union of four disjoint embedded copies of $S^k$,
each of codimension three. For $k=0$ these are four pairs of points; their
signs are fixed by the factor $dr$, as verified below.

Let \(g(x)=(x_1,-x_2,-x_3,x_4,\ldots,x_d)\). For \(2<p<3\), let \(\mathcal A_{\lambda,R,p}^{(d)}\) be the maps \(v:\mathbb R^d\to S^2\) satisfying \[
\begin{gathered}
v-e_1\in W^{1,p}(\mathbb R^d;\mathbb R^3),\\
v=e_1\ \hbox{a.e. outside some compact set depending on }v,\qquad
E(v)<\infty,
\end{gathered}
\] \begin{equation}
dF(v)=\mathcal J_{\lambda,R},\qquad v(gx)=\mathcal Rv(x)\quad\hbox{a.e.}
\label{eq:shell-proof-A}
\end{equation} No common support and no radial symmetry in \(X\) are required in this class.

The local lower bound uses the planar symmetry observation of \cite[Lemmas~2.1--2.2]{ziener2026}: an equivariant sphere-valued map on the compactified plane with odd degree takes the opposite pole at the origin, and the resulting disk Dirichlet bound is \(8\pi-2\rho^{-2}\|w-e_3\|_2^2\) in that paper's full-energy normalisation. We give the argument below with the rotated pole \(e_1\) and
half-energy normalisation, and derive the higher-dimensional estimates
from the explicit defect geometry.

\subsection{The defect geometry and comparison map}\label{appsrc-shell-proof-defect-geometry-the-comparison-map-and-rotations}

The condition $dF(v)=\mathcal J_{\lambda,R}$ prescribes the signed
distribution of topological charges. The symmetry condition
$v(gx)=\mathcal Rv(x)$ requires simultaneous reversal of $x_2$ and $x_3$
to correspond to a rotation through $\pi$ about $e_1$ in the target sphere.
We first verify that the admissible class is nonempty by placing the
three-dimensional profile along the prescribed shells, then check that
compactly supported target rotations respecting this symmetry preserve
the class.

\begin{proposition}[The admissible comparison map]\label{prop:shell-comparison}
The source \(\mathcal J_{\lambda,R}\) is closed. The class in \eqref{eq:shell-proof-A} is nonempty and is preserved by every compact smooth target-rotation flow \begin{equation}
\partial_t v_t=\psi(x)\times v_t,\qquad v_0=v,\qquad
  \psi\in C_c^\infty(\mathbb R^d;\mathbb R^3),\quad
  \psi(gx)=\mathcal R\psi(x).
\label{eq:shell-proof-R}
\end{equation} In particular, define \[
  W_R(x)=w_*(|X|-R,x_2,x_3),\qquad
  W_{\lambda,R}(x)=W_R(x/\lambda),
  \] where \(W_R=e_1\) near \(X=0\). These maps satisfy \eqref{eq:shell-proof-A}. If \[
  T_R=T_d(W_R),\qquad V_R=V_d(W_R),
  \] then \(0<T_R\le B_{d,R}\), \(0\le V_R<\infty\), \(P_d(W_{\lambda,R})=0\), and \begin{equation}
J_{c,d}(W_{\lambda,R})
  =\lambda^{d-2}T_R+\lambda^dV_R
  \le\lambda^{d-2}(C_{d,R,b}-\Delta)+\lambda^dV_R.
\label{eq:shell-proof-T}
\end{equation}
\end{proposition}

\begin{proof}
Here \(k=d-3\ge0\), \(N=d-1\ge2\), and \(q=N/(N-1)\), with the coordinates and target vectors fixed above. On \(r>0\), polar coordinates give \begin{equation}
dx=r^k\,dr\,d\theta\dd x_2\dd x_3,\qquad \theta\in S^k.
\label{eq:shell-proof-1-1}
\end{equation} For $k=0$, the coordinates are the two
half-spaces $x_1=r$ and $x_1=-r$, with $r>0$; the volume formula is the sum of
the two integrals $dr\dd x_2\dd x_3$. For oriented forms the second change of
coordinates reverses orientation. In particular, for $a>0$,
\begin{equation}\label{eq:shell-zero-sign}
\delta(|x_1|-a)\,d|x_1|=(\delta_a-\delta_{-a})\dd x_1.
\end{equation}
Thus \eqref{eq:shell-proof-S} in dimension three consists of the four
charges near $x_1=\lambda R$ and their reflected copies near $x_1=-\lambda R$
with opposite signs. This proves the prescribed signs of all eight
point masses. Every source is supported away from $r=0$.

A distributional differential form is specified by its coefficients in smooth local coordinates. Under the change \(X=r\theta\), the form \(\delta(r-a)\,dr\wedge \dd x_2\wedge \dd x_3\) has precisely these coefficients, without an extra \(r^k\) in its coordinate expression. In Euclidean coordinates its scalar delta distribution does use the Euclidean measure as stipulated in the definition of \eqref{eq:shell-proof-S}. This is the usual change-of-variables rule for distributional forms.

Each term in \(\mathcal J_{\lambda,R}\) is closed. In its polar expression, differentiating any of its scalar delta factors produces a one-form already appearing among \(dr,\dd x_2,\dd x_3\), so the exterior product is zero. Its coefficient has no angular dependence. The support stays a positive distance from \(r=0\), so extension by zero adds no boundary distribution there. Each of the four spherical supports is an oriented compact manifold
without boundary, including the signed pairs in \eqref{eq:shell-zero-sign}.

First take \(\lambda=1\). The support of \(W_R-e_1\) is contained in \begin{equation}
(r-R)^2+x_2^2+x_3^2\le1.
\label{eq:shell-proof-1-2}
\end{equation} In particular \(W_R=e_1\) for \(r<R-1\); the formula is unambiguous at \(X=0\). On any angular coordinate chart in this shell, \(W_R\) is the product of \(w_*(r-R,x_2,x_3)\) with the angular variables. The smooth change of coordinates and bounds \(0<R-1\le r\le R+1\) show that it lies in \(W^{1,p}\) for \(2<p<3\). The same fact follows directly by integrating its derivatives using \eqref{eq:shell-proof-1-1}, as \begin{equation}
|\nabla W_R|^2
=|\partial_1 w_*|^2+|\partial_2  w_*|^2+|\partial_3  w_*|^2
\quad\hbox{evaluated at }(r-R,x_2,x_3).
\label{eq:shell-proof-1-3}
\end{equation} Product Sobolev membership and changes of variables justify this formula as a weak derivative identity. Across the boundary of the shell there is no jump: \(w_*\) is a globally defined Sobolev map, equal to \(e_1\) outside its unit ball. A finite angular cover and a partition of unity therefore prove the global assertion \(W_R-e_1\in W^{1,p}\). Its height and Dirichlet energies are finite. Formula \eqref{eq:shell-proof-1-3} gives \begin{equation}
T_R=\frac{\sigma_k}{2}\int_{\mathbb R^3}(R+\xi_1)^k|\nabla w_*(\xi)|^2\dd \xi
\le\sigma_k(R+1)^kT_*.
\label{eq:shell-proof-1-4}
\end{equation} Only \(|\xi_1|\le1\) contributes here. The corresponding formula with \(w_{*,3}^2\) gives \(V_R<\infty\). Positivity of \(T_R\) follows from \(T_*>0\).

In a polar coordinate chart, all angular components of \(F(W_R)\) vanish, and its remaining components are those of \(F(w_*)\) evaluated at \((r-R,x_2,x_3)\). Taking the exterior derivative coefficient by coefficient and using the charge law of \cref{thm:strict-comparison} gives exactly \(\mathcal J_{1,R}\). This computation is valid distributionally: tests factor into an angular test and a normal-coordinate test, and Fubini reduces it to the three-dimensional distributional identity. Finite linear combinations and the angular partition of unity give general tests. It also verifies the identity across the shell boundary and near \(r=0\), where the map is constant. The target and domain reflection identities of \(w_*\) give \(W_R(gx)=\mathcal RW_R(x)\).

For its momentum, almost everywhere in the shell, \begin{equation}
F_{12}(W_R)
=\frac{x_1}{r}\,w_*(r-R,x_2,x_3)\cdot
   \big(\partial_1 w_*(r-R,x_2,x_3)\times\partial_2  w_*(r-R,x_2,x_3)\big).
\label{eq:shell-proof-1-5}
\end{equation} The last factor depends only on \(r,x_2,x_3\). The entire momentum integrand is odd in \(x_1\), and is integrable by \eqref{eq:shell-proof-1-2}--\eqref{eq:shell-proof-1-3}. Hence \(P_d(W_R)=0\). A Euclidean dilation \(x\mapsto x/\lambda\) gives the source \eqref{eq:shell-proof-S}, preserves the reflection, and gives \begin{equation}
T_d(W_{\lambda,R})=\lambda^{d-2}T_R,\qquad
V_d(W_{\lambda,R})=\lambda^dV_R,\qquad
P_d(W_{\lambda,R})=\lambda^{d-1}P_d(W_R)=0.
\label{eq:shell-proof-1-6}
\end{equation}

To choose the geometry, first take \(b>2\) with \begin{equation}
\frac{16\pi}{1+2\sqrt3/b}>T_*.
\label{eq:shell-proof-1-7}
\end{equation} For fixed such \(b\), as \(R\to\infty\), \begin{equation}
\frac{C_{d,R,b}}{\sigma_kR^k}\longrightarrow
\frac{16\pi}{1+2\sqrt3/b},\qquad
\frac{B_{d,R}}{\sigma_kR^k}\longrightarrow T_*.
\label{eq:shell-proof-1-8}
\end{equation} For $k=0$ the formulas are already
exact:
\begin{equation}\label{eq:shell-zero-gap}
T_R=2T_*,\qquad V_R=2V_*,\qquad
B_{3,R}=2T_*,\qquad C_{3,R,b}=\frac{32\pi}{1+2\sqrt3/b}.
\end{equation}
Here $W_R(x_1,x_2,x_3)=w_*(|x_1|-R,x_2,x_3)$, so $W_R-e_1$ is supported in the two
disjoint unit balls centred at $(\pm R,0,0)$ and $W_R=e_1$ near $x_1=0$.
Condition \eqref{eq:shell-proof-1-7} gives the strict gap for any $R>2b$.
For $k\ge1$, the limiting assertions follow because \(R^{-k}\int_{R-1}^{R+1}r^k\,dr\to2\), \((1-b/R)^k\to1\), and \((1+1/R)^k\to1\). Thus an \(R>2b\) satisfies \eqref{eq:shell-proof-G}. This choice is independent of \(c,p,\lambda\).

The class \eqref{eq:shell-proof-A} has exactly the prescribed-current,
Sobolev, exterior and equivariance conditions treated in
\cref{prop:rotation-proof}. Since $\mathcal Re_1=e_1$,
$\mathcal Re_3=-e_3$, and $g$ fixes the travel direction, that proposition
gives preservation under \eqref{eq:shell-proof-R} and the local momentum
increment for $P_d$.
\end{proof}

\subsection{Projected charge and the momentum potential}\label{appsrc-shell-proof-cancellation-under-projection-and-the-momentum-potential}

To control momentum independently of the support, we integrate the vorticity in the propagation direction. Cancellation of the projected charges makes the resulting transverse field a gradient, giving a scalar potential for the momentum.

\begin{lemma}[The momentum potential]\label{lem:shell-potential}
Every \(v\) in \eqref{eq:shell-proof-A} has a compactly supported momentum potential \(\mathfrak p_v\in W^{1,1}(\mathbb R^N)\) satisfying \begin{equation}
\partial_j\mathfrak p_v(x')=-\int_{\mathbb R}F_{1j}(v)(x_1,x')\dd x_1
  \quad (j=2,\ldots,d),\qquad
  P_d(v)=\int_{\mathbb R^N}\mathfrak p_v,\qquad
  \|\mathfrak p_v\|_{L^q}\le T_d(v).
\label{eq:shell-proof-P}
\end{equation}
\end{lemma}

\begin{proof}
Fix an admissible \(v\), at any fixed \(\lambda>0\). Put \begin{equation}
A_j(x')=\int_{\mathbb R}F_{1j}(v)(x_1,x')\dd x_1,\qquad j=2,\ldots,d.
\label{eq:shell-proof-2-1}
\end{equation} These are compactly supported \(L^1\) functions. The elementary pointwise bound needed below is \begin{equation}
\left(\sum_{j=2}^d|F_{1j}(v)|^2\right)^{1/2}
\le|\partial_1v|\left(\sum_{j=2}^d|\partial_jv|^2\right)^{1/2}
\le\frac12|\nabla v|^2.
\label{eq:shell-proof-2-2}
\end{equation} It uses \(|v|=1\) and Cauchy--Schwarz, and therefore applies to the weak derivatives almost everywhere.

The projection of the source in the travel direction is zero: \begin{equation}
\int_{\mathbb R}\iota_{\partial_1}\mathcal J_{\lambda,R}\dd x_1=0
\quad\hbox{as a distributional two-form on }\mathbb R^N.
\label{eq:shell-proof-2-3}
\end{equation} To verify this precisely, contraction with \(\partial_1\) replaces \(dr\) by \(x_1/r\), so each term is a scalar multiple of \[
\delta(r-a)\delta(x_2-\xi_2)\delta(x_3)\frac{x_1}{r}\dd x_2\wedge \dd x_3.
\] For a test function independent of \(X_1\), its integral on the radius-\(a\) sphere is zero under the measure-preserving reflection \(X_1\mapsto-X_1\). This verifies \eqref{eq:shell-proof-2-3} without a pointwise formula
for a projected delta density. It includes the equators when $k\ge1$
and the cancellation of each reflected pair in
\eqref{eq:shell-zero-sign} when $k=0$.

For transverse indices \(i,j\), the \((1,i,j)\) component of \(dF(v)=\mathcal J_{\lambda,R}\) is \[
\partial_1F_{ij}-\partial_iF_{1j}
+\partial_jF_{1i}=(\mathcal J_{\lambda,R})_{1ij}.
\] Test with a transverse compact test times a longitudinal cutoff equal to one on this map's support and on the source. Integration in \(x_1\), followed by \eqref{eq:shell-proof-2-3}, yields \begin{equation}
\partial_iA_j=\partial_jA_i.
\label{eq:shell-proof-2-4}
\end{equation} The cutoff is used only to justify an identity for this fixed map; no estimate will involve its radius.

Define \begin{equation}
\mathfrak p_v(x')=\int_{x_2}^\infty
 A_2(t,x_3,x_4,\ldots,x_d)\,dt.
\label{eq:shell-proof-2-5}
\end{equation} Its \(x_2\)-derivative is \(-A_2\). For \(j\ne2\), \eqref{eq:shell-proof-2-4} and the vanishing of \(A_j\) for large \(x_2\) give \(\partial_j\mathfrak p_v=-A_j\) in distributions. To check compact support on the negative \(x_2\) side, put \(A_{\mathrm{tot}}(x_3,x_4,\ldots,x_d)=\int_{\mathbb R}A_2\dd x_2\). Equation \eqref{eq:shell-proof-2-4} gives every distributional derivative of \(A_{\mathrm{tot}}\) equal to zero. Thus \(A_{\mathrm{tot}}\) is constant almost everywhere on the connected Euclidean space \(\mathbb R^{N-1}\); mollification, followed by the fact that a smooth function with zero gradient is constant, proves this statement if needed. Since \(A_{\mathrm{tot}}\) is compactly supported and integrable, that constant is zero. Formula \eqref{eq:shell-proof-2-5} is therefore compactly supported and integrable. It belongs to \(W^{1,1}\), and \eqref{eq:shell-proof-2-1}--\eqref{eq:shell-proof-2-2} give \begin{equation}
\|\nabla'\mathfrak p_v\|_{L^1(\mathbb R^N)}
\le\int_{\mathbb R^d}\left(\sum_{j=2}^d|F_{1j}|^2\right)^{1/2}
\le T_d(v).
\label{eq:shell-proof-2-6}
\end{equation} Finally integration by parts in \(x_2\), with a cutoff equal to one on the support of \(\mathfrak p_v\), yields \begin{equation}
P_d(v)=\int_{\mathbb R^N}x_2 A_2\,dx'
=-\int_{\mathbb R^N}x_2\partial_2 \mathfrak p_v\,dx'
=\int_{\mathbb R^N}\mathfrak p_v\,dx'.
\label{eq:shell-proof-2-7}
\end{equation}

We include the nonsharp Sobolev inequality used here. For a smooth compactly supported real function \(f\) on \(\mathbb R^N\), the fundamental theorem in each coordinate gives \[
|f(x)|\le F_i(\widehat x_i):=\int_{\mathbb R}|\partial_i f|\,dx_i.
\] Consequently \[
|f|^{N/(N-1)}\le\prod_{i=1}^N F_i(\widehat x_i)^{1/(N-1)}.
\] For nonnegative integrable \(F_i\) independent of their \(i\)-th variable, \begin{equation}
\int_{\mathbb R^N}\prod_{i=1}^N F_i(\widehat x_i)^{1/(N-1)}
\le\prod_{i=1}^N\left(\int_{\mathbb R^{N-1}}F_i\right)^{1/(N-1)}.
\label{eq:shell-proof-2-8}
\end{equation} Here is an induction proof. For \(N=2\) the integral factors. For \(N\ge3\), integrate first in \(x_N\) and apply Hölder with \(N-1\) factors to the functions indexed by \(i<N\). This replaces those functions by \(G_i=\int_{\mathbb R}F_i\,dx_N\) on their \(N-2\) remaining variables. In the remaining \(N-1\) variables, Hölder with exponents \(N-1\) and \((N-1)/(N-2)\) separates \(F_N^{1/(N-1)}\). The second factor is the \((N-2)/(N-1)\)-power of \(\int\prod_{i<N}G_i^{1/(N-2)}\), which is bounded by the induction hypothesis. The resulting exponents are exactly those in \eqref{eq:shell-proof-2-8}. Truncation and Tonelli justify the argument if some intermediate integrals are initially infinite.

It follows that \begin{equation}
\|f\|_{L^q}\le\prod_i\|\partial_i f\|_1^{1/N}
\le\frac1N\sum_i\|\partial_i f\|_1
\le\|\nabla f\|_1.
\label{eq:shell-proof-2-9}
\end{equation} Mollification and Fatou extend this to compactly supported \(W^{1,1}\) functions. Equations \eqref{eq:shell-proof-2-6} and \eqref{eq:shell-proof-2-9} prove \eqref{eq:shell-proof-P}.
\end{proof}

\subsection{Coercivity and its local form}\label{appsrc-shell-proof-longitudinal-loops-the-global-action-bound-and-its-localization}

The potential estimate controls the momentum contribution from slices with large energy or a nonzero integer area correction. On the remaining slices, a phase representation bounds momentum directly by energy; together these estimates give small-energy coercivity.

\begin{proposition}[A support-independent momentum bound]\label{prop:shell-momentum}
Every $v\in\mathcal A_{\lambda,R,p}^{(d)}$ satisfies the following bound. For every \(0<\delta\le1/4\), \begin{equation}
|P_d(v)|\le
  \frac{E(v)}{\sqrt{1-2\delta}}
  +\delta^{-1/N}E(v)^{1+1/N}
  +\pi^{-1/(N-1)}E(v)^q .
\label{eq:shell-proof-M}
\end{equation} All these bounds are independent of the support of \(v-e_1\), of \(\lambda\), and of \(p\).
\end{proposition}

\begin{proof}
We establish the required coercivity estimates for the prescribed-current
class by a longitudinal-loop decomposition.
Write \(e=E(v)\), and for almost every \(x'\) define \begin{equation}
\ell(x')=\frac12\int_{\mathbb R}
\big(|\partial_1v(x_1,x')|^2+v_3(x_1,x')^2\big)\dd x_1.
\label{eq:shell-proof-3-1}
\end{equation} Its integral is at most \(e\). Each such longitudinal slice is absolutely continuous on bounded intervals and equals \(e_1\) near both ends of the line. For its height \(\zeta(x_1)=v_3(x_1,x')\), \begin{equation}
\sup_{x_1}|\zeta(x_1)|^2
\le2\int_{\mathbb R}|\zeta\zeta'|
\le\int_{\mathbb R}(\zeta^2+|\zeta'|^2)\le2\ell(x').
\label{eq:shell-proof-3-2}
\end{equation} Fix \(0<\delta\le1/4\), let \(h=\sqrt{1-2\delta}\), and put \(G=\{\ell\le\delta\}\). On \(G\), the slice avoids the target poles \(\pm e_3\), so it has an absolutely continuous real phase \(\vartheta\), with \[
v=(\sqrt{1-\zeta^2}\cos\vartheta,\sqrt{1-\zeta^2}\sin\vartheta,\zeta),
\qquad
\vartheta'=\frac{v_1\partial_1v_2-v_2\partial_1v_1}{v_1^2+v_2^2}.
\] The phase can be defined by integrating the last expression with an initial phase; direct differentiation recovers the first two components. Define \[
b_v(x')=-\int_{\mathbb R}\zeta\vartheta'\dd x_1\quad(x'\in G).
\] The identity \(|\partial_1v|^2=\zeta'^2/(1-\zeta^2)+(1-\zeta^2)\vartheta'^2\) implies \begin{equation}
|b_v(x')|\le h^{-1}\int_{\mathbb R}|v_3||\partial_1v|\dd x_1
\le h^{-1}\ell(x').
\label{eq:shell-proof-3-3}
\end{equation}

We next justify, for the precise Sobolev class here, \begin{equation}
\mathfrak p_v(x')-b_v(x')\in2\pi\mathbb Z
\quad\hbox{for almost every }x'\in G.
\label{eq:shell-proof-3-4}
\end{equation} Fix almost every value of \((x_3,x_4,\ldots,x_d)\). Sobolev slicing gives a \(W^{1,p}\) map in the \((x_1,x_2)\)-plane, equal to \(e_1\) outside a compact rectangle. Since \(p>2\), its continuous representative is uniformly continuous, sphere valued everywhere, and constant outside that rectangle. Componentwise convolution followed by normalisation gives smooth sphere-valued approximations, with a common larger compact support, converging uniformly and strongly in \(W^{1,p}\). Uniform convergence makes normalisation legitimate. After choosing a subsequence with summable integrated derivative errors, Fubini gives strong longitudinal \(W^{1,p}\) convergence for almost every horizontal line \(x_2\). On a line in \(G\), the limit avoids the poles by \eqref{eq:shell-proof-3-2}, so these approximations avoid the poles there as well; the phase-derivative formula and its denominator bounded away from zero show convergence of \(b_v\) on this line. The Jacobians converge in \(L^1\) on the planar rectangle, and therefore their integrals on the half-plane \(t>x_2\) converge to \(\mathfrak p_v(x')\), as defined in \eqref{eq:shell-proof-2-5}.

For a smooth cap on a rectangle in \((x_1,t)\) with bottom edge \(t=x_2\) and three other edges constant, the outward sphere area form away from its poles is \[
\omega=d\vartheta\wedge d\zeta=d(-\zeta\,d\vartheta).
\] The oriented bottom edge runs in increasing \(x_1\), so its primitive integral is \(b_v\). If the poles are regular values of the cap, remove small disks around their finitely many interior preimages and apply Stokes. Each deleted circle contributes an integer multiple of \(2\pi\) in the limit: \(\zeta\) tends to one of \(1,-1\), and the change of \(\vartheta\) is \(2\pi\) times an integer. At a regular preimage, the height error on a radius-\(\epsilon\) circle is \(O(\epsilon^2)\) and the angular variation is bounded, which justifies the limit. The area over the deleted disks tends to zero. Thus cap area minus the bottom primitive lies in \(2\pi\mathbb Z\). The general smooth cap is reduced to this case by arbitrarily small constant target rotations choosing a regular antipodal pair; Sard's theorem ensures such pairs, and compactness of the boundary curve keeps it away from the perturbed poles. The other edges remain constant. Pass to the identity rotation, then to the Sobolev limit described above. Both real quantities converge, and \(2\pi\mathbb Z\) is closed. This proves \eqref{eq:shell-proof-3-4}.

Let \(G_1=\{x'\in G:\mathfrak p_v-b_v\ne0\}\). By \eqref{eq:shell-proof-3-3}--\eqref{eq:shell-proof-3-4}, on \(G_1\), \[
|\mathfrak p_v|\ge2\pi-\delta/h\ge\pi.
\] Hence \eqref{eq:shell-proof-P} gives \begin{equation}
\int_{G_1}|\mathfrak p_v|
\le\pi^{1-q}\|\mathfrak p_v\|_q^q
\le\pi^{-1/(N-1)}e^q.
\label{eq:shell-proof-3-5}
\end{equation} Also \(|G^c|\le e/\delta\), and Hölder with \(1-1/q=1/N\) gives \begin{equation}
\int_{G^c}|\mathfrak p_v|
\le\|\mathfrak p_v\|_q |G^c|^{1/N}
\le\delta^{-1/N}e^{1+1/N}.
\label{eq:shell-proof-3-6}
\end{equation} On \(G\setminus G_1\), \(\mathfrak p_v=b_v\); integrate \eqref{eq:shell-proof-3-3} there. Together with \eqref{eq:shell-proof-2-7}, this proves
\eqref{eq:shell-proof-M}. All these estimates include $N=2$, where
$q=2$ and the two error powers are $e^{3/2}$ and $e^2$.
\end{proof}

We now choose the slice threshold and energy cap so that the momentum contribution can be absorbed into the energy.

For \(0<c<1\), put \[
\delta_c=\frac{1-c^2}{4},\quad
h_c=\sqrt{1-2\delta_c},\quad
a_c=1-\frac{c}{h_c}>0,\quad
\gamma_c=\frac{a_c}{2},
\] \begin{equation}
M_{d,c}=\min\left\{1,\ \delta_c\left(\frac{a_c}{4c}\right)^N,\
\pi\left(\frac{a_c}{4c}\right)^{N-1}\right\}.
\label{eq:shell-proof-C}
\end{equation} 

\begin{corollary}[Small-energy coercivity]\label{cor:shell-coercivity}
With the constants above, \begin{equation}
J_{c,d}(v)\ge\gamma_cE(v)
\quad\hbox{whenever }v\in\mathcal A_{\lambda,R,p}^{(d)},\ E(v)\le M_{d,c}.
\label{eq:shell-proof-Co}
\end{equation}
\end{corollary}

\begin{proof}
For \(\delta=\delta_c\), the coefficient \(h_c\) is strictly greater than \(c\). If \(e\le M_{d,c}\), the first error term of \(c\) times \eqref{eq:shell-proof-M}, divided by \(e>0\), is at most \(a_c/4\); the second has the same bound, by \eqref{eq:shell-proof-C}. Thus \[
J_{c,d}(v)\ge(1-c/h_c)e-(a_c/2)e=\gamma_ce.
\] The zero-energy case satisfies the inequality as well. This proves \eqref{eq:shell-proof-Co}.
\end{proof}

The same decomposition also gives a local estimate, separating the energy in a fixed bounded set from the contribution of its complement.

\begin{lemma}[Localised action bound]\label{lem:shell-local-action}
Let $v\in\mathcal A_{\lambda,R,p}^{(d)}$, let $B\subset\mathbb R^d$ be bounded and measurable, and write $e=E(v)$ and $e_{\rm out}=E(v;\mathbb R^d\setminus B)$. For $0<\delta\le1/4$ and $h=\sqrt{1-2\delta}$,
\begin{equation}
|P_d(v)|\le h^{-1}e_{\rm out}
+h^{-1}\|v_3\|_{L^2(B)}\|\partial_1v\|_{L^2(B)}
+\delta^{-1/N}e^{1+1/N}+\pi^{-1/(N-1)}e^q.
\label{eq:shell-proof-3-7}
\end{equation}
If $e\le H$, taking $\delta=\delta_c$ and $h=h_c$ gives
\begin{equation}
J_{c,d}(v)\ge T_d(v;B)
-\frac{c}{h_c}(2H)^{1/2}\|v_3\|_{L^2(B)}
-c\delta_c^{-1/N}H^{1+1/N}
-c\pi^{-1/(N-1)}H^q.
\label{eq:shell-proof-3-8}
\end{equation}
\end{lemma}

\begin{proof}
Let \(B\) be any fixed bounded measurable set in \(\mathbb R^d\), put \(e_{\rm out}=E(v;\mathbb R^d\setminus B)\), and use \eqref{eq:shell-proof-3-3} only on \(G\setminus G_1\). The integral of \(|v_3||\partial_1v|\) over its longitudinal fibres is bounded by the integral over all of \(B\), plus the integral over all of its complement. Cauchy--Schwarz on \(B\) and \(2ab\le a^2+b^2\) on the complement give \eqref{eq:shell-proof-3-7}. No support bound was used. For \(\delta=\delta_c\) and \(e\le H\), \eqref{eq:shell-proof-3-8} follows because \(1-c/h_c>0\) permits discarding the remaining exterior energy and the local height energy.
\end{proof}

\subsection{The local cost of collapse}\label{appsrc-shell-proof-a-local-geometric-lower-bound-for-collapse}

We now estimate the energy that must remain near the prescribed defects even if a sequence approaches $e_1$ locally. The first step is a planar symmetry argument, which forces an odd-degree slice to pass through $-e_1$.

\begin{lemma}[Equivariant planar slices]\label{lem:shell-planar}
A continuous map $U:\mathbb R^2\to S^2$, equal to $e_1$ outside a compact set and satisfying $U(-\xi)=\mathcal RU(\xi)$, has $U(0)=-e_1$ whenever its degree is odd. A continuous $H^1$ map $U:B_\rho^2\to S^2$ with this equivariance and $U(0)=-e_1$ satisfies
\begin{equation}
\frac12\int_{B_\rho^2}|\nabla U|^2
\ge4\pi-\rho^{-2}\int_{B_\rho^2}|U-e_1|^2.
\label{eq:shell-proof-4-1}
\end{equation}
\end{lemma}

\begin{proof}
Let \(U:\mathbb R^2\to S^2\) be continuous, equal to \(e_1\) outside a compact set, and \(U(-\xi)=\mathcal RU(\xi)\). Its value at zero is \(e_1\) or \(-e_1\). If it were \(e_1\), a normalised straight-line homotopy with a radial cutoff in a small disk where \(U\cdot e_1>0\) would make it constant near zero, preserving equivariance and the degree of its extension to the compactified plane. Radial convolution and normalisation then approximate this modified map by smooth equivariant maps constant near zero and infinity, without changing degree: sufficiently close sphere maps are homotopic by normalising their straight-line interpolation.

For any such smooth map, write \(\xi=(\rho\cos\theta,\rho\sin\theta)\), and let \(\mathcal R_\theta\) be rotation by \(\theta\) around \(e_1\). The map \(\widetilde U(\rho,\theta)=\mathcal R_{-\theta}U(\rho,\theta)\) is \(\pi\)-periodic in \(\theta\), and factors through the degree-two map \((\rho,\theta)\mapsto(\rho,2\theta)\) of the compactified plane. Its degree is even. Direct differentiation gives \[
\widetilde U\cdot(\widetilde U_\rho\times\widetilde U_\theta)
=U\cdot(U_\rho\times U_\theta)+\partial_\rho U_1.
\] Since \(U_1=1\) at both radial endpoints, integration gives the same degree for \(U\) and \(\widetilde U\). Therefore an odd-degree equivariant planar map must satisfy \(U(0)=-e_1\).

For a continuous \(H^1\) map \(U:B_\rho^2\to S^2\) with this equivariance and \(U(0)=-e_1\), the angular components \(U_2,U_3\) are antiperiodic under \(\theta\mapsto\theta+\pi\). The one-dimensional circle Poincaré inequality, applied to these zero-mean components, gives \[
\int_0^{2\pi}|U_\theta|^2\,d\theta
\ge\int_0^{2\pi}(1-U_1^2)\,d\theta
\quad\hbox{for almost every radius}.
\] Tangency gives \(U_t\cdot(e_1-U_1U)=\partial_tU_1\), and hence \[
t|U_t|^2+\frac{1-U_1^2}{t}\ge2\partial_tU_1.
\] Integrate on annuli and let the inner radius tend to zero; continuity supplies the inner value \(-e_1\). For almost every \(h<\rho\) this gives \[
\frac12\int_{B_h^2}|\nabla U|^2
\ge4\pi-\frac12\int_0^{2\pi}|U(h,\theta)-e_1|^2\,d\theta.
\] Bound the left side by the energy on \(B_\rho^2\), multiply by \(h\), and integrate from zero to \(\rho\). The result is \eqref{eq:shell-proof-4-1}. Sobolev slicing on annuli justifies all derivatives used here. The circle Poincaré inequality itself follows, for example, by expanding a smooth zero-mean circle function in its Fourier series and comparing \(\sum_{n\ne0}|a_n|^2\) with \(\sum_{n\ne0}n^2|a_n|^2\); density gives its \(H^1\) form.
\end{proof}

The prescribed source determines the degrees of the transverse slices. Applying the planar lemma to these slices gives a lower bound for the energy in the central cylinder.

\begin{lemma}[The central cylindrical cost]\label{lem:shell-core}
For $v\in\mathcal A_{1,R,p}^{(d)}$ and $0<\rho<1$, set $\mathcal C_\rho=\{R-1<r<R+1,\ x_2^2+x_3^2<\rho^2\}$. Then
\begin{equation}
T_d(v;\mathcal C_\rho)\ge
4\pi\sigma_k\int_{R-1}^{R+1}r^k\,dr
-\rho^{-2}\int_{\mathcal C_\rho}|v-e_1|^2.
\label{eq:shell-proof-4-4}
\end{equation}
\end{lemma}

\begin{proof}
Now take \(v\in\mathcal A_{1,R,p}^{(d)}\).
Pull back $F(v)$ to a polar coordinate chart, or, when $k=0$,
to either half-space $x_1=\pm r$. Consider its three normal components
and the associated normal vector \[
D^\theta=(F_{23},F_{3r},F_{r2}).
\] The coefficient of \(dr\wedge \dd x_2\wedge \dd x_3\) in the exterior derivative differentiates only in \(r,x_2,x_3\), not in the angular directions. Thus for almost every \(\theta\), on the half-space \(r>0\), \begin{equation}
\operatorname{div}_{r,x_2,x_3}D^\theta
=4\pi[\delta_{(R,1,0)}+\delta_{(R,-1,0)}
-\delta_{(R+1,0,0)}-\delta_{(R-1,0,0)}].
\label{eq:shell-proof-4-2}
\end{equation} For $k=0$, this is directly the coefficient identity on each of the
two half-spaces, with the signs already checked in
\eqref{eq:shell-zero-sign}. For $k\ge1$, test the corresponding
coefficient identity in an angular chart with products of a normal
test and an angular test. Fubini gives it for almost every angular point for a countable dense family of normal tests. Local \(L^1\) integrability of the normal current extends the identity to all normal tests. A finite chart cover gives a common full-measure subset of \(S^k\). On compact subintervals of \(r>0\) the weight \(r^k\) is bounded above and below, so the required local integrability follows from \(H^1\) slicing. No angular symmetry of \(v\) was used.

For these \(\theta\), integrate \eqref{eq:shell-proof-4-2} in \(x_2,x_3\), using transverse cutoffs equal to one on this map's support. The function \[
F_\theta(r)=\int_{\mathbb R^2}F_{23}(r,\theta,x_2,x_3)\dd x_2\dd x_3
\] has derivative \(4\pi(2\delta_R-\delta_{R-1}-\delta_{R+1})\) on \(r>0\), and is zero for large \(r\). Its uniquely determined values are \begin{equation}
F_\theta=-4\pi\ \hbox{on }(R-1,R),\qquad
F_\theta=4\pi\ \hbox{on }(R,R+1),\qquad
F_\theta=0\ \hbox{elsewhere in }(0,\infty).
\label{eq:shell-proof-4-3}
\end{equation} The terminal condition determines the integration constant on the whole connected half-line; no condition at \(r=0\) is being imposed.

For almost every \(X\), the slice of \(v\) in the \((x_2,x_3)\)-variables lies in \(W^{1,p}(\mathbb R^2)\) relative to its compact exterior constant, and has a continuous sphere-valued representative. The representative is equivariant everywhere since it is so almost everywhere. Its degree is \((4\pi)^{-1}\int F_{23}\): the formula for smooth maps extends by the uniform and strong \(W^{1,p}\) smoothing used in \cref{appsrc-shell-proof-longitudinal-loops-the-global-action-bound-and-its-localization}, with degrees preserved by uniform closeness. Equations \eqref{eq:shell-proof-4-3} therefore give odd degree whenever \(R-1<|X|<R+1\), for almost every such \(X\). The planar assertion proves \(v(x_1,0,0,x_4,\ldots,x_d)=-e_1\) for these representatives.

For the cylinder $\mathcal C_\rho$ in the statement, apply \eqref{eq:shell-proof-4-1} and integrate in \(X\), using \eqref{eq:shell-proof-1-1}. Since the full Dirichlet energy dominates its \(x_2,x_3\) part, we obtain \eqref{eq:shell-proof-4-4}.
\end{proof}

The cylindrical estimate accounts for only part of the defect geometry. A second test gives an energy cost in a disjoint region; adding the two yields the required lower bound for collapse.

\begin{lemma}[The combined local cost]\label{lem:shell-cost}
For $v\in\mathcal A_{1,R,p}^{(d)}$, $0<\rho<1$ and $\mathcal K=\{R-b<r<R+b,\ |x_2|<b,\ |x_3|<b\}$, with $\mathcal C_\rho$ as in \cref{lem:shell-core},
\begin{equation}
\left(1+\frac{2\sqrt3}{b}\right)T_d(v;\mathcal K)
\ge 4\pi\sigma_k\int_{R-1}^{R+1}r^k\,dr
+8\pi\sigma_k(1-\rho)(R-b)^k
-\rho^{-2}\int_{\mathcal C_\rho}|v-e_1|^2.
\label{eq:shell-proof-4-10}
\end{equation}
Consequently, for fixed $\lambda>0$, every sequence $v_n\in\mathcal A_{\lambda,R,p}^{(d)}$ converging to $e_1$ strongly in $L^2_{\rm loc}$ satisfies
\begin{equation}
\liminf_n T_d(v_n;\lambda\mathcal K)
\ge\lambda^{d-2}C_{d,R,b}.
\label{eq:shell-proof-4-12}
\end{equation}
\end{lemma}

\begin{proof}
Define Lipschitz cutoffs \[
\chi(r)=
\begin{cases}
1,&|r-R|\le b/2,\\
2-2|r-R|/b,&b/2<|r-R|<b,\\
0,&|r-R|\ge b,
\end{cases}
\quad
\zeta(x_3)=(1-|x_3|/b)_+,
\] and \[
\eta(x_2)=
\begin{cases}
1,&|x_2|\le1,\\
(b-|x_2|)/(b-1),&1<|x_2|<b,\\
0,&|x_2|\ge b,
\end{cases}
\qquad
f_\rho(x_2)=\min\{(|x_2|-\rho)_+,1-\rho\}\eta(x_2).
\] Put \(\phi_\rho(r,x_2,x_3)=f_\rho(x_2)\chi(r)\zeta(x_3)\), and set \begin{equation}
\mathcal S_\rho=\mathcal K\cap\{\rho<|x_2|<1\}.
\label{eq:shell-proof-4-5}
\end{equation} The set \(\mathcal K\) is bounded and has closure disjoint from \(r=0\). At the positive charges \(\phi_\rho=1-\rho\), and at the negative charges it is zero. Decompose its normal gradient as \[
\nabla_{r,x_2,x_3}\phi_\rho
=(0,\operatorname{sgn}(x_2)\mathbf1_{\{\rho<|x_2|<1\}}\chi\zeta,0)+E_\rho.
\] Every component of \(E_\rho\) has absolute value at most \(2/b\), so \begin{equation}
|E_\rho|\le2\sqrt3/b,\qquad \operatorname{supp} E_\rho\subset\overline{\mathcal K}
\quad\hbox{up to null sets}.
\label{eq:shell-proof-4-6}
\end{equation} These bounds use \(b>2\), \(0\le f_\rho\le1\), \(|\chi'|\le2/b\), \(|\zeta'|\le1/b\), and \(|\eta'|\le2/b\). The main gradient term is supported on \(\mathcal S_\rho\), which is disjoint from \(\mathcal C_\rho\).

At almost every point the three derivatives \(\partial_r v,\partial_2 v,\partial_3 v\) lie in the two-dimensional tangent space of \(S^2\). Expanding them in an orthonormal tangent basis gives coefficient vectors \(\alpha,\beta\in\mathbb R^3\), with \(D^\theta=\pm\alpha\times\beta\). Therefore \begin{equation}
|D^\theta|\le\frac12(|\partial_r v|^2+|\partial_2 v|^2+|\partial_3 v|^2)=:\tau_{\rm nor},
\qquad |F_{3r}|\le\tau_{\rm nor},
\qquad
\tau_{\rm nor}\le\frac12|\nabla v|^2.
\label{eq:shell-proof-4-7}
\end{equation} Test \eqref{eq:shell-proof-4-2} with \(\phi_\rho\). The Lipschitz test is legitimate by convolution on its compact normal-coordinate support: the mollified tests converge uniformly at the four charges and their uniformly bounded gradients converge almost everywhere against the integrable current. It follows from \eqref{eq:shell-proof-4-6}--\eqref{eq:shell-proof-4-7} that, for almost every \(\theta\), \begin{equation}
8\pi(1-\rho)
\le\int_{\mathcal S_\rho^\theta}\tau_{\rm nor}\,dr\dd x_2\dd x_3
+\frac{2\sqrt3}{b}\int_{\mathcal K^\theta}\tau_{\rm nor}\,dr\dd x_2\dd x_3.
\label{eq:shell-proof-4-8}
\end{equation} Here the superscript denotes the normal slice at that \(\theta\). Multiply by \((R-b)^k\) and integrate over \(S^k\). On \(\mathcal K\), \(r^k\ge(R-b)^k\), so \eqref{eq:shell-proof-1-1}, \eqref{eq:shell-proof-4-7}, and \eqref{eq:shell-proof-4-8} give \begin{equation}
8\pi\sigma_k(1-\rho)(R-b)^k
\le T_d(v;\mathcal S_\rho)+\frac{2\sqrt3}{b}T_d(v;\mathcal K).
\label{eq:shell-proof-4-9}
\end{equation} Add \eqref{eq:shell-proof-4-4}. Since \(\mathcal C_\rho,\mathcal S_\rho\) are disjoint subsets of \(\mathcal K\), we conclude \eqref{eq:shell-proof-4-10}. All cutoffs in this estimate are fixed by the geometry. The curvature and cutoff contributions are accounted for by the factor \(1+2\sqrt3/b\).

If \(v_n\to e_1\) strongly in \(L^2_{\rm loc}\), take the lower limit in \eqref{eq:shell-proof-4-10} at fixed \(\rho\), then let \(\rho\downarrow0\). This gives \begin{equation}
\liminf_n T_d(v_n;\mathcal K)\ge C_{d,R,b}.
\label{eq:shell-proof-4-11}
\end{equation} For a sequence in the class with scale \(\lambda\), apply \eqref{eq:shell-proof-4-10} to \(v_n(\lambda\,\cdot)\). Dirichlet energy scales by \(\lambda^{d-2}\), and the local \(L^2\) convergence is preserved at fixed \(\lambda\). Hence \eqref{eq:shell-proof-4-12}.
\end{proof}

\subsection{The strict gap and the minimising sequence}\label{appsrc-shell-proof-collapse-action-the-strict-gap-and-minimizing-sequences}

Combining the local geometric bound with the localised momentum estimate gives a lower bound for the action of any sequence collapsing to $e_1$. We compare this with the explicit trial map before considering minimising sequences.

\begin{proposition}[Action cost of constant collapse]\label{prop:shell-collapse}
For fixed \(\lambda>0\), \(0<c<1\), \(2<p<3\), and \(H>0\), every sequence \(v_n\in\mathcal A_{\lambda,R,p}^{(d)}\) with \(E(v_n)\le H\) and \(v_n\to e_1\) strongly in \(L^2_{\mathrm{loc}}(\mathbb R^d)\) obeys \begin{equation}
\liminf_n J_{c,d}(v_n)\ge
  \lambda^{d-2}C_{d,R,b}
  -c\delta_c^{-1/N}H^{1+1/N}
  -c\pi^{-1/(N-1)}H^q.
\label{eq:shell-proof-L}
\end{equation}
\end{proposition}

\begin{proof}
Apply \eqref{eq:shell-proof-3-8} with \(B=\lambda\mathcal K\) to a sequence as in the statement. The local height norm tends to zero, since the maps converge locally in \(L^2\) to \(e_1\). Combine with \eqref{eq:shell-proof-4-12} to obtain \eqref{eq:shell-proof-L}. This step takes \(n\to\infty\) at fixed \(\lambda\). It uses neither momentum continuity nor a common compact support.
\end{proof}

At small spatial scales, the momentum errors and anisotropy energy are of higher order than the gap between the two Dirichlet-energy constants. Choosing $\lambda$ sufficiently small therefore separates the comparison action from the collapse cost.

\begin{corollary}[A strict comparison gap]\label{cor:shell-gap}
Set \[
  K=2\gamma_c^{-1}(B_{d,R}+V_R)>0.
  \] Every \(0<\lambda\le1\) satisfying \begin{equation}
K\lambda^{d-2}<M_{d,c}/2
\label{eq:shell-proof-small1}
\end{equation} and \begin{equation}
\lambda^2V_R+
  c\delta_c^{-1/N}K^{1+1/N}\lambda^{(d-2)/N}
  +c\pi^{-1/(N-1)}K^q\lambda
  \le\Delta/2
\label{eq:shell-proof-small2}
\end{equation} has \begin{equation}
\liminf_n J_{c,d}(v_n)
  \ge J_{c,d}(W_{\lambda,R})+\tfrac12\Delta\lambda^{d-2}
\label{eq:shell-proof-gap}
\end{equation} for every sequence satisfying the hypotheses of \cref{prop:shell-collapse} with \(H=K\lambda^{d-2}\). There are arbitrarily small such \(\lambda\), for each fixed \(d,c,b,R\).
\end{corollary}

\begin{proof}
For \(H=K\lambda^{d-2}\), the errors in \eqref{eq:shell-proof-L}, divided by \(\lambda^{d-2}\), are \[
c\delta_c^{-1/N}K^{1+1/N}\lambda^{(d-2)/N}
\quad\hbox{and}\quad
c\pi^{-1/(N-1)}K^q\lambda,
\] because \((d-2)(q-1)=1\). Subtract the trial estimate \eqref{eq:shell-proof-T}. Condition \eqref{eq:shell-proof-small2} gives \eqref{eq:shell-proof-gap}. Both error powers and the height power \(\lambda^2\) are strictly
positive. In dimension three they are respectively $1/2$, $1$, and $2$.
This proves the existence of arbitrarily small allowable scales
satisfying \eqref{eq:shell-proof-small1}--\eqref{eq:shell-proof-small2}.
\end{proof}

The strict gap now prevents a minimising sequence from having a constant limit. Small-energy coercivity also keeps the sequence strictly inside the energy cap.

\begin{proposition}[A nonconstant limit of minimising sequences]\label{prop:shell-proof}
For any $0<\lambda\le1$ satisfying \eqref{eq:shell-proof-small1}--\eqref{eq:shell-proof-small2}, let \[
  I_{\lambda,R,c,p}^{(d)}
  =\inf\{J_{c,d}(v):v\in\mathcal A_{\lambda,R,p}^{(d)},\
  E(v)\le M_{d,c}\}.
  \] This infimum is finite and nonnegative. Every minimising sequence has a subsequence converging weakly in \(H^1_{\mathrm{loc}}\), strongly in \(L^2_{\mathrm{loc}}\), and almost everywhere to a nonconstant \(u\in H^1_{\mathrm{loc}}(\mathbb R^d;S^2)\), with \begin{equation}
E(u)\le\gamma_c^{-1}
  \big(\lambda^{d-2}T_R+\lambda^dV_R\big),\qquad
  u-e_1\in L^{2d/(d-2)}(\mathbb R^d),\qquad
  u(gx)=\mathcal Ru(x).
\label{eq:shell-proof-NC}
\end{equation} Every sufficiently late term of the minimising sequence has energy \(<M_{d,c}/2\).
\end{proposition}

\begin{proof}
The trial is admissible for the energy cap by \eqref{eq:shell-proof-small1} and \(\lambda\le1\): \[
E(W_{\lambda,R})\le
(B_{d,R}+V_R)\lambda^{d-2}
=\frac{\gamma_cK}{2}\lambda^{d-2}<M_{d,c}.
\] Here \(0<\gamma_c<1\). Coercivity gives \begin{equation}
0\le I_{\lambda,R,c,p}^{(d)}
\le J_{c,d}(W_{\lambda,R})<\infty.
\label{eq:shell-proof-5-1}
\end{equation} For a minimising sequence, write \(J_n=J_{c,d}(v_n)\). For all sufficiently large \(n\), \(J_n\le 2J_{c,d}(W_{\lambda,R})\), since this trial action is positive. Coercivity then gives \begin{equation}
E(v_n)\le2\gamma_c^{-1}J_{c,d}(W_{\lambda,R})
\le K\lambda^{d-2}<M_{d,c}/2.
\label{eq:shell-proof-5-2}
\end{equation} Thus the eventual cap margin is uniform. Local \(H^1\) bounds, \(|v_n|=1\), Rellich compactness on bounded balls, and a diagonal subsequence give a sphere-valued limit \(u\) with the stated weak, strong, and almost-everywhere local convergence. Lower semicontinuity on balls followed by increasing balls gives \begin{equation}
E(u)\le\liminf_nE(v_n)
\le\gamma_c^{-1}I_{\lambda,R,c,p}^{(d)}
\le\gamma_c^{-1}J_{c,d}(W_{\lambda,R}).
\label{eq:shell-proof-5-3}
\end{equation} For the middle bound, \(E(v_n)\le\gamma_c^{-1}J_n\) holds termwise and \(J_n\to I\).

Each \(v_n-e_1\) is compactly supported and in \(H^1(\mathbb R^d)\). The homogeneous Sobolev inequality \begin{equation}
\|f\|_{L^{2d/(d-2)}}\le S_d\|\nabla f\|_{L^2}
\label{eq:shell-proof-5-4}
\end{equation} for such functions holds with a finite constant depending only on the dimension. It can also be obtained from \eqref{eq:shell-proof-2-9} with \(N\) replaced by \(d\): apply that \(W^{1,1}\) inequality to \(|f|^\alpha\), with \(\alpha=2(d-1)/(d-2)\), and use Cauchy--Schwarz on \(\alpha\int |f|^{\alpha-1}|\nabla f|\), since \(2(\alpha-1)=2d/(d-2)\). Division by the resulting norm power gives \eqref{eq:shell-proof-5-4} for smooth \(f\), and approximation gives it for compactly supported \(H^1\) functions; the vector-valued case follows by applying the scalar argument to \(|f|\). Therefore \eqref{eq:shell-proof-5-2} bounds \(v_n-e_1\) in \(L^{2d/(d-2)}\), uniformly. Fatou under the almost-everywhere convergence gives the same membership for \(u-e_1\). Equivariance passes through the local convergence.

If \(u\) were constant, this global integrability would force \(u=e_1\): any other constant difference has nonzero magnitude on a set of infinite measure. The subsequence would then be a constant-collapse sequence with energy bound \(K\lambda^{d-2}\). By \eqref{eq:shell-proof-gap}, \[
I_{\lambda,R,c,p}^{(d)}
=\lim_nJ_n
\ge J_{c,d}(W_{\lambda,R})+\tfrac12\Delta\lambda^{d-2},
\] contradicting \eqref{eq:shell-proof-5-1}. Thus \(u\) is nonconstant. This proves the comparison and noncollapse estimates.
\end{proof}

\subsection{Conclusion of the higher-dimensional proof}\label{sec:higher-conclusion}

\begin{proof}[Proof of \cref{thm:all-dimensional} for $d\ge3$]
Fix $d\ge3$ and $c\in(0,1)$, and take $p=5/2$.
The comparison profile in \cref{app:charged-comparison} gives the admissible maps of
\cref{prop:shell-comparison}, with $k=d-3$ and $S^0$ interpreted as two points when $d=3$. Choose its geometry $b,R$ and a
sufficiently small scale $\lambda>0$.
By \cref{prop:shell-proof}, the capped infimum in the prescribed-current class
$\mathcal A_{\lambda,R,p}^{(d)}$ is finite, and a minimising sequence
$v_n$ has a nonconstant weak $H^1_{\rm loc}$/strong $L^2_{\rm loc}$
limit $u$. Its sufficiently late terms satisfy
$E(v_n)<M_{d,c}/2$.

Apply \cref{prop:rotation-proof} with
$G=\{\mathrm{Id},g\}$, $g(x)=(x_1,-x_2,-x_3,x_4,\ldots,x_d)$,
$\rho(g)=\mathcal R=\operatorname{diag}(1,-1,-1)$,
$M=M_{d,c}$, and $\delta=M_{d,c}/2$.
Here $g$ fixes the travel direction, $\mathcal Re_1=e_1$, and
$\mathcal Re_3=-e_3$. The momentum is
$P_d(v)=\int x_2F_{12}(v)$.
The final assertion of the rotation proposition verifies all its
variation hypotheses for this exact class and gives \eqref{eq:weak}
against arbitrary compact smooth vector fields.
The energy bound and nonconstancy follow from
\cref{prop:shell-proof}. In fact
$E(u)\le C_{d,c}\lambda^{d-2}$, and the allowed scales are arbitrarily
small for each fixed $d,c$.
\end{proof}

\section*{Acknowledgements}

The research of Z.~Li is supported by the National Natural Science
Foundation of China under grants 12422110 and 12371244. The research of
J.~Wei is partially supported by GRF of RGC of Hong Kong entitled
``On critical and supercritical Fujita equation''. R.~Juneman wishes
to thank the Institute of Mathematical Sciences at CUHK for the
excellent working environment.

The authors acknowledge the use of AI tools in the preparation of
this manuscript, including assistance with its organisation,
exposition, and proof ideas. The authors have reviewed the complete
manuscript and take full responsibility for all mathematical statements
and proofs.

\end{document}